\documentclass[12pt,a4paper]{amsart}

\usepackage[arrow,tips,matrix]{xy}
\SelectTips{cm}{10}

\usepackage{pgfplots,tkz-euclide}

\usepackage{a4wide}
\usepackage{graphicx}
\usepackage[colorlinks,linkcolor=blue,citecolor=purple]{hyperref}
\advance\textheight by 3\baselineskip

\title[The versal deformation of conic bundle manifolds]{The versal deformation of small resolutions of conic bundles over $\P^1\times \P^1$ with two sections blown down}

\author{Bernd Kreu\ss ler}
\address{Mary Immaculate College,
South Circular Road,
Limerick,
V94 VN26,
Ireland}
\email{bernd.kreussler@mic.ul.ie} 
\author{Jan Stevens}
\address{Matematiska vetenskaper, G\"oteborgs universitet 
och Chalmers tekniska h\"ogskola, 41296
G\"oteborg, Sweden}
\email{stevens@chalmers.se}
\dedicatory{Dedicated to Herbert Kurke on the occasion of his 85$^{\text{th}}$
  birthday.} 
\def\KLB{SRCB }

\def\wt#1{\widetilde{#1}}
\def\wh#1{\widehat{#1}}
\def\wl#1{\overline{#1}}
\def\cal{\mathcal}
\def\Bbb{\mathbb}
\def\sier#1{{\cal O}_{#1}}
\def\C{{\Bbb C}} 
\def\P{{\Bbb P}} 
 \def\F{{\Bbb F}}
\def\cw{{\cal W}} \def\ce{{\cal E}}
 \def\cv{{\cal V}}
\def\cy{{\cal Y}} \def\cR{{\cal R}}
\def\ct{{\cal T}} \def\cz{{\cal Z}}
 \def\ci{{\cal I}}
\def\cn{{N}}
 \def\cs{{\cal S}}
\def\co{{\cal O}} \def\cx{{\cal X}}
\def\ep{\varepsilon}
\def\vp{\varphi} \def\la{\lambda}
\def\al#1{\alpha_{#1}}
\def\Rank{\mathop{\rm Rank}}

\def\End{\mathop{\rm End}}

\def\Def{\mathop{\rm Def}}
\def\Pic{\mathop{\rm Pic}}
\def\Im{\mathop{\rm Im}}
\def\Sym{\mathop{\rm Sym}}
\def\pr#1{\mathop{\rm pr}_{#1}}
\def\lra{\longrightarrow}
\def\mapright#1{\mathrel{%
\smash{\mathop{\longrightarrow}\limits^{#1}}}}

\def\mapdown#1{\Big\downarrow
 \rlap{$\vcenter{\hbox{$\scriptstyle #1$}}$}}

\def\semic{\mathrel{;}}
\newcommand{\D}{\mathrm{d}}

\newtheorem{theorem}{Theorem}[section]
\newtheorem*{theorem*}{Theorem}
\newtheorem{proposition}[theorem]{Proposition}
\newtheorem{lemma}[theorem]{Lemma}
\theoremstyle{definition}
\newtheorem*{defn}{Definition}
\theoremstyle{remark}
\newtheorem{remark}[theorem]{Remark} 
\newtheorem{example}[theorem]{Example}
\newtheorem*{notation}{Notational convention}

\def\FuncX{(4*t^4+7*t^3-2*t)/(4+7*t-2*t^3)}
\def\FuncY{(2*t^4-7*t^2-4*t)/(2*(t+1)^2*(t+2)^2*(2*t+1)^2)}
\pgfplotsset{every axis plot/.append style={line width=0.8pt,samples=50}}

\begin{document}
\begin{abstract}
  Twistor spaces are certain compact complex threefolds with an additional
  real fibre bundle structure.
  We focus here on twistor spaces over $\P^2\#\P^2\#\P^2$.
  Such spaces are either small resolutions of
  double solids or they can be described as modifications of conic
  bundles. The last type is the more special one: they deform into
  double solids. We give an explicit description of this deformation,
  in a more general context.
\end{abstract}



\maketitle

\section*{Introduction}
This paper describes the deformation of a special non-projective complex
manifold to another one, which is also Moishezon, but is given by different equations.
Such manifolds firstly occurred as twistor spaces. 

The original Penrose twistor correspondence interprets four-dimensional Minkowski space as
a space of lines in complex projective three-space.
The twistor construction in the context of Riemannian geometry was first
worked out in detail by Atiyah, Hitchin and Singer \cite{ahs}.
They showed that the conformal structure of a self-dual four-dimensional Riemannian manifold
determines a three-dimensional complex manifold, the twistor space, from which
the conformal structure can be reconstructed with the aid of an
anti-holomorphic involution (real structure) and a covering family of ``lines''.
For an introduction from an algebro-geometric point of view, see LeBrun's
survey \cite{lb3}.

Most twistor spaces are not projective: only $\P^{3}$ and the flag variety $\F(1,2)$ are K\"ahler
twistor spaces, as  shown by Hitchin \cite{h1}, and independently by Friedrich and Kurke \cite{fk}.
Indeed, the K\"ahler condition implies that the self-dual metric on the Riemannian four-manifold has 
positive scalar curvature; moreover, it can only be $S^4$, $\P^2$, $\P^2\#\P^2$ or $\P^2\#\P^2\#\P^2$. 
The twistor space can only be projective 3-space, the flag variety,
the intersection of two quadrics or a double solid with quartic branch surface, and the last two are excluded,
because of their Euler characteristic. Poon \cite{po86} showed that a twistor space with
real fibre bundle structure over $\P^2\#\P^2$ is a non-projective small resolution of a singular
intersection of two quadrics; there is a $1$-parameter
family of such twistor spaces.

The existence of twistor spaces over  $\P^2\#\cdots\#\P^2$, the connected sum of $n\ge 3$
copies of $\P^{2}$ considered as a real $4$-manifold, was for all $n$ shown by Donaldson--Friedman \cite{df}
with a deformation argument, using $d$-semistable reducible spaces. 
For $n\ge 4$, twistor spaces do no longer need to be Moishezon. LeBrun \cite{lb1} constructed 
by more elementary methods explicit metrics for all values of $n$. The associated twistor spaces
are called LeBrun twistor spaces. They are Moishezon, and are bimeromorphic
to a small resolution of a certain conic bundle over the quadric
$Q=\P^1 \times \P^1$ \cite{lb1}.

In the case $n=3$, all twistor space are Moishezon, but the LeBrun twistor spaces
are not the most general. Also small resolutions of double solids, branched over a singular
quartic, are twistor spaces, by the work of Poon (see  \cite{po92}) and of Kurke and the first Author
\cite{krs,k-k}. In the course of this work, Kurke realised that not only double solids but also 
modifications of conic bundles are candidates for twistor spaces (see \cite[Theorem 2.3]{krs}).
From a general deformation theory argument \cite{df}, 
it was clear that there must be a deformation of a LeBrun twistor space into a small resolution 
of a double solid.
The purpose of this paper is to explain in detail how this can be done.
This answers a question asked by Herbert Kurke at the Berlin--Hamburg--Oslo
Seminar on Singularities which was hosted by Arnfinn Laudal on 25--27 October
1991 in Oslo.

During the past two decades, Honda has shed much light on the structure of
twistor spaces over $\P^2\#\cdots\#\P^2$, see, for example,
\cite{ho00}, \cite{ho}, \cite{ho2}, \cite{ho3}, \cite{ho09}, \cite{ho10},
\cite{ho14}, \cite{ho15-1}, \cite{ho15-2}. 
For example, for $n>4$, the general twistor space has algebraic dimension $0$.
Many examples of twistor spaces of algebraic dimension $1$ or $3$ have been
described explicitly for all $n\ge4$.
Examples of algebraic dimension $2$ are known \cite{ck99}, \cite{ho00},
\cite{ho06} for $n=4$, but remain mysterious for $n>4$, see \cite{hok}.

Recently, in \cite{ho23}, Honda proved the astonishing result that a
Moishezon twistor space over $\P^2\#\cdots\#\P^2$, whose fundamental linear
system is a pencil, either is bimeromorphic to a conic bundle over a possibly
singular surface or to a branched double cover over a three-dimensional
variety. This manifests in full generality the dichotomy we have seen in case
$n=3$.
Our work is a first step towards understanding how Moishezon
twistor spaces of one type deform into the other.


Before stating our results in  more detail, we recall the description
of LeBrun twistor spaces over  $\P^2\#\cdots\#\P^2$. 
Such a space  is bimeromorphic
to a small resolution of a certain conic bundle over the quadric
$Q=\P^1 \times \P^1$ \cite{lb1}, see also \cite{ku}. Let $\vp_1$,
\dots, $\vp_n$ be bihomogeneous forms, defining smooth curves of
type $(1,1)$ on $Q$. Consider the hypersurface $W$ in the $\P^2$-bundle
$\P(\sier Q\oplus \sier Q (1-n,-1)\oplus \sier Q(-1,1-n))$ over $Q$ with
bihomogeneous equation 
\begin{equation}\label{conic}
 w_1w_2-\vp_1\cdots\vp_n w_0^2=0 \;.
\end{equation}
The singular points of $W$ lie over the intersection points of the
curves $\vp_i$ in  $w_1=w_2=0$. Let $\wt W$ be a suitably
chosen small resolution of
$W$ (the resolution has to be compatible
with the real structure on $W$; only one choice  leads 
to a twistor space, see \cite{lb1} and \cite{ku} for details). The sections $E_1\colon
w_0=w_1=0$ and $E_2\colon w_0=w_2=0$ in $W$ do not pass through the
singular locus, and therefore, they can also be considered as divisors
in $\wt W$. Each can be blown down  to a curve, along opposite
rulings of the quadric, both in $W$ and $\wt W$. As a result, we
get a singular threefold $Z$ and its small resolution $\wt Z$, the
LeBrun twistor space.

As the real structure and the twistor lines play no role in our main
construction, we can consider more general discriminant curves, only
subject to the condition that the singularities of $W$ admit a small
resolution. In particular, we can replace the product
$\vp_1\cdots\vp_n$ by any non-singular form of type $(n,n)$. We call a
smooth space obtained by blowing down the two sections $E_1$, $E_2$
an \emph{\KLB manifold}\footnote{\KLB stands for \textbf{s}mall
  \textbf{r}esolution of a \textbf{c}onic \textbf{b}undle}. 
In this Introduction, we mainly restrict ourselves to the case of the product
$\vp_1\cdots\vp_n$, with $n=3$.

The general twistor space over $\P^2\#\P^2\#\P^2$ is a small
resolution of a double cover of $\P^3$, branched along a quartic
surface with 13 ordinary double points \cite{po92,k-k}, with equation
of the form
\begin{equation}\label{double}
Q^2-L_1L_2L_3L_4=0\;.
\end{equation}
A similar phenomenon, of two families, occurs for $K3$-surfaces with
a polarisation of degree $2$ \cite{sh}. The general one is a double
cover of $\P^2$, branched along a sextic curve, but it can also
happen that the linear system only maps the surface to a conic.
Twice the linear system exhibits such surfaces as double covers of
the projective cone over the rational normal curve of degree $4$.
This cone deforms to the Veronese surface, which is an embedding of
$\P^2$. Similarly, the Veronese embedding of $\P^3$ in $\P^9$ is a
deformation of the projective cone over $\P^1\times\P^1$, embedded
with $\co(2,2)$ in $\P^8\subset\P^9$. 

The previous observation was the
starting point for our investigations. The first problem one runs
into is that the conic bundle involves three planes, whereas the
double solid has four linear forms in its equation, each giving rise
to two surfaces of degree $1$ on the double cover (here degree is
measured by intersecting with twistor lines), and these are the
only surfaces of degree $1$. The conic bundle twistor spaces have
two pencils of surfaces of degree $1$ (the inverse images of the
rulings on the quadric), so most surfaces do not survive under the
deformation. The solution to this problem is to take two surviving
divisors as additional information, and consider the deformation
theory of the manifold $\wt Z$ together with two divisors $\wt S_1$,
$\wt S_2$, which are the strict transforms of the  inverse image of
two intersecting lines on $Q$. To specify the two surfaces, we just
give a tangent plane to the quadric, thereby introducing a fourth
plane; it intersects the quadric in two lines.

Instead of working with double covers, we extend the $\P^2$-bundle
over $\P^3$. This is possible if we add the two lines to the
discriminant curve, which then becomes a curve of degree $(4,4)$.
The bundle is
\[
\P(\co\oplus \co(-2)\oplus \co(-2))\rightarrow \P^{3}\;.
\]

We connect conic bundles of the form \eqref{conic} (with $n=3$) and
double solids of the form \eqref{double} by a family  $\cy$ with
equations
\begin{equation}\label{family}
\begin{split}
y_{1}y_{2}-L_1L_2L_3L_4y_{0}^{2}&=0\;,\\
\al2y_{1}+\al1y_{2}-Qy_{0}&=0\;,
\end{split}
\end{equation}
where $L_1$, $L_2$ and $L_3$ are linear forms on $\P^3$, which
restrict on the quadric to $\vp_1$, $\vp_2$ and $\vp_3$, respectively,
and $L_4$ depends on the product $\al1\al2$ in a way to be
specified; for $\al1\al2=0$, it gives the mentioned tangent plane,
whereas for $\al1\al2\neq0$, there should be a 13th singular point.
The general fibre is indeed a double solid, for if $\al1\al2\neq0$
we can eliminate,  say $y_1$,  and find the double cover of $\P^3$,
branched along the quartic with equation
\[
Q^{2}-4\al1\al2L_1L_2L_3L_4=0\;.
\]
Equations \eqref{family}
are the correct equations to write down, but to obtain a deformation
of an \KLB manifold, we need to change the family by
birational transformations.
The central fibre of the family $\cy$  is reducible, with the sections
$y_0=y_1=0$ and $y_0=y_2=0$ as extra components. Furthermore, the
remaining component is only birational to a conic bundle with
discriminant curve of degree $(3,3)$. Fortunately, we have here a
case of two wrongs making one right.

The rational map  to a conic bundle can be
factored as a small resolution of the singularities introduced by
the tangent plane $L_4$ (seven ordinary double points for a
tangent plane in general position) and the fibre-wise blow down of
two ruled surfaces. We realise this factorisation, and
the subsequent map to an \KLB manifold, by a sequence of
birational transformations of the total space:
\[
 \cy \longleftarrow \cy^-
\dashrightarrow\cy^+ \lra \cz
\longleftarrow \wt\cz\;.
\]
The morphism $\cy^-\to \cy$ is the simultaneous small  resolution of
seven isolated singularities of the fibres of the family. Each ruled
surface survives over one of the divisors $\alpha_i=0$, with normal
bundle in the total space restricting
to each line of the ruling as $\co(-1)\oplus\co(-1)$. Therefore, all
the lines can be simultaneously flopped in the total space. The flop 
$\cy^-\dashrightarrow\cy^+$ induces a
blowing down in the components in which the surfaces lie. On the
other hand, it gives a blow-up of the extra components, which 
become relatively exceptional, and can be contracted to lines,
by the Castelnuovo--Moishezon--Nakano criterion.
This is achieved by  the morphism $\cy^+\to\cz$. Finally, one
resolves the remaining singularities. The manifold $\wt\cz$ is the
total space of a deformation of the original  \KLB manifold.
By varying the coefficients of $L_1, L_2, L_3, L_4$ in an appropriate way, we
obtain a family $\wt\cz\rightarrow\Pi$. Our main result is the
following.

\begin{theorem*}
The fibres of the so-constructed family $\wt\cz\rightarrow\Pi$ over
$\al1=\al2=0$  are \KLB manifolds, and the deformation
$\wt\cz\rightarrow\Pi$ is for small $\al1,\al2$
versal for deformations of triples $(\wt
Z,\wt S_1,\wt S_2)$.
\end{theorem*}

We prove the theorem (as stated) for  
general \KLB manifolds, that is,
for general discriminant curves. Also, there are almost no
restrictions on the position of the tangent plane. To prove
versality, we need to know the spaces of infinitesimal deformations
and obstructions. We compute them in the needed generality.
With no extra effort, this can be done for general $n$. In the
twistor case, the computations were performed by LeBrun \cite{lb2},
under the additional assumption that all singularities are of type
$A_1$. We follow the  same strategy.

We spend some time to define the family $\wt\cz\rightarrow\Pi$,
by specifying a family $\cy\rightarrow\Pi$, making 
the above statement, that we vary the occurring  coefficients 
in an appropriate way, more precise. The problem is that the base space
$\Pi$ can in general only be given implicitly. 
In the twistor case, we can give a rather explicit
description. In two other cases, we have been able to
give precise formulas, namely if the general fibre is 
a double solid, branched along a 14-nodal quartic, or along a
Kummer surface.

It is known that a small real deformation 
of a LeBrun twistor space has again the structure of a twistor space.
Our formulas simplify if the
LeBrun twistor space has a torus action. Such twistor spaces
and their deformations  have been
studied by Honda  \cite{ho,ho2,ho3}.
Therefore, our construction gives a new proof of a result of
Honda \cite[Theorem 2.1]{ho4}, which states the existence of degenerate double
solids (double solids with branch surface with higher singularities)
as twistor spaces. Finally, it is interesting to note that in our
construction, we can interchange the rulings of the quadric. The effect 
on the double solid is a flop of all exceptional curves. Therefore, there
are two small resolutions admitting a twistor structure.

The content of this paper is as follows. In the first section, we
define and describe \KLB manifolds, for general $n$. Their 
infinitesimal deformations are studied in the next section.
The third section contains the main result, the construction of 
a versal family. Some fibres, which cannot have a real structure,
lie halfway between \KLB manifolds and double solids, and they are
studied in more detail in a separate section. Finally, we 
specialise to the twistor case. We also
give some other examples of our construction.


\begin{notation} Often, we will have two objects,
for example, $E_1$, $E_2$,  under study, which behave in a similar way. To
have efficient notation, we treat both cases at the same time by
dropping all indices, that is, we write $E, \Delta, C$, and so on, instead of
$E_i, \Delta_{i}, C_{i}$. But it is understood that all the new
objects we introduce will pick up an index $i$ if we return to deal
with the global situation.
\end{notation}

\section{\KLB manifolds}\label{sec:KL}
In this section, we define and describe the manifolds we consider.
Although our deformation of \KLB manifolds works only for
$n=3$, we give the definition for arbitrary $n$.
We remark that, on the one hand, Honda has constructed Moishezon twistor
spaces, for arbitrary $n$, which are generalisations of the double solids for
$n=3$ \cite{ho3}, \cite{ho15-1}, \cite{ho23}.
On the other hand, Honda has also constructed Moishezon twistor spaces, for
arbitrary $n$, which are generalisations of LeBrun twistor spaces \cite{ho10}.
It would be interesting to explore in which way our construction
of a versal deformation could be extended to $n\ge4$ and whether this could
give a deformation of (generalised) LeBrun twistor spaces into Honda's
generalised double solid twistor spaces.

We start with the $\P^2$-bundle $\P(\ce_Q):=\P(\sier Q\oplus \sier Q
(1-n,-1)\oplus \sier Q(-1,1-n))$ over the quadric $Q=\P^1 \times
\P^1$, in which we consider the hyper-surface $W$ with bihomogeneous
equation
\begin{equation}\label{klb}
w_1w_2-\vp w_0^2=0 \;.
\end{equation}
Here,  $(w_{0}:w_{1}:w_{2})$ are fibre coordinates corresponding to
the direct summands in this order and $\vp\in H^0(Q,\sier Q(n,n))$.
\begin{remark}
We shall consistently work with homogeneous coordinates.
The bundle $\P(\ce_Q)$ is obtained as the quotient of 
$\left(\C^{2}\setminus 0\right) \times \left(\C^{2}\setminus 0\right)\times
\left(\C^{3}\setminus 0\right)$ under the $\left(\C^{\ast}\right)^{3}$-action
which is given by $(a,b,c)\cdot(s_0,s_1\semic t_0,t_1\semic w_{0},w_{1},w_{2})
= (as_0,as_1\semic bt_0,bt_1\semic cw_{0},ca^{n-1}bw_{1},cab^{n-1}w_{2})$.

Obviously, Equation \eqref{klb} defines a family of conic
bundles which is contained in 
$\P (\ce_Q) \times
H^0(\sier Q(n,n))$. Because we get an isomorphic conic bundle, if we replace
$\varphi$ with a non-zero multiple of it, it is desirable to use the equation
to define a family of conic bundles over 
$\P(H^0(\sier Q(n,n))^*)$.
It turns out that we do not get a family of conic bundles in the product space
$\P (\ce_Q) \times
\P(H^0(\sier Q(n,n))^*)$ but in a non-trivial bundle over  
$\P(H^0(\sier Q(n,n))^*)$. This bundle is most conveniently described as the
quotient of 
$\left(\C^{2}\setminus 0\right) \times \left(\C^{2}\setminus 0\right)\times
\left(\C^{3}\setminus 0\right) \times \left(H^0(\sier Q(n,n))\setminus
  0\right)$ under the $\left(\C^{\ast}\right)^{4}$-action which is given by 
\begin{multline*}
\qquad
(a,b,c,\lambda)\cdot(s_0,s_1\semic t_0,t_1\semic w_{0},w_{1},w_{2}
\semic \varphi)\\
{}= (as_0,as_1\semic bt_0,bt_1\semic cw_{0},\lambda ca^{n-1}bw_{1},cab^{n-1}w_{2}\semic
\lambda\varphi)\;.\qquad
\end{multline*} 
The zero set of Equation \eqref{klb} is invariant
under this action, hence defines a family of conic bundles $\cw\lra\P(H^0(\sier
Q(n,n))^*)$. Note that if we evaluate Equation \eqref{klb} at the point
$(as\semic bt\semic cw_{0},ca^{n-1}bw_{1},cab^{n-1}w_{2}\semic \varphi)$, 
the symbol $\varphi$ is to be seen as $\varphi(as\semic
bt)=a^{n}b^{n}\varphi(s\semic t)$.  
Briefly speaking, the family of $\P^{2}$-bundles constructed this way is
the quotient of $\P (\ce_Q) \times H^0(\sier Q(n,n))$ 
under the $\C^{\ast}$-action
which is given by 
$\lambda\cdot (s\semic t\semic w_{0},w_{1},w_{2}\semic \varphi) = 
(s\semic t\semic w_{0},\lambda w_{1},w_{2}\semic \lambda\varphi)$. 
We could equally well let $\lambda\in\C^{\ast}$ act on the $w_{2}$-component
rather than on $w_{1}$. The two choices are related by the $\C^\ast$-action
on the conic bundle. In fact, there is a $\left(\C^{\ast}\right)^{2}$-action
on $\P (\ce_Q) \times H^0(\sier Q(n,n))$ (or a 
$\left(\C^{\ast}\right)^{5}$-action if we include the $(a,b,c)$), given
by 
$(\lambda,\mu)\cdot (s\semic t\semic w_{0},w_{1},w_{2}\semic \varphi) = 
(s\semic t\semic w_{0},\lambda w_{1},\mu w_{2} \semic \lambda\mu\varphi)$. 
The subgroup $\lambda\mu=1$ survives as  $\C^{\ast}$-action
on the conic bundle.
\end{remark}

The quadric $Q=\P^1 \times \P^1$ has two projections onto $\P^1$,
the projection $\pr1$ on the first factor, and $\pr2$ on the second
factor. The section $E_1\cong Q$ of the $\P^2$-bundle $\P(\ce_Q)$, given by
$w_0=w_1=0$, is contained in $W$, and its  normal bundle  in $W$ is
$\sier Q (-1,1-n)$, so it restricts to $\sier{\P^1}(-1)$ on each
line $\pr2^{-1}(t)$ of the second ruling. Therefore, $E_1$  can be
contracted along the second ruling, according to the
Castelnuovo--Moishezon--Nakano criterion \cite[6.11]{ar},
\cite[Theorem 3.1]{pet}.

\begin{theorem}[Castelnuovo--Moishezon--Nakano criterion]
\label{CMN}
Let $Y$ be smooth of codimension 1 in the manifold $X$ and let $Y\to
Y'$ be a fibration with fibres projective spaces. If the normal
bundle $N_{Y/X}$ restricts to each fibre as $\sier{}(-1)$, then
there exists a contraction $X\to X'$ which blows down $Y$ to
$Y'\subset X'$.
\end{theorem}

In the case at hand, one easily can write down the contraction
explicitly. On $W$, consider the chart  $U_1=W-(w_2)\supset E_1$.
With the homogeneous coordinates $(s_0:s_1\semic t_0:t_1)$ on $Q$
introduced above, 
we define functions
\begin{equation}\label{chartcoord}
v_{i\semic j}:=\frac{w_0}{w_2}s_0^{1-i}s_1^it_0^{n-1-j}t_1^j, \qquad
i=0,1,\quad j=0,\dots,n-1.
\end{equation}
Together with the homogeneous coordinates $(t_0:t_1)$ these
functions map  $U_1$ onto the subset $V_1$ of $\C^{2n}\times \P^1$
defined by the $2\times2$ minors of the matrix
\[
\begin{pmatrix}
t_0^{n-1} & t_0^{n-2}t_1 & \dots & t_1^{n-1} \\
v_{0\semic 0}   & v_{0\semic 1}      & \dots & v_{0\semic n-1} \\
v_{1\semic 0}   & v_{1\semic 1}      & \dots & v_{1\semic n-1} \\
\end{pmatrix}\;.
\]
As $t_0$ and $t_1$ are not both zero, two equations defining
a subspace of $\C^{4}\times \P^1$ suffice:
\begin{equation}\label{chart}
\begin{split}
v_{0\semic 0}t_1^{n-1}-v_{0\semic n-1} t_0^{n-1}&=0\;, \\
v_{1\semic 0}t_1^{n-1}-v_{1\semic n-1} t_0^{n-1}&=0\;. 
\end{split}
\end{equation}
This is the description of the blow-down given by Kurke \cite{ku}.
The section $E_1$ is blown down to the curve $C_1\colon \{0\}\times
\P^1$. From the equations, it is immediate that its normal bundle is
$\co(1-n)\oplus\co(1-n)$.

Likewise, the section $E_2\colon w_0=w_2=0$ can be contracted along  
the first ruling. An explicit  
contraction $U_{2}=W-(w_{1}) \to V_2$ 
is constructed from the above formulas by reversing the role of
the $s_i$ and $t_i$ and of $w_1$ and $w_2$. Altogether we obtain a blowing
down map $\beta\colon W\to Z$.

The space $Z$ may have singularities, but we require that they admit
a small resolution. To study this condition, we first observe that
$W$ and $Z$ have the same singularities, as the singular points of
$W$ lie  in $w_1=w_2=0$ over the singular points of the curve
$D\colon \vp=0$, outside the charts $U_1$ and $U_2$.
We could also resolve first and then blow down. 
This leads to the commutative diagram 
\[
\begin{matrix} \wt W & \mapright{\sigma_W}  & W \\ \mapdown{\wt \beta}
&& \mapdown{\beta} \\[2mm] \wt Z & \mapright{\sigma_Z}  & Z
\makebox[0pt][l]{$\;.$}
\end{matrix}\hphantom{\;.}
\]
The postulated existence of a small resolution limits the possible
types of singularities of $W$ and $Z$. Around each point, one can
find local coordinates such that the singularity has an equation of
the form $w_1w_2-g(x,y)=0$. This is a so-called $cA_k$ point, with
general hyperplane section of type $A_k$, where $k+1$ is the
multiplicity of the plane curve singularity $g$. Such a singularity
admits a small resolution if and only if $g$ factors as a product of
$k+1$ factors, so defines a curve singularity $\Gamma$ with $k+1$
smooth branches \cite[p.~676]{fr}.

\begin{defn} An {\em \KLB manifold\/} $\wt Z$ is a small resolution of
the threefold $Z$ obtained by blowing down the sections $E_1\colon
w_0=w_1=0$ and $E_2\colon w_0=w_2=0$ along opposite rulings in the
hypersurface $W$ in
\[
\P \left(\sier Q\oplus \sier Q (1-n,-1)\oplus \sier Q(-1,1-n)\right)
\]
with bihomogeneous Equation \eqref{klb}:
\[
w_1w_2-\vp w_0^2=0 \;,
\]
where $\vp$ is a section of $\sier Q(n,n)$, 
defining a curve, whose singularities have only smooth branches.
\end{defn}

A $cA_k$ singularity $X$ of the form  $w_1w_2-g(x,y)=0$, where $g$
defines a curve singularity $\Gamma$ with $k+1$ smooth branches, has
$(k+1)!$ small resolutions. Simultaneous small resolutions of such
singularities have been studied by Friedman \cite{fr}. 
Each deformation of a given  small resolution $\wt X$ of $X$
blows down to  a deformation of $X$, as $H^1(\sier{\wt X})=0$
\cite{rie}. This gives a map $\Def_{\wt X}\to \Def_X$ of deformation
spaces (actually a map between the deformation functors), which is
an inclusion of germs: 
the induced map $H^1(\Theta_{\wt X})\to T^1_X$ of tangent spaces
has as kernel 
the vanishing local cohomology group  $H^1_C(\Theta_{\wt X})$, where
$C\subset\wt X$ is the exceptional curve \cite[Prop.~2.1]{fr}. 

The miniversal deformation of the threefold singularity
$X$, which is a double suspension of the curve singularity $\Gamma$,
is obtained by double suspension of the miniversal deformation of
$\Gamma$: it is of the form $w_1w_2-G(x,y;u_1,\dots,u_\tau)=0$.
For the infinitesimal deformations, one has the isomorphism
$T^1_X=\C\{w_1,w_2,x,y\}/(w_2,w_1,g_x,g_y, w_1w_2-g)\cong
\C\{x,y\}/(g_x,g_y, g)=T^1_{\Gamma}$.

The image of $\Def_{\wt X}$ in $\Def_X$ corresponds to the 
locus in the
deformation space of $\Gamma$, where one still  has the
factorisation in $k+1$ branches. Such deformations are obtained by
deforming the parametrisation of the curve $\Gamma$;
this means that we deform the composed map germ $\nu\colon \wt \Gamma
\to \C^2$, where  $\wt \Gamma$ is
the normalisation of $\Gamma$. 
These deformations are
exactly the  $\delta$-constant deformations, where $\delta$ is the
number of (virtual) double points of the singularity $\Gamma$.
The  locus of $\delta$-constant deformations in $\Def_{\Gamma}$
is smooth of codimension $\delta$, as 
$\Gamma$ has smooth branches (a convenient reference is
\cite[Sect.~II.2.7]{GLS}). 

Seen the other way round, the above discussion means that,
given a small resolution $\wt X$ of $X$, and a small deformation 
of $X$, obtained from a $\delta$-constant deformation of $\Gamma$,
there is a unique
simultaneous small resolution extending the given one.

But monodromy may prevent the existence of global
simultaneous small resolutions: even if for a family $\cx \to S$
all small resolutions of the singularities of the fibres $X_s$ extend
to a neighbourhood of $s\in S$, this does not mean that they
extend to a simultaneous small resolution   $\wt\cx\to S$.

\begin{example} \label{examp}
Consider the 1-parameter deformation
$w_1w_2-x(x+y^2-\ep)=0$ of the $cA_1$ singularity $w_1w_2-x(x+y^2)=0$.
The corresponding deformation of the $A_3$ curve singularity $\Gamma
\colon x(x+y^2)=0$
is the simplest $\delta$-constant deformation one can imagine:
the curve $\Gamma$ consists of a parabola and its vertical tangent, and
we move the parabola. So, the general fibre has two singular points,
lying at  $(w_1,w_2,x,y) =(0,0,0,\pm \sqrt{\ep})$, which are ordinary
$A_1$-singularities. Therefore, there are four small
resolutions. 
The central fibre has only one singularity, and 
two small resolutions, both given as closure of the graph of a map to $\P^1$.
These resolutions extend over the whole $\ep$-axis.  The maps are
\[
\frac{w_1}x=\frac{x+y^2-\ep}{w_2}=\frac\sigma\tau
\qquad\mbox{and}\qquad
\frac{w_1}{x+y^2-\ep}=\frac x{w_2}=\frac\sigma\tau\;.
\]
In both cases, the exceptional curve is a smooth $\P^1$ with normal bundle
$\co\oplus\co(-2)$, which splits under deformation in two curves
with normal bundle $\co(-1)\oplus \co(-1)$, see also \cite[p.~678]{fr}.
The extension of these two small resolutions to the general fibre
gives us two of the four small resolutions of the general fibre.
The other two are obtained by resolving  the singularity at $y=\sqrt{\ep}$ 
with one of the two resolutions above and the singularity at $y=-\sqrt{\ep}$
with the other. But this cannot be done consistently for all
$\ep\neq0$; one needs a base change. We consider for each $\ep$ in the unit disk
the set of all possible small resolutions. By local extendability, these fit together
into three components, two disks corresponding to the two global
small resolutions, and one punctured disk which is a $2\!:\!1$ covering 
of the punctured disk.
\end{example}

The description of the singularities translates into the following
description of the discriminant curve $D\colon \vp=0$. Let $\wt D$ be
the normalisation of $D$ and consider the composed map $\nu\colon \wt D
\to Q$. Then, $\nu$ is an immersion. We can deform the \KLB
manifold $Z$ by deforming $\wt D$ and the map $\nu$.
Let $r$ be the number of components
of  $D$ and $g$  the sum of the geometric genera of these components.
Let $\delta$ be the sum of the delta invariants
of the singular points of $D$. As the arithmetic genus of $D$, a curve of type
$(n,n)$,  is $(n-1)^2$, the cohomology exact sequence of the sequence
\[
0 \lra \sier D \lra \sier{\wt D} \lra {\textstyle \bigoplus_p}
 \sier{\wt D,p} /\sier{D,p} \lra 0
\]
gives the relation
$1+\delta + g = r + (n-1)^2$ between these invariants.

The space  $\P(H^0\sier Q(n,n)^*)$ parametrising discriminant
curves $D$ splits naturally into different strata, according to the
values of the total $\delta$. 
These strata may
have several components, which sometimes can be distinguished by the
number $r$ of components of $D$; for example, a curve of type $(3,3)$ with
$\delta=4$ can be a rational curve with four double points, or the
intersection of an elliptic curve (type $(2,2)$) with a conic (type
$(1,1)$).  
In each stratum, we consider the open set
where the normalisation $\nu$ is an immersion.
As will be shown as a consequence of Lemma \ref{lem} in the
next section, this open set 
is smooth of codimension $\delta$. 
It constitutes a versal (but not miniversal) deformation at each of its
points (versal for $\delta$-const deformations). 
We prefer to work with this larger family 
and not with the miniversal deformation of a given curve $D$.
In fact, as we shall see, the dimension of the miniversal deformation
depends not only on $\delta$, but also on the automorphism group of $D$,
whereas the dimension of the stratum depends on $\delta$ only.
Locally, one can always find a miniversal deformation by taking a
suitable transverse slice to the orbit of the group of
coordinate transformations. 

Let $\wl\Lambda_0\subset\P(H^0(\sier Q(n,n))^*)$ be a maximal open set in
a stratum, parametrising possible discriminant curves $D$ with 
smooth branches everywhere.
For each fibre of the family of conic bundles over it, we consider
all possible small resolutions. These fit together to a covering
$\Lambda_0\lra \wl\Lambda_0$ 
with finite fibres, which may not be locally trivial (see  Example 
\ref{examp}), but which still is a local homeomorphism. In this way, we get a family 
$W_{\Lambda_0} \to \Lambda_0$ with simultaneous small resolution
$\wt W_{\Lambda_0} \to \Lambda_0$. Blowing down sections finally
gives a family   $\wt Z_{\Lambda_0} \to \Lambda_0$, where each fibre
$\wt Z_{\lambda_0}$ is an \KLB manifold.

\begin{remark}
The largest stratum, where $\delta=0$, parametrises smooth curves
$D$, so no small resolution is needed. The next stratum $(\delta=1)$
gives irreducible curves with one ordinary double point or a cusp; we only consider
the open set parametrising ordinary double points, so there
are two small resolutions. Presumably, the double cover $\Lambda_0$
is globally irreducible. For $\delta=2$, the local situation of
Example \ref{examp}  can occur.
In the twistor case, we have $g=0$ and $r=n$ and so $\delta=n(n-1)$, which is
equal to $6$ if $n=3$, the case we specialise to later on.
\end{remark}

\section{Infinitesimal deformations}\label{sec:infdef}
To study the deformation theory of an \KLB manifold $\wt Z$, we
first determine the cohomology of its tangent sheaf. We also
consider infinitesimal deformations of $\wt Z$ together with two
surfaces $\wt S_1$, $\wt S_2$. In the twistor case, such computations
were done by LeBrun \cite{lb2} under the assumption that all
singularities are of type $A_1$. A closely related proof that
$h^i(\Theta_{\wt Z})=0$ for $i\geq 2$ was given by Campana
\cite{ca1,ca2}, under the same assumptions. Honda \cite{ho2} considered
the case of extra symmetry. All these computations do not use the
additional real structure and they generalise to our situation. We
follow the same strategy. As our computations mostly take place on
the singular spaces $Z$ and $W$, we 
use the general theory
of the cotangent complex, but for our purposes, the treatment in
\cite{pa} suffices.

Let $W$ be a conic bundle over $Q=\P^1\times \P^1$ with Equation
\eqref{klb}: $w_1w_2-\vp w_0^2=0$.
Let $r$ be the number of components of the discriminant curve
$D\colon\vp=0$ and $g$  the sum of the geometric genera of these
components. Let $\delta$ be the sum of the delta invariants of the
singular points of $D$. 

\begin{proposition}
\label{raakw} Let $\sigma\colon\wt W\to W$ be any small resolution
of the threefold $W$. Suppose $n\ge3$. Then, $h^0(\Theta_{\wt W})=2$,
if $D$ admits a $1$-dimensional symmetry group, and $h^0(\Theta_{\wt
W})=1$ otherwise, $h^1(\Theta_{\wt W})-h^0(\Theta_{\wt
W})=n^2+2n-7-\delta=g-r+4n-7$ and $h^j(\Theta_{\wt W})=0$ for
$j\geq2$.
\end{proposition}

\begin{proof} We first study the tangent cohomology of $W$. It turns
out that every deformation of $W$ comes about by changing the
equation inside the $\P^2$-bundle 
\[
\P:= \P \left(\sier Q\oplus
\sier Q (1-n,-1)\oplus \sier Q(-1,1-n)\right) \mapright\pi Q\;.
\] 
To prove this, we look at  global tangent cohomology. Consider the
embedding $i\colon W\to \P$. The functor $\Def_{W/\P}$ describes the
deformations of $W$ and $i$ inside $\P$. We want to compare it with
the deformation functor of $W$. For the corresponding tangent
cohomology sheaves, we have a long exact sequence
\cite[(4.3.1)]{pa}
\[
0\lra \ct^0(W/\P) \lra \ct^0_W \lra \ct^0(\P,\sier W) \lra  \ct^1(W/\P)
 \lra \ct^1_W \lra \cdots\;.
\]
As $i$ is an embedding, $\ct^0(W/\P)=0$ and $\ct^1(W/\P)=\cn$, the
normal sheaf of $W$ in $\P$. The long exact sequence reduces to the
familiar exact sequence
\[
0 \lra \ct^0_W \lra \Theta_{\P}|_W \lra \cn \lra \ct^1_W \lra 0\;,
\]
because $\ct^i(\P,\sier W) =0$  for $i>0$ ($\P$ is smooth) and
$\ct^i_W=0$ for $i\geq2$ as $W$ has only hypersurface singularities.
The local-to-global spectral sequence now gives $T^0(W/\P)=0$ and
$T^i(W/\P)=H^{i-1}(\cn)$. The global version of the long exact
sequence is
\[
0 \lra T^0(W/\P) \lra T^0_W \lra T^0(\P,\sier W) \lra  T^1(W/\P)
 \lra T^1_W \lra \cdots\;.
\]
As (again by a local-to-global argument) $T^i(\P,\sier
W)=H^i(\Theta_{\P}|_W) $, one obtains
\begin{equation}\label{eq:theta}
0  \lra T^0_W \lra H^0(\Theta_{\P}|_W) \lra  H^{0}(\cn)
 \lra T^1_W \lra \cdots\;.
\end{equation}
The relevant groups have been  computed by Campana \cite[Lemma 3.8]{ca2} and
LeBrun (in the  proof of \cite[Prop.~1]{lb2}). 
The tangent bundle to $\P$ is described by the Euler sequence
\[
0\lra \sier{\P} \lra (\pi^* \ce_{Q}^{\vee})(1) \lra \Theta_{\P} \lra 
  \pi^*\Theta_{Q} \lra 0\;,
\]
where $\ce_{Q}$, as before, is the vector bundle we have used to define 
$\P=\P(\ce_{Q})  \mapright\pi Q$. 
We first twist this sequence
with the ideal sheaf $\sier{\P}(-W)=\sier{\P}(-2)\otimes
\pi^*\sier Q(-n,-n)$ to compute $H^i(\Theta_{\P}(-W))$
using the Leray spectral sequence for the map $\pi$.
As $H^i(\sier{\P^2}(-1))=H^i(\sier{\P^2}(-2))=0$
for all $i$, we find that
$H^i(\Theta_{\P}(-W))=0$. This implies that
$H^i(\Theta_{\P}|_W)\cong H^i(\Theta_{\P})$, which is independent of
$W$. We compute these groups using the exact sequence 
\[
0\lra \Theta_\pi \lra \Theta_{\P} \lra 
  \pi^*\Theta_{Q} \lra 0\;,
\]
defining the sheaf $\Theta_\pi$ of vertical vector fields.  
Using $\pi_{*}\sier{\P}=\sier{Q}$, $R^{1}\pi_{*}\sier{\P}=0$ and
$\pi_{*}\sier{\P}(1)=\ce_{Q}$, the exact sequence 
\(0\rightarrow \sier{\P} \rightarrow (\pi^* \ce_{Q}^{\vee})(1)
\rightarrow \Theta_\pi \rightarrow 0\), obtained from the Euler
sequence, implies $H^i(\Theta_\pi)=H^i(Q,\End_0(\ce_{Q}))$,
where $\End_0(\ce_{Q})$ is the sheaf of traceless endomorphisms of $\ce_{Q}$.
One finds $h^0(\Theta_{\P})=4n+8$,
$h^1(\Theta_{\P})=2(n-1)(n-3)$ and $h^i(\Theta_{\P})=0$ for
$i\geq2$.

The cohomology of $\cn$ can be computed from
\[
0 \lra \sier\P \lra \sier\P(2)\otimes\pi^*\sier Q(n,n) \lra \cn
\lra 0 \;.
\]
The result is \cite{lb2}: $h^0(\cn)=n^2+6n+1$, $h^1(\cn)=2(n-1)(n-3)$
and $h^i(\cn)=0$ for $i\geq2$. Here, one uses
$\pi_*(\sier{\P}(k))=\mathop{\rm Sym}^k(\ce_{Q})$ for $k\geq0$.

All these groups occur  in our long exact sequence
\eqref{eq:theta}, which now
reads
\[
0 \to T^0_W \lra H^0(\Theta_{\P})  \lra H^0(\cn)
  \lra T^1_W
\lra H^1(\Theta_{\P})  \lra H^1(\cn)
  \lra T^2_W \to 0\;.
\]
The map  $H^1(\Theta_{\P}) \to H^1(\cn)$ is an isomorphism.
For $n=3$, both groups are zero. For $n>3$, the
proof is given in \cite[3.11]{ca2} and \cite{lb2}.
One realises  $H^1(\Theta_{\P})$ as
$H^1(\sier Q(2-n,n-2))w_2\frac{\partial}{\partial w_1}
\oplus H^1(\sier Q(n-2,2-n))w_1\frac{\partial}{\partial w_2}$
and $ H^1(\cn)$ as 
$H^1(\sier Q(2-n,n-2))w_1^2\oplus H^1(\sier Q(n-2,2-n))w_2^2$.
The map is given by evaluating a vector field on the polynomial
$w_1w_2-\vp w_0^2$ defining $W$; it clearly is an isomorphism. 
We note that it is independent
of the precise form of the polynomial $\vp$.
We conclude that $T^i_W=0$
for $i\ge2$, so the deformation space of $W$ is smooth. Furthermore,
every deformation of $W$ comes from deforming the equation of $W$
inside $\P$ and we get the exact sequence
\[
0\lra T^0_W\lra H^0(\Theta_{\P}) \to H^0(\cn)\lra T^1_W\lra0\;,
\]
giving us a formula for the dimension of $T^1_W$, once we know the
symmetry group of $W$.
The automorphism group of $\P$ has dimension $4n+8$. Automorphisms
of the  form 
\[
(w_0:w_1:w_2)\mapsto (w_0:w_1+\psi_1 w_0: w_2+\psi_2 w_0)\;,
\]
with $\psi_1\in H^0(\sier Q(n-1,1))$ and $\psi_2\in H^0(\sier Q(1,n-1))$,
which form a $4n$-dimensional family, do not preserve the defining equation
$w_1w_2-\vp w_0^2=0$ for $W$. The $1$-parameter subgroup 
$\la\cdot(w_0,w_1,w_2)=(w_0,\la w_1,\la^{-1}w_2)$ of the
two-dimensional diagonal action on $(w_0:w_1:w_2)$
leaves $W$ invariant. Finally, one has the automorphisms of the base quadric 
$Q$. The dimension of the subgroup of automorphisms preserving the
curve $D$ is at most $1$. A description of the occurring
cases is given in remark \ref{rmk:automor} below.
This implies $1\le \dim T^{0}_{W} \le 2$.

To get to $\wt W$, we use the Leray spectral sequence
$H^p(R^q\sigma_* \Theta_{\wt W}) \Rightarrow H^{p+q}(\Theta_{\wt
W})$ and compare it to the local-to-global spectral sequence for
$T^i_W$, as in \cite[(3.4)]{fr}: 
\[
\xymatrix
{
  0                              \ar[r]
  &H^1(\ct^0_W)                   \ar[r]\ar[d]
  &H^1(\Theta_{\wt W})              \ar[r]\ar[d]
  &H^{0}(R^1\sigma_*\Theta_{\wt W})  \ar[r]\ar[d]
  &H^2(\ct^0_W)                   \ar[r]\ar[d]
  &H^2(\Theta_{\wt W})              \ar[r]\ar[d]
  &0
  \\
  0             \ar[r]
  &H^1(\ct^0_W)  \ar[r]
  &T^1_W         \ar[r]
  &H^0(\ct^1_{W}) \ar[r]
  &H^2(\ct^0_W)  \ar[r]
  &T^2_{W}       \ar[r]
  &0\makebox[0pt][l]{$\;.$}
}\hphantom{\;.}
\]
We have shown that $T^2_{W}=0$. In general the map
$ T^1_W \to  H^0(\ct^1_{W}) $ will not be surjective,
that is, the singular points do not impose independent conditions. So,
$H^2(\ct^0_W)$ will be non-zero. 
We compare the image of the deformations of $\wt W$ in the deformations of 
$W$ with the image of the  $\delta$-constant 
deformations of the curve $D\colon \vp=0$ (the discriminant of the conic fibration
$W$) in all deformations of $D$.

Given a deformation of $W$, we get a deformation of the
discriminant curve and conversely, given a deformation $\vp'$ of
$\vp$, we can write the equation $w_1w_2-\vp'w_0^2$. This shows that
the natural map $\Def_{W/\P}\to \Def_{D/Q}$ is a smooth surjection
of functors. On the tangent space level, the  kernel is the space of
global sections of vertical vector fields.

To study $\delta$-constant deformations of $D$, we let $\wt D$ be the
normalisation of $D$ and we consider the composed map $\nu\colon \wt
D\to Q$. We study the deformation functor  $\Def_{\wt D/Q}$. The
tangent cohomology $\ct^1(\wt D/Q)$ is the normal bundle along the map $\nu$
occurring in the exact sequence
\[
0\lra \Theta_{\wt D} \lra \nu^*\Theta_Q \lra \cn_\nu \lra 0 \;.
\]
The assumption that all branches at the singular points of $D$ are
smooth gives that $\cn_\nu$ is a line bundle.
The tangent space to  $\Def_{\wt D/Q}$ is $H^0(\cn_\nu)$ and obstructions
lie in $H^1(\cn_\nu)$.
A straightforward calculation shows 
$\chi\left(\nu^*\Theta_Q\right) = 4n+2(r-g)$,
$\chi\left(\Theta_{\wt D}\right) = 3(r-g)$ and hence
$\chi\left(\cn_\nu\right) = 4n-(r-g)=(n+1)^{2}-1-\delta$.

\begin{lemma}\label{lem}
$H^1(\cn_\nu)=0$.
\end{lemma}

\begin{proof}
As the degree of $\nu^*\Theta_Q$ is positive on every component
$\wt D_i$ of $\wt D$, we have $\deg \cn_\nu|_{\wt D_i}>2g(\wt D_i)-2$
and $H^1(\wt D_i,\cn_\nu)=0$.
\end{proof}

\noindent We conclude that $\Def_{\wt D/Q}$ is smooth of the
expected codimension $\delta$ in $\Def_{D/Q}$, so $H^1(\Theta_{\wt W})$ 
also has the expected codimension in $T^1_W$. This is also the
codimension of $H^{0}(R^1\sigma_*\Theta_{\wt W}) $ in 
$ H^0(\ct^1_{W})$. Therefore, 
$H^1(\Theta_{\wt W}) \to H^{0}(R^1\sigma_*\Theta_{\wt W})$ and $T^1_W \to
H^0(\ct^1_{W})$ have the same cokernel. This shows that $ H^2(\Theta_{\wt
W})=0$.
\end{proof}

\begin{remark} \label{rmk:automor}
The symmetries of the
curve $D$ on $Q=\P^1\times \P^1$ can be determined with the same
methods as for plane curves \cite{wdp}. We are grateful to C.T.C.\ Wall
for help on this point. Only a few cases arise. If
the symmetry group is semisimple, the action can be diagonalised.
The monomials $s_0^is_1^{n-i}t_0^jt_1^{n-j}$ occurring in the
equation $\vp=0$ must be such that  the points $(i,j)$ lie on a
straight line in $[0,n]\times [0,n]$. If this line joins two
opposite corners, we get an equation of the form
$\prod_{i=1}^n(a_is_0t_0-b_is_1t_1)$ or with $t_0$ and $t_1$ interchanged; 
the twistor spaces with torus
action studied by Honda \cite{ho2} are of this type. The total
$\delta$ increases if one or two of the conics degenerate, giving the
cases $s_0t_0\prod_{i=1}^{n-1}(a_is_0t_0-b_is_1t_1)$ and
$s_0s_1t_0t_1\prod_{i=1}^{n-2}(a_is_0t_0-b_is_1t_1)$. The line from
$(1,0)$ to $(n,n-1)$  leads to
$s_0t_1\prod_{i=1}^{n-1}(a_is_0t_0-b_is_1t_1)$, containing the
degenerate conic $s_0t_1$. For general $n$, these are the only
possibilities to obtain a reduced curve with singularities with
smooth branches, but for $n=3$ and $n=4$, one also has
$t_0t_1(s_0^nt_0^{n-2}-s_1^nt_1^{n-2})$. For unipotent symmetry
group, there are a priori two cases, but only
$\la\cdot(s_0:s_1\semic t_0:t_1) = (s_0:s_1+\la s_0\semic t_0:t_1+\la t_0)$
gives reduced curves, namely tangent conics
$\prod_{i=1}^n(s_0t_1-s_1t_0-c_is_0t_0)$, one of which may
degenerate to the line pair $s_0t_0$.
\end{remark}
 
\begin{example}
In our discussion of Honda's deformation with torus action
\cite{ho2} in Example \ref{nonumber} in Section \ref{examples},
we consider a  discriminant  curve with one node
and two singularities of type $\wt E_7$ (four lines through one
point). In suitable coordinates its equation is
\[
s_0t_0(a_1s_0t_0-b_1s_1t_1)(a_2s_0t_0-b_2s_1t_1)(a_3s_0t_0-b_3s_1t_1)=0\;.
\]
The cross ratio of the four lines at the two singular points
$(1:0;0:1) $ and $(0:1;1:0)$ is the same. This shows that the stratum
does not form a versal $\delta$-const deformation of 
each of the singular points separately.
We have here an example where 
the singular points do not impose independent conditions.
\end{example}

\begin{theorem}
\label{raakz} Let $\wt Z$ be an \KLB manifold, that is,
$\sigma\colon\wt Z\to Z$ is any small resolution of the threefold
$Z$, obtained by blowing down sections of $W$ along opposite
rulings. Suppose $n\ge3$. Then, $h^0(\Theta_{\wt Z})=2$, if the
discriminant curve $D$ admits a one-dimensional symmetry group, and
$h^0(\Theta_{\wt Z})=1$ otherwise, $h^1(\Theta_{\wt
Z})-h^0(\Theta_{\wt Z})=n^2+6n-15-\delta=g-r+8n-15$ and
$h^j(\Theta_{\wt Z})=0$ for $j\geq2$. In particular, the deformation
space is smooth.
\end{theorem}

\begin{proof}[Proof\/ {\rm(cf.~\cite[Prop.~2 and Cor.~1]{lb2})}] Let
$\beta\colon W\to Z$ be the blowing down map. As the singular points
of $W$ are disjoint from the exceptional locus of the map $\beta$,
we can choose a small resolution $\sigma\colon\wt W\to W$ such that
$\wt W$ blows down to $\wt Z$. The image under $\beta$
of the section $E_1$ is the
curve $C_1$ and the curve $C_2$ is the image of $E_2$.
These objects are not changed under the small resolution, but if
we consider them on $\wt W$ or $\wt Z$, we call them $\wt E_i$ and $\wt C_i$. 
Let $(\Theta_{\wt
Z})_{\wt C_1,\wt C_2}$ be the sheaf of those vector fields which are tangent
to  $\wt C_1$ and $\wt C_2$. Then $\wt\beta_* \Theta_{\wt W}=(\Theta_{\wt
Z})_{\wt C_1,\wt C_2}$ and $R^i\wt\beta_* \Theta_{\wt W}=0 $ for $i>0$. A
local computation shows that on $\wt Z$, we have the exact sequence
\[
0\lra (\Theta_{\wt Z})_{\wt C_1, \wt C_2} \lra \Theta_{\wt Z}
\lra \cn_{\wt C_1} \oplus \cn_{\wt C_2} \lra 0
\]
\cite[Proof of Prop.~2]{lb2}. The normal bundle of the curves $\wt C_1$
and $\wt C_2$ is $\co(1-n)\oplus \co(1-n)$, so the result follows from
Proposition \ref{raakw}: we have $H^j(\Theta_{\wt Z}) \cong
H^j( \Theta_{\wt W})=0 $ for $j\geq2$ and the exact sequence
\begin{equation}\label{thetawz}
0\lra H^1( \Theta_{\wt W}) \lra H^1(\Theta_{\wt Z})
\lra H^1(\cn_{\wt C_1}) \oplus H^1(\cn_{\wt C_2}) \lra 0\;.
\end{equation}
\end{proof}

For our construction, it is natural not only to consider the
deformation theory of $\wt Z$ itself, but also deformations of $\wt
Z$ together with two divisors $\wt S_1$, $\wt S_2$.

Let $l_1=\pr1^{-1}(s)$ be a line  in first the ruling of $Q$, which
is not blown down under the contraction $E_1\to C_1$. We assume that
$l_1$ is not a component of the discriminant curve $D$. Let
$R_1=\pi^{-1}(l_1)\subset W$ be its inverse image under the conic-bundle
projection $\pi:W\rightarrow Q$. If $l_1$ does not
pass through the singular points of $D$, we can consider $R_1$  also
as surface $\wt R_1$ on a small resolution $\wt W$. In the opposite
case, we denote by $\wt R_1$ the strict transform of our surface on
$\wt W$. By blowing down, we obtain the divisors $S_1:=\beta_* (R_1)$
on $Z$ and $\wt S_1:=\wt\beta_* (\wt R_1)$ on $\wt Z$. The surface
$\wt S_1$ contains the curve $\wt C_1$. Choosing a line $l_2$ in the
other ruling yields a surface $\wt S_2$. The choice of $l_1$ and
$l_2$ is determined by giving a point on the quadric $Q$.
Explicitly, with the coordinates of Equations \eqref{chartcoord},
a point $(a_0: a_1\semic b_0:b_1)$  
gives the line $l_1\colon a_1s_0-a_0s_1=0$ on the quadric. On $\wt W$,
the surface $\wt S_1$ is then given by 
$a_1v_{0\semic 0}-a_0v_{1\semic 0}=0$ and
$a_1v_{0\semic n-1}-a_0v_{1\semic n-1}=0$. It follows 
from Equations \eqref{chart} that the normal bundle of
$\wt C_1$ in $\wt S_1$ is  $\co(1-n)$.

In general, $\wt S_1$
and $\wt S_2$ are smooth and intersect transversally, but they may
have at most rational singularities if the line in question is
tangent to the curve $D\colon \vp=0$. 

We now look at deformations of the triple $(\wt Z,\wt S_1, \wt S_2)$. The
relevant deformation theory is  that of the map $f\colon \wt S_1\amalg \wt S_2
\to \wt Z$. Here, it is allowed to deform all the spaces $\wt S_1$, $\wt
S_2$, $\wt Z$ and the map $f$. Infinitesimal deformations are given
by $T^1_f$, while obstructions land in $T^2_f$. These groups occur
in the following long exact sequence, which relates them to
deformations of $\wt Z$, cf.~\cite[(4.5.2)]{pa} :
\[
\cdots \lra T^i{\left(\wt S_1\amalg \wt S_2/\wt Z\right)} \lra T^i_f
\lra H^i(\Theta_{\wt Z})  \lra
 \cdots \;.
\]
The local tangent cohomology $\ct^i{\left(\wt S_1\amalg \wt S_2/\wt
Z\right)}$ is concentrated in degree 1 and is equal to $\cn_{\wt
S_1}\oplus \cn_{\wt S_2}$. Therefore, $ T^i{\left(\wt S_1\amalg \wt
S_2/\wt Z\right)}= H^{i-1}(\cn_{\wt S_1})\oplus H^{i-1}(\cn_{\wt
S_2})$. On $\wt W$, the normal bundle of $\wt R_i$ is trivial, so by
blowing down the section $E_i$, we obtain $\cn_{\wt S_i}\cong \sier
{\wt S_i}(\wt C_i)$, $i=1,2$.

\begin{lemma}\label{lemfour}
The map $H^1(\Theta_{\wt Z}) \to H^1(\cn_{\wt S_1})\oplus
H^1(\cn_{\wt S_2})$ is surjective.
\end{lemma}

\begin{proof} We first study the normal bundle of the surfaces
separately. According to our notational convention (from  the 
end of the Introduction), we drop indices.
The isomorphism $\cn_{\wt S}\cong \sier {\wt S}(\wt C)$ gives that
$H^i(\cn_{\wt S}(-\wt C))=H^i(\sier{\wt S} )=0$ for $i>0$ as ${\wt S}$
is a rational surface with at most rational singularities. From the
exact sequence
\[
0 \lra \cn_{\wt S}(-\wt C) \lra \cn_ {\wt S}\lra \cn_{\wt S}|_{\wt C} \lra 0\;,
\]
we conclude that $H^1(\cn_{\wt S})=H^1(\cn_{\wt S}|_{\wt C})$. This last
group also occurs in the normal bundle sequence of $\wt C$:
\[
0\lra \cn_{\wt C/\wt S} \lra \cn_{\wt C} \lra  \cn_{\wt S}|_{\wt C}  \lra 0\;.
\]
Now we return to the global situation and consider the commutative
diagram
\begin{equation}\label{comdia}
\begin{matrix}
H^1(\Theta_{\wt Z})
  &\lra&
  H^1(\cn_{\wt S_1})\oplus H^1(\cn_{\wt S_2})
  \\[1mm]
\Big\downarrow&&\mapdown{\cong}\\[2mm]
H^1(\cn_{\wt C_1})\oplus H^1(\cn_{\wt C_2})
  &\lra&
  H^1(\cn_{\wt S_1}|_{\wt C_1})\oplus H^1(\cn_{\wt S_2}|_{\wt C_2})
  \makebox[0pt][l]{$\;.$}
\end{matrix}
\hphantom{\;.}
\end{equation}
By the exact sequence  \eqref{thetawz}, the first vertical map is
surjective. 
\end{proof}

\begin{theorem}
\label{raakzss} The deformation space of the triple $(\wt Z,\wt S_1,
\wt S_2)$ is smooth of dimension $n^2+4n-8-\delta=g-r+6n-8$, except
when the curve $D\cup l_1 \cup l_2$ admits a one-dimensional
symmetry group, in which case the dimension is
$n^2+4n-7-\delta=g-r+6n-7$.
\end{theorem}

\begin{proof}
We look at the long exact sequence containing $T^1_f$:
\begin{multline*}
 0 \lra T^0_f \lra H^0(\Theta_{\wt Z})  \lra
H^0(\cn_{\wt S_1})\oplus H^0(\cn_{\wt S_2}) \lra {}\\{}
\lra T^1_f \lra H^1(\Theta_{\wt Z})  \lra H^1(\cn_{\wt S_1})\oplus
H^1(\cn_{\wt S_2}) \lra T^2_f \lra 0 \;.
\end{multline*}
If $h^0(\Theta_{\wt Z}) =2$, then the map $H^0(\Theta_{\wt Z})  \lra
H^0(\cn_{\wt S_1})\oplus H^0(\cn_{\wt S_2})$ is the zero-map only if
the symmetries of the curve $D$ also preserve $l_1\cup l_2$.
Otherwise the image is one-dimensional.
If $h^{0}(\Theta_{\wt Z}) = 1$, this map is the zero-map again. This shows
that $\dim T_{f}^{0} = 1$, except when the curve $D\cup l_{1}\cup l_{2}$ admits
a one-dimensional symmetry group, in which case $\dim T_{f}^{0} = 2$.

The curve $\wt C_i$ has self-intersection $1-n$ on $\wt S_i$, so also the
degree of $\cn_{\wt S_i}|_{\wt C_i}$ is $1-n$ ($i=1,2$). Therefore,
$H^0(\cn_{\wt S_i})\cong H^0(\sier {\wt S_i})$ is one-dimensional,
corresponding to moving the line $l_i$,  which obviously induces a
deformation of $f$ (which can be trivial if it comes from an
automorphism of $D$). Furthermore, $\dim H^1(\cn_{\wt S_i})=n-2$. The
result now follows from Lemma \ref{lemfour} and Theorem \ref{raakz}.
\end{proof}

For $n\ge4$, the dimension of the deformation space of the triples
$(\wt Z, \wt S_1, \wt S _2)$ is less than the dimension of the
deformation space of the \KLB manifold $\wt Z$ itself. 
In the absence of extra symmetry, these dimensions are the same for $n=3$,
but the map  $T^1_f\to H^1(\Theta_{\wt Z})$ is not surjective, as it
has always  a non-trivial kernel (of dimension
$2+\dim T^{0}_{f} - h^{0}(\Theta_{\wt Z})$, which is either $1$ or $2$).
In fact, for all $n$, the codimension of the image  
of $T_{f}^{1}$ in $H^{1}(\Theta_{\wt Z})$ 
is $h^1(\cn_{\wt C_1})+h^1(\cn_{\wt C_2}) = 2(n-2)$.

We now compute the image of $T_{f}^{1}$ in $H^{1}(\Theta_{\wt Z})$.
It is the kernel of the
map $H^1(\Theta_{\wt Z})  \lra H^1(\cn_{\wt S_1})\oplus
H^1(\cn_{\wt S_2})$, which is the same as the kernel of
the composed map $H^1(\Theta_{\wt Z})  \lra 
H^1(\cn_{\wt C_1})\oplus H^1(\cn_{\wt C_2}) \lra
  H^1(\cn_{\wt S_1}|_{\wt C_1})\oplus 
    H^1(\cn_{\wt S_2}|_{\wt C_2})$, by the diagram
\eqref{comdia}. The kernel of the first surjection consists, by the exact 
sequence \eqref{thetawz}, of the deformations obtained by changing the 
equation of $W$.
For the second map,  we explicitly
determine the restriction map $H^1(\cn_{\wt C_1})\to H^1(\cn_{\wt
S_1}|_{\wt C_1})$, the case with $i=2$ being similar. We use the
coordinates of  Equations \eqref{chartcoord}.
We cover $\wt C_1$ with two coordinate patches $\{t_0\neq0\}$ and
$\{t_1\neq0\}$ and compute \v Cech cocycles.
In the chart $t_0=1$, we have coordinates 
$v_{0\semic 0}$, $v_{1\semic 0}$ and $t_1$,
with $\wt C_1$ given by $v_{0\semic 0}=v_{1\semic 0}=0$. 
A basis of $H^1(\cn_{\wt C_1})$
is conveniently represented by elements in $H^1(\Theta_{\wt
Z}|_{\wt C_1})$. We take the following  \v Cech cocycles:
\[
\frac1{t_1^i}\frac\partial{\partial v_{0\semic 0}}, \quad
\frac1{t_1^i}\frac\partial{\partial v_{1\semic 0}}, \qquad i=1,\dots,
n-2\;.
\]
The kernel of $H^1(\cn_{\wt C_1})\to H^1(\cn_{\wt
S_1}|_{\wt C_1})$ lies in $H^1(\cn_{\wt
C_1/\wt S_1})$. It is represented by
cocycles of vector fields, which are tangent to $\wt S_1$.
This means that evaluation on  a defining equation for $\wt S_1$
yields zero, in the appropriate cohomology group, which is $H^1(\sier
{\wt C_1}\otimes \sier{\wt S_1}(\wt C))$. The choice of a pair of lines
$(l_1,l_2)$ is determined by a point $(a_0: a_1\semic b_0:b_1)$ on
$Q$. We start out from the line $a_1s_0-a_0s_1=0$ on the quadric
$Q$. In our chart, the surface $\wt S_1$ is then given by 
$a_1v_{0\semic 0}-a_0v_{1\semic 0}=0$, as remarked above.
We compute the action of a vector field:
\[
\sum\left( \frac{\tau_{0\semic i}}{t_1^i}\frac\partial{\partial v_{0\semic 0}}+
\frac{\tau_{1\semic i}}{t_1^i}\frac\partial{\partial v_{1\semic 0}}\right)
(a_1v_{0\semic 0}-a_0v_{1\semic 0}) = \sum
\frac{a_1\tau_{0\semic i}-a_0\tau_{1\semic i}}{t_1^i} \;,
\]
where $\tau_{0\semic i}$, $\tau_{1\semic i}$ are coordinates on
$H^1(\cn_{\wt C_1})$. The kernel of the map $H^1(\cn_{\wt C_1})\to
H^1(\cn_{\wt S_1}|_{\wt C_1})$ is therefore given by the equations
$a_1\tau_{0\semic i}-a_0\tau_{1\semic i}=0$. For $l_2$, we similarly find
equations $b_1\sigma_{0\semic i}-b_0\sigma_{1\semic i}=0$.

For $n=3$, we have only one equation 
$a_1\tau_{0\semic 1}-a_0\tau_{1\semic 1}=0$.
Considered in $\wt C_1\times H^1(\cn_{\wt C_1}) \cong \P^1 \times \C^2$,
it describes
the blow-up of the origin in $ H^1(\cn_{\wt C_1}) $.
Likewise, the  blow-up of the origin in $ H^1(\cn_{\wt C_2}) $
is given in $\wt C_2\times H^1(\cn_{\wt C_2})$ by
$b_1\sigma_{0\semic 1}-b_0\sigma_{1\semic 1}=0$.
Therefore, the union of the images of  $T_{f}^{1}$ for all choices of 
$\wt S_{1} \cap \wt S_{2}$ is $H^{1}(\Theta_{\wt Z})$, provided that the 
discriminant curve $D$ does not contain a line.

For $n\ge4$, the union of all images for all pairs of lines
(again under the assumption that $D$ does not contain a line)
is given by the equations
\[
\Rank \begin{pmatrix} \sigma_{01} & \dots & \sigma_{0,n-2} \\
                \sigma_{11} & \dots & \sigma_{1,n-2} \end{pmatrix}
  \leq 1 \;,  \quad
\Rank \begin{pmatrix}\tau_{01} & \dots & \tau_{0,n-2} \\
                \tau_{11} & \dots & \tau_{1,n-2} \end{pmatrix}
  \leq 1 \;.
\]
They describe the
tangent cone to the image of the deformation space of triples. 

For $n=3$, the main result of this paper is the construction
of a family, which is versal for deformations of triples
$(\wt Z,\wt S_1, \wt S_2)$, along the locus of \KLB manifolds.
In Section \ref{examples}, we consider several
examples in which we explicitly describe a global family. In these examples, 
the structure of the map to the deformation space of  $\wt Z$, 
transverse to the locus  of \KLB manifolds, is exactly what the above computation
suggests: the map is ${\rm Bl}_0 \C^2 \times {\rm Bl}_0 \C^2 \to
\C^2 \times \C^2$ with fibre over the origin $Q=\P^1\times\P^1$.
On the deformation spaces, we have a $\C^*$  action. 
The quotient
of $\C^2 \times \C^2$ by this $\C^*$-action is a three-dimensional
ordinary double point. Such singularities can be encountered
on compactifications of moduli spaces of double solids, or of
the branch quartics (e.g.\ for Kummer surfaces).

\section{A versal family for $n=3$}
In this section, we construct the family, which gives the wanted
deformation of \KLB manifolds with $n=3$, or more precisely
of triples $(\wt Z,\wt S_1, \wt S_2)$, and prove versality of the
constructed family.

Let $\wt Z$ be an \KLB manifold. It is a fibre
of a family $\wt Z_{\Lambda_0}\to\Lambda_0$, as constructed at the
end of Section~\ref{sec:KL}. It is defined from a conic bundle
bundle $W$, lying in a family $W_{\Lambda_0}\to\Lambda_0$,  with an equation
of the form \eqref{klb}, that is,
$w_1w_2-\vp w_0^2=0$,
in
\[
\P \left(\sier Q\oplus \sier Q (-2,-1)\oplus \sier
Q(-1,-2)\right)\;.
\]
The singularities of the curve $D\colon \vp=0$ have only smooth
branches.  We also
choose two surfaces $\wt S_1$, $\wt S_2$, by picking a pair of
intersecting lines on $Q$, that is,  by choosing a tangent plane.
We do not allow\footnote{This assumption is made to exclude
non-isolated singularities along such a line in our construction.
Presumably, our arguments can be extended to this case.} 
that one of these lines  is a component of $D$. Then each line
on $Q$ intersects $D$ in 3 points (counted
with multiplicity).

We add the chosen lines to the discriminant curve, which therefore
becomes a curve of type $(4,4)$. The successful idea is, to consider the
resulting equation in the bundle $\P \left(\sier Q\oplus \sier Q
(-2,-2)\oplus \sier Q(-2,-2)\right)$, because this bundle can be 
extended to $\P^3$. 

We therefore choose an embedding of $Q=\P^1\times \P^1$ as
smooth quadric in $\P^3$, whose  equation again will be called  $Q$.
The section $\vp\in H^0(Q,\sier Q(3,3))$ can be written as
restriction of a cubic form $K\in H^0(\P^3,\sier {\P^3}(3))$. Such a
form is not unique, as it can be altered with a multiple of $Q$, but
for now we choose a lift; we will end up by considering all possible
choices. A section $\vp_1\in H^0(Q,\sier
Q(1,1))$  has a unique lift to a linear form 
$K_1\in H^0(\P^3,\sier {\P^3}(1))$.
If $\vp$ is of the form 
$\vp_1\vp_2$ with $\vp_2\in H^0(Q,\sier
Q(2,2))$, we  lift $\vp$ to a reducible form $K=K_1K_2$. In particular, if
$\vp$ is the product of three forms of type $(1,1)$, we take $K$ as a 
product of linear factors. Let $L$ be an equation of the chosen
tangent plane.

We now work with  the $\P^{2}$-bundle
\[
\P(\co\oplus \co(-2)\oplus
\co(-2))\rightarrow \P^{3}
\]
over $\P^{3}$, with fibre coordinates $(y_{0}:y_{1}:y_{2})$
corresponding to the three direct summands in this order. Let $Y$ be
the subspace given by the equations
\begin{equation}\label{yyy}
\begin{split}
y_{1}y_{2}-KLy_{0}^{2}&=0\;,\\
Q&=0\;.
\end{split}
\end{equation}
The space $Y$ is again a conic bundle over the quadric $Q$, but we
have already written it in terms of $\P^3$.

\begin{lemma}\label{lem:birational}
The conic bundle $Y$ is birational to the conic bundle $W$.
\end{lemma}

\begin{proof}
Given the  linear form $L$, we choose two sections $\psi_1\in
H^0(\sier Q(1,0))$ and $\psi_2\in H^0(\sier Q(0,1))$ such that their
product is the restriction of $L$ to $Q$. We now define a rational
map  between bundles over $Q$ by the formula
\[
(y_0:y_1:y_2)=(w_0:\psi_2w_1:\psi_1w_2)\;.
\]
Note that the map depends on the choice of  $\psi_1$ and $\psi_2$,
but different maps are connected by the
$\C^*$-action on $W$. It is straightforward to check that this map is indeed
birational. A detailed analysis of it can be found after Lemma
\ref{lem:smooth}.
\end{proof}

We introduce deformation variables $\al1$ and $\al2$ and consider the
equations 
\begin{equation}\label{ydef}
\begin{split}
y_{1}y_{2}-KLy_{0}^{2}&=0\\
\al2y_{1}+\al1y_{2}-Qy_{0}&=0
\end{split}
\end{equation}
on the $\P^2$-bundle $\P(\co\oplus \co(-2)\oplus \co(-2))\rightarrow \P^{3}$.
From equations of this form, we eventually get a versal family for 
our \KLB manifold. This involves varying $K$ and $L$ in an appropriate
way; in general, the linear form $L$ 
does no longer define a tangent plane to the quadric $Q$. 

We stress the fact that Equations \eqref{ydef} do not define a deformation of
the space $Y$: the fibre over $\al1=\al2=0$ (for a specific value of $KL$) 
is reducible, having $Y$ as one of its components only. On the other hand, the
general fibre is a double solid of the type we are after:
if $\al1\neq0$ and $\al2\neq0$, we eliminate say $y_1$ and find
$\al1y_{2}^2-Qy_{0}y_{2}+\al2KLy_{0}^{2}=0$, which is a  double
cover of $\P^3$, branched along the quartic with equation
\[
Q^{2}-4\al1\al2KL=0\;.
\]

\subsection{The family $\cy \to\Pi$}
Our first task is to define a family, containing $Y$ (as a 
component of a fibre), with equations of the form \eqref{ydef}. 
The  deformation $\wt\cz\to\Pi$ of $\wt Z$ will
be constructed from it by birational transformations.

We start more generally from equations  
\begin{equation}\label{yphi}
\begin{split}
y_{1}y_{2}-\Phi y_{0}^{2}&=0\;,\\
\al2y_{1}+\al1y_{2}-Qy_{0}&=0\;,
\end{split}
\end{equation}
with $\Phi\in H^0(\P^3,\sier {\P^3}(4))$.
These equations define a family of complete intersections.
To have an explicit description of the base, we want to consider the 
coefficients of $\Phi$ as parameters. We therefore look 
at Equations \eqref{yphi} as a family 
over $\C^2 \times H^0(\P^3,\sier {\P^3}(4))$. 
In order to get a family involving $V=\P(H^0(\P^3,\sier {\P^3}(4))^*)$, 
we have to be
careful. We shall not end up with a family over $\C^2 \times V$, but
over a non-trivial vector bundle $\cv$ over $V$. 
The family of $\P^{2}$-bundles, in which Equations \eqref{yphi} make
sense, is obtained as the quotient of  
$\left(\C^{4}\setminus 0\right) \times \left(\C^{3}\setminus 0\right) \times
\C^2 \times \left(H^0(\sier {\P^3}(4))\setminus 0\right)$ under the
$\left(\C^{\ast}\right)^{3}$-action given by
\begin{multline*}\qquad
(a,b,\lambda)\cdot(x_{0},x_{1},x_{2},x_{3}\semic y_{0},y_{1},y_{2}\semic
\alpha_{1},\alpha_{2}\semic \Phi) 
\\{}=
(ax_{0},ax_{1},ax_{2},ax_{3}\semic by_{0},\lambda ba^{2}y_{1},ba^{2}y_{2}\semic
\alpha_{1},\lambda^{-1}\alpha_{2}\semic \lambda \Phi)\;.
\qquad
\end{multline*} 
In this context, the equation $y_{1}y_{2} - \Phi y_{0}^{2}=0$ at the
point $(ax\semic y\semic \alpha\semic \Phi)$, for example, is to be read as
$y_{1}y_{2} - \Phi(ax) y_{0}^{2}=0$. With this in mind, it is easy to see
that the zero set of Equations (\ref{yphi}) is invariant under this
action; hence, they define a family of singular spaces over the quotient of
$\C^2 \times \left(H^0(\sier {\P^3}(4))\setminus 0\right)$  
under the $\C^{\ast}$-action given by
$\lambda\cdot(\alpha_{1},\alpha_{2}\semic \Phi) =
(\alpha_{1},\lambda^{-1}\alpha_{2}\semic \lambda \Phi)$.
This quotient is the vector bundle $\cv$ of rank two over $V$ whose
sheaf of sections is $\co\oplus\co(-1)$. Again, we may let
$\lambda\in\C^{\ast}$ act on $y_{2}$ and $\al1$, as opposed to $y_{1}$ and
$\al2$, but this produces an isomorphic family.

For $\Phi=KL$, we want to consider the coefficients
of $K$ and $L$ as coordinates.
Note that the multiplication map 
$\P(H^0(\sier{}(3))^*)\times\P(H^0(\sier{}(1))^*)
\lra \P(H^0(\sier{}(4))^*)$
is not an embedding, it factors as a Segre embedding followed by a linear
projection. 
On the open set where the first component is an irreducible cubic,
the multiplication map is an isomorphism onto its image. We denote
this image by $V_{3;1}\subset V=\P(H^0(\sier{}(4))^*)$, and
by  $\cv_{3;1}$, the restriction of the vector bundle $\cv$ to it.

On the open set $V_{2,1}\subset
\Im \left\{\P(H^0(\sier{}(2))^*)\times\P(H^0(\sier{}(1))^*)
\lra \P(H^0(\sier{}(3))^*)\right\}
$ where $K$ splits as the product of a linear form
and an irreducible quadric, the multiplication map 
$\P(H^0(\sier{}(3))^*)\times\P(H^0(\sier{}(1))^*)
\lra \P(H^0(\sier{}(4))^*)$
is a branched
covering onto its image. As the linear form $L$ plays a special
role in our construction, we do not define $V_{2,1;1}$ as this
image, but as the subset $V_{2,1}\times\P(H^0(\sier{}(1))^*) \subset
\P(H^0(\sier{}(3))^*)\times\P(H^0(\sier{}(1))^*)$,
and $\cv_{2,1;1}$ as the pull-back of the vector bundle $\cv$ to it.
Likewise, we define $V_{1^3;1}$ as the subset  
$V_{1^3}\times\P(H^0(\sier{}(1))^*)$, where
$V_{1^3}= \Im \left\{\Sym^3 \P(H^0(\sier{}(1))^*)
\lra \P(H^0(\sier{}(3))^*)\right\}$ is the locus of 
the products of three linear forms, with vector bundle
$\cv_{1^3;1}$ over it.

\begin{remark}
The fibre over a general point of $\cv_{3;1}$ has six 
double points, but the codimension of $\cv_{3;1}$ in
$\cv$ is 12. There are trivial deformations of the fibre,
which do not preserve the splitting $\Phi=KL$. To see this,
observe that the defining ideal is not changed if
we add a multiple of the second equation in \eqref{yphi} to the
first one.
We take an arbitrary quadratic form $M$, compute
\begin{multline*}\qquad
y_1y_2-\Phi y_0^2-My_0(\al2y_{1}+\al1y_{2}-Qy_{0})
\\{}=
(y_1-\al1My_0)(y_2-\al2My_0)-(\Phi-M(Q-2\al1\al2M)-\al1\al2M^2) y_0^2=0
\qquad
\end{multline*} 
and rewrite the second equation as 
\[
\al2(y_1-\al1My_0)+\al1(y_2-\al2My_0)-(Q-2\al1\al2M)y_0=0\;.
\]
For small $\al1\al2$, the quadric $Q-2\al1\al2M$ is still non-degenerate
and by a coordinate transformation, we get equations of the
type \eqref{yphi}.
The upshot of this computation is that for fixed $\Phi$, the 
family  
\begin{equation}\label{trivdef}
\begin{split}
y_{1}y_{2}-(\Phi\circ h-MQ-\al1\al2M^2)y_{0}^{2}&=0\\
\al2y_{1}+\al1y_{2}-Qy_{0}&=0\;
\end{split}
\end{equation}
is trivial, where $h$ is the coordinate transformation 
depending on the product $\al1\al2$, satisfying
$Q \circ h= Q+2\al1\al2M$.
The family lies in $\cv_{3;1}$ if $M$ is divisible by $L\circ h$.
\end{remark}

In order to stratify according to singularities, we locate
them in the fibres  of the family \eqref{yphi}.

\begin{proposition}\label{prop:sing}
Above $\al1\al2=0$, in $y_0=0$, lie non-isolated singularities, which are
the intersections of the irreducible components 
of the fibres of $\cy \to \Pi$. The other singularities
above $\Lambda\colon \al1=\al2=0$ lie in $y_1=y_2=0$.
The isolated singularities above $\al1=0$ and $\al2=0$ are isomorphic to those
above $\al1=\al2=0$.
\end{proposition}

\begin{proof}
The condition for a singular point  is
\begin{equation}\label{sing}
\Rank \begin{pmatrix} y_2 & y_1 & -2y_0 \Phi & -y_0^2 \D\Phi \\
                 \al2 & \al1 & -Q & -y_0\D Q\end{pmatrix}  \leq 1\;.
\end{equation}
We first consider a neighbourhood of $y_0=0$.
We look at the chart $y_2=1$. The first equation
of \eqref{yphi} gives $y_1=\Phi y_0^2$. Equations \eqref{sing}
then reduce to  $\al1=
\al2 \Phi y_0^2$,  $Q= 2\al2\Phi y_0$ and 
$y_0(\D Q- \al2 y_0 \D \Phi)=0$. 
If $y_0=0$, then $\al1=0$ and $y_1=Q=0$. Indeed, for $\al1=0$,
the section $y_0=y_1=0$ of the 
$\P^2$-bundle $\P(\co\oplus \co(-2)\oplus \co(-2))\rightarrow \P^{3}$
is a component of the fibre, and the intersection
of the components is part of the singular locus. 
Likewise, for $\al2=0$,
the section $y_0=y_2=0$ is a component. 

The singularities
outside $y_0=0$ satisfy $\D Q- \al2 y_0 \D \Phi=0$ (in the chart
$y_2=1$).
As $Q$ is non-degenerate, none of these tend to $y_0=0$ if
$\al1$ tends to zero.
In fact, for  $\al1=\al2=0$, the singular
points with $y_0\neq0$ 
lie  in $y_1=y_2=0$ on $Q=\Phi=0$ and are given by the
condition that the differential $\D \Phi$ is proportional to $\D Q$. 
So, they lie above the singular points of the curve $\Phi=0$ on $Q\subset
\P^3$.

For $\al1=0$, $\al2\neq0$, we find that $y_1=0$, as $\D Q\neq0$.
So, the  singular points still satisfy $Q=\Phi=0$ and lie
above the singular points of the curve $\Phi=0$ on $Q\subset
\P^3$. The singularities are isomorphic to those over $\al1=\al2=0$.
This is best seen by a computation in local coordinates. We set
$y_0=1$. As $Q$ is non-degenerate, we can take $Q$ as coordinate
$x_0$. We write $\Phi(x_0,x_1,x_2)=\vp(x_1,x_2)+x_0\psi(x_0,x_1,x_2)$.
Equations \eqref{yphi} become $y_1y_2-\vp-x_0\psi=0$,
$\al2y_1-x_0=0$, from which $x_0$ can be eliminated:
\[
y_1(y_2-\al2\psi(\al2y_1,x_1,x_2))-\vp(x_1,x_2)=0\;.
\]
\end{proof}

The singularities outside $y_0=0$ for $\al1\al2=0$ 
are $cA_k$ points (not necessarily isolated).
Such singularities can be characterised as having 
corank at most two.
An isolated $cA_k$ singularity
has the form $w_1w_2-g(x,y)=0$, with $g(x,y)=0$ a 
reduced curve singularity, whose $\delta$-invariant is
an invariant of the $cA_k$ singularity.

We stratify the parameter spaces  $\cv_{3;1}$,
$\cv_{2,1;1}$ and $\cv_{1^3;1}$.
One stratum contains the non-isolated singularities
and the singularities of corank $3$. The
complementary open sets are stratified
according to the sum of the
$\delta$-invariants over all 
isolated $cA_k$  singularities outside $y_0=0$.
We consider only the open sets of connected components of strata,
where all isolated singularities admit small
resolutions, or equivalently that the corresponding plane
curve singularities have smooth branches.

Let $\wl \Pi$ be such an open set of a stratum in  $\cv_{3;1}$.
It can be defined locally at a point $p\in \wl \Pi$
using the induced map
of germs $\kappa_{3;1}\colon (\cv_{3;1}, p)
\to \prod \Def(\cy_{p}, x_{i})$, where the $x_{i}$ are the isolated
singularities of the fibre $\cy_{p}$. 
Let $S^\delta=\prod S^\delta_i$ be the product of the
$\delta$-constant strata in the deformation spaces
$\Def(\cy_{p}, x_{i})$. Then, $\wl \Pi=\kappa_{3;1}^{-1}(S^\delta)$.
If it is possible to choose $K$ reducible, we consider open sets of
strata in  $\cv_{2,1;1}$ or $\cv_{1^3;1}$, definable by
$\kappa_{2,1;1}$ or $\kappa_{1^3;1}$, respectively.

As  in Section \ref{sec:KL}
for \KLB manifolds, for each fibre in the family over a 
maximal open set $\wl \Pi$  in a stratum 
in  $\cv_{3;1}$, $\cv_{2,1;1}$ or $\cv_{1^3;1}$, 
we consider all possible small resolutions of the isolated singularities. 
Again, these fit together to a covering $\Pi  \to \wl \Pi$ with finite
fibres.  
Therefore, we have a simultaneous
small resolution of all isolated
singularities in  a family  
$\cy \to \Pi$ of complete intersections in
our $\P^2$-bundle over $\P^{3}$.

Let $\wt Z$ be an 
\KLB manifold, defined by an equation of the form
(\ref{klb}). The given choice of small resolutions defines 
a point in a stratum $\Lambda_0$, which is a covering of
a stratum $\wl\Lambda_0$ in $\P(H^0(\sier Q(3,3))^*)$. 

To get to Equations (\ref{yyy}), we
note that  restriction to
$Q$ defines a rational map
\[\P(H^0(\P^3,\sier{}(3))^*)\dashrightarrow \P(H^0(\sier Q(3,3))^*)\;,\]
which is not defined along the projective subspace 
$\P((Q\cdot H^{0}(\sier{}(1)))^{*})$, because $Q\cdot H^{0}(\sier{}(1))$ is
the kernel of the surjective restriction map $H^{0}(\sier{}(3)) \lra
H^{0}(\sier Q(3,3))$. 
This also implies that the fibres of the morphism 
$\P(H^0(\sier{}(3))^*) \setminus \P((Q\cdot H^{0}(\sier{}(1)))^{*})
\lra \P(H^0(\sier Q(3,3))^*)$ are isomorphic to $H^{0}(\sier{}(1))$.
We are forced to enlarge our family again, but we avoid this as
much as possible by considering three cases according to the way
$\vp\in H^0(\sier Q(3,3))$ splits.

Given $\Lambda_0$, 
we first consider a stratum 
$\wl\Lambda_{1}\subset\P(H^{0}(\sier{}(3))^{*})$, 
which maps surjectively onto
$\wl\Lambda_{0}$ under the restriction map
$\P(H^0(\sier{}(3))^*) \setminus \P((Q\cdot H^{0}(\sier{}(1)))^{*})
\lra \P(H^0(\sier Q(3,3))^*)$. 
In general, $\wl\Lambda_{1}$ is
the inverse image of $\wl\Lambda_{0}$, 
but if all members of $\wl\Lambda_{0}$
split as $\varphi_{1}\varphi_{2}\varphi_{3}$ or as $\varphi_{1}\psi$ with
$\varphi_{i} \in H^{0}(\sier Q(1,1)), \psi\in H^{0}(\sier Q(2,2))$,
then we restrict the inverse image to $V_{2,1}$, resp.~$V_{1^3}$.
In the case of three linear forms,  which occurs in the twistor case, 
we  have $\wl\Lambda_{1} \cong \wl\Lambda_{0}$. 

Now we add the linear form $L$ defining
a tangent plane to $Q$. 
We get  a stratum 
$\wl\Lambda_1\times Q^* \subset 
\P(H^0(\sier{}(3))^*)\times\P(H^0(\sier{}(1))^*)$ over
$\wl\Lambda_0\times Q^* \subset 
\P(H^0(\sier{Q}(3,3))^*)\times\P(H^0(\sier{Q}(1,1))^*)$. 
We require that the restriction of $KL$ to
$Q$ has only isolated singularities: we have to exclude that 
the tangent plane $L$ and the curve $K=Q=0$
both contain the same line. This can only happen if a curve 
$D$ in the stratum $\wl\Lambda_0$ (and therefore each curve $D$)
contains a line of one
of the rulings. In this case, we replace 
$\wl\Lambda_1\times Q^*$ by its open subset which is obtained by removing 
those pairs $(K, L)$ for which the tangent plane $L$ and the curve $K=Q=0$
both contain the same line. In all cases, we denote the resulting 
stratum by $\wl\Lambda$. According to our choices
it lies entirely  in  $V_{3;1}$,
$V_{2,1;1}$ or $V_{1^3;1}$.

We now go over to Equations  (\ref{ydef}), where we have the deformation
parameters $\al1$ and $\al2$. 
We identify each stratum $\wl\Lambda\subset
V_{3;1}$ (resp.\ $V_{2,1;1}$ or $V_{1^3;1}$) with its image in the zero 
section of $\cv_{3;1}$ (resp.\ $\cv_{2,1;1}$ or $\cv_{1^3;1}$).
The fibres over $\wl \Lambda$ are now reducible, consisting of $Y$ and the
sections $\ce_i\colon  y_0=y_i=0$ of the  $\P^2$-bundle.
The stratum $\wl \Lambda$ lies 
in a unique stratum $\wl \Pi$, and
inside $\wl\Pi$, it is given by the equations 
$\alpha_1=\alpha_2=0$.

The covering $\Lambda_0\to\wl\Lambda_0$  induces  a covering 
$\Lambda_1\to\wl\Lambda_1$ from which we obtain a covering
$\Lambda\to\wl\Lambda$. The linear form $L$ introduces new singularities,
which are to be resolved in a specific way, governed by the
birational isomorphism  between $Y$ and $W$ (Lemma \ref{lem:birational}).
We describe it in detail in Remark \ref{rmk-explicit}.
The resolution can be done simultaneously over the whole stratum $\Lambda$.
Therefore, we pick out a specific small resolution, and we get
an embedding $\Lambda \subset \Pi$. 

\begin{lemma}\label{lem:smooth}
The base space $\Pi$ is  smooth in a neighbourhood of
$\Lambda$. 
\end{lemma}

Before embarking on the proof in full generality, we treat the 
simpler case that all isolated singularities are ordinary
double points. As the question is local, it suffices to prove that 
$\wl \Pi$ is smooth at the points of $\wl\Lambda$. 
Of the three cases to be considered, we 
take the one relevant for twistor spaces.
We therefore suppose that the stratum $\wl \Pi$ lies in $\cv_{1^3;1}$.
We fix a  point $p\in\wl\Lambda\subset\wl\Pi$. Locally 
around $p$, the stratum is the fibre over $0$ of the map
$\kappa_{1^3;1}\colon (\cv_{1^3;1}, p)
\to \prod \Def(\cy_{p}, x_{i})$, as the $\delta$-const stratum
of an ordinary double point consists of the origin only. The map
$\kappa_{1^3;1}$ is not surjective, as all fibres over $\cv_{1^3;1}$
with only ordinary double points have at least $12$ such points: 
they lie above  the intersection of $Q$ with the edges
of the tetrahedron $L_1L_2L_3L_4$. We therefore divide the set $\Sigma$ 
of isolated singular points into two, the first, $\Sigma_1$, consisting
of these $12$ points, while the remaining $\delta-5$ double points
make up $\Sigma_2$ (here $\delta$ is the number of double points 
of the \KLB manifold we started with). The image of
$\cv_{1^3;1}$ is contained in $0\times\prod_{\Sigma_2} \Def(\cy_{p}, x_{i})
\subset \prod_{\Sigma_1} \Def(\cy_{p}, x_{i})
\times\prod_{\Sigma_2} \Def(\cy_{p}, x_{i})$. To prove smoothness
of $\wl \Pi$ at $p$, we establish that 
$\kappa_{1^3;1}\colon (\cv_{1^3;1}, p)
\to 0\times\prod_{\Sigma_2} \Def(\cy_{p}, x_{i})$ is 
a submersion. In fact, this already holds for the restriction
of $\kappa_{1^3;1}$ to the zero section $(V_{1^3;1},p)$
of $(\cv_{1^3;1},p)$. 
We factor this restriction through $(W_{1^3;1},[p])\subset 
(\P(H^0(\sier{Q}(3,3))^*)\times\P(H^0(\sier{Q}(1,1))^*),[p])$,
where $W_{1^3;1}$ is defined analogously to $V_{1^3;1}$, and
$[p]$ is the image of $p$. The map  $W_{1^3;1}\to \P(H^0(\sier{Q}(4,4))^*) $
is a branched covering onto its image.
But ramification occurs exactly when $L$ coincides with one
of the factors of $K$. By our assumption on isolatedness of the
newly introduced singularities, this does not occur at the point $[p]$.
We can therefore identify the germ with its image germ.
We get the following diagram:
\[
\xymatrix 
{
  (\wl\Lambda_1\times Q^*,p) \ar[r]\ar[d]& (\wl\Lambda_0\times Q^*,[p]) 
    \ar[r]\ar[d] & 0 \ar[d] \\
(V_{1^3;1},p) \ar[r] & (W_{1^3;1},[p]) \ar[r]\ar[d]  & 
0\times\prod_{\Sigma_2} \Def(\cy_{p}, x_{i}) \ar[d] \\
  & (\P(H^0(\sier{Q}(4,4))^*),[p])\ar[r] & 
 \prod_{\Sigma_1} \Def(\cy_{p}, x_{i})
\times\prod_{\Sigma_2}\Def(\cy_{p}, x_{i})\makebox[0pt][l]{$\;.$}
}\hphantom{\;}
\]
Note that we do not put $(V,p)=(\P(H^0(\sier{}(3))^*),p)$
in the lower left corner of the diagram, as $(V_{1^3;1},p)$ is not the
inverse image of $ 0\times\prod_{\Sigma_2} \Def(\cy_{p}, x_{i})$ in
$(V,p)$: the trivial deformations of type \eqref{trivdef} do
not respect the splitting of $\Phi$.

The two horizontal maps on the left are submersions. 
According to Lemma \ref{lem} and the remarks following it,
the stratum 
$\wl\Lambda_0\times Q^*$ is smooth at $[p]$, of codimension $\delta+7
=(\delta-5)+12$
in $(\P(H^0(\sier{Q}(4,4))^*),[p])$. As it is the inverse image of
$0 \in  \prod_{\Sigma_1} \Def(\cy_{p}, x_{i})
\times\prod_{\Sigma_2}\Def(\cy_{p}, x_{i})$, the map
$(\P(H^0(\sier{Q}(4,4))^*),[p])\lra
 \prod_{\Sigma_1} \Def(\cy_{p}, x_{i})
\times\prod_{\Sigma_2}\Def(\cy_{p}, x_{i})$ is a submersion.
The dimension of $\P(H^0(\sier{Q}(4,4))^*)$ is 24 and that of the
smooth space $W_{1^3;1}$ is 12, so the codimension of 
$(W_{1^3;1}[p])$ in $ (\P(H^0(\sier{Q}(4,4))^*),[p])$ is also 12.
We conclude that $(W_{1^3;1},[p])  \to  
0\times\prod_{\Sigma_2} \Def(\cy_{p}, x_{i})$ is a submersion
and therefore also $(V_{1^3;1},p)  \to  
0\times\prod_{\Sigma_2} \Def(\cy_{p}, x_{i})$. 

The proof in the general case is based on the same ideas.
To describe the subspace in which the image of 
$(V_{3;1},p)$, $(V_{2,1;1},p)$ or
$(V_{1^3;1},p)$ lands,
we need some more facts about deformations of plane curve
singularities.
For instance, in $V_{3;1}$, we change $K$ and $L$ separately.
This can also be done for components of reducible curve singularities.
We borrow the term \textit{equi-intersectional} for the resulting
stratum from \cite{GLS}; the theory can be modelled on the
careful treatment of the $\delta$-const stratum in \cite{GLS}.

We consider a curve singularity $C=\bigcup_{i=1}^k C_i$, where the
$C_i$ might themselves be reducible. We look at the deformation
theory of the map $\amalg C_i\to C$.

\begin{defn}
The equi-intersectional stratum $S^{\rm ei}$ of the curve
$C=\bigcup_{i=1}^k C_i$ is the image of the deformation space
of the map $\amalg C_i\to C$ in the deformation space of $C$.
\end{defn}

\begin{lemma}
The equi-intersectional stratum $S^{\rm ei}$ is smooth of codimension
$\sum_{i<j} (C_i\cdot C_j)$.
\end{lemma}

\begin{proof}
One first shows that the deformation functors of the map
$\amalg C_i\to C$ and that of $\amalg C_i\to \C^2$ are 
isomorphic, with the reasoning of the proof \cite[Prop.~II.2.23]{GLS};
it uses \cite[Prop.~II.2.9]{GLS}, which is formulated in great generality
and also applies to our situation. The vector space $T^2$
for the second deformation problem vanishes (this is not true for
the first problem) by computations similar to those in 
\cite[Sect.~II.2.4]{GLS}. Therefore, the deformation space of
$\amalg C_i\to C$ is smooth.

Let $C=\bigcup_{i=1}^k C_i$ be given by $f_1\cdots f_k=0$.
The image in $\C^2$  of a deformation of $\amalg C_i\to C$
is given by a product $F_1\cdots F_k=0$. If this is a
trivial deformation of $C$, then necessarily, the 
$C_i$ are deformed trivially and also the map. Therefore, 
the natural forgetful map from the deformation space of
$\amalg C_i\to C$
to the deformation space of $C$
is an immersion. 

To compute the codimension, it suffices  by openness of
versality
to look at a general point of the stratum, for which the
deformed curves $C_i$ intersect transversally. The number of
intersection points is $\sum_{i<j} (C_i\cdot C_j)$.
\end{proof}

\begin{proof}[Proof of Lemma  \ref{lem:smooth}]
As the question is local, it suffices to prove that 
$\wl \Pi$ is smooth at the points of $\wl\Lambda$. 
We fix a  point $p\in\wl\Lambda\subset\wl\Pi$.
We prove that $\wl \Pi$ is smooth at $p$ by realising it
as the inverse image of the $\delta$-const stratum $S^\delta$
under a submersion. Let $S^{\rm ei}$ be the product of the
equi-intersectional strata of the isolated singularities
$(\cy_{p}, x_{i})$.

A complication is that in general the map from $(\cv,p)
\to \prod\Def(\cy_{p}, x_{i})$ is not surjective.
Nevertheless, by Lemma \ref{lem} the stratum $\wl\Lambda_0
\times Q^*$ is smooth of codimension $\delta+7$ 
in $\P(H^0(\sier{Q}(4,4))^*)$. We therefore map down to a
smooth space with tangent space isomorphic to 
$\prod T^1_{(\cy_{p}, x_{i})} / TS^\delta$.
As the strata $S^\delta \subset S^{\rm ei}$ are smooth,
we can choose a transversal slice $S^\perp$ and a 
projection $\sigma \colon \prod\Def(\cy_{p}, x_{i})\to S^\perp$
with $S^\delta = \sigma^{-1}(0)$ and 
$\sigma^{-1}(\sigma(S^{\rm ei}))=S^{\rm ei}$.

We  treat the three cases, that $\wl \Pi$ lies entirely in 
$\cv_{3;1}$, $\cv_{2,1;1}$ or $\cv_{1^3;1}$, 
at the same time by writing $\cv_{\alpha;1}$ with
$\alpha$ standing for a partition. The stratum $S^{\rm ei}$
depends on the chosen partition.

Smoothness
of $\wl \Pi$ at $p$ follows when we show  that 
$\sigma\circ\kappa_{\alpha;1}\colon (\cv_{\alpha;1}, p)
\to \sigma(S^{\rm ei}) $ is 
a submersion. In fact, this already holds for the restriction
of $\sigma\circ\kappa_{\alpha;1}$ to the zero section $(V_{\alpha;1},p)$
of $(\cv_{\alpha;1},p)$. 
We factor this map through $(W_{\alpha;1},[p])\subset 
(\P(H^0(\sier{Q}(3,3))^*)\times\P(H^0(\sier{Q}(1,1))^*),[p])$,
where $W_{\alpha;1}$ is defined analogously to $V_{\alpha;1}$, and
$[p]$ is the image of $p$. 
In the two cases $\alpha=(2,1)$ and $\alpha=(1^3)$,
the map  $W_{\alpha;1}\to \P(H^0(\sier{Q}(4,4))^*) $ is 
a branched covering onto its image.
But ramification does not  occur at the point $[p]$.
We can therefore identify the germ with its image germ.

In this way, we get immersions 
$(\wl\Lambda_0\times Q^*,[p]) \subset
(W_{\alpha;1},[p]) \subset (\P(H^0(\sier{Q}(4,4))^*),[p])$. 
According to Lemma \ref{lem} and the remarks following it,
the stratum 
$\wl\Lambda_0\times Q^*$ is smooth at $[p]$, of codimension $\delta+7$
in $(\P(H^0(\sier{Q}(4,4))^*),[p])$. 
As this stratum is the inverse image of
$0 \in S^\perp$, the map
$(\P(H^0(\sier{Q}(4,4))^*),[p])\lra (S^\perp,0)$ is a submersion.
The dimension of $\P(H^0(\sier{Q}(4,4))^*)$ is 24.
The dimensions of the
smooth spaces $W_{3;1}$, $W_{2,1;1}$ and  $W_{1^3;1}$ 
are 18, 14 and 12, respectively, so the codimensions of the
$(W_{\alpha;1}[p])$ in $ (\P(H^0(\sier{Q}(4,4))^*),[p])$ are
6, 10 and 12, respectively. These numbers are the respective global
intersection numbers of the components of $KL$, so they coincide with the
codimensions of $S^{\rm ei}$ in $\prod\Def(\cy_{p}, x_{i})$ and
therefore also of $\sigma(S^{\rm ei})$ in $S^\perp$.  
We conclude that $(W_{\alpha;1},[p])  \to  
\sigma(S^{\rm ei})$ is a submersion
and therefore also $(V_{\alpha;1},p)  \to  
\sigma(S^{\rm ei})$.
\end{proof}

\begin{example}\label{ex:base}
We can give a more explicit description of $\wl\Pi$, when this stratum is 
given by only one equation in $\cv_{\alpha;1}$.
This is the case if $K$ lies in the open stratum  inside the
space of forms of given splitting type. Our aim is to describe $L$
depending on $K$ and $\al1\al2$.  We use homogeneous coordinates
on $\P(H^0(\sier{}(3))^*)\times\P(H^0(\sier{}(1))^*)$. On the total space, we
have the
$\left(\C^{\ast}\right)^{4}$-action given by
\begin{multline*}\qquad
(a,b,\lambda,\mu)\cdot(x_{0},x_{1},x_{2},x_{3}\semic y_{0},y_{1},y_{2}\semic
\alpha_{1},\alpha_{2}\semic K, L) 
\\{}=
(ax_{0},ax_{1},ax_{2},ax_{3}
   \semic by_{0},\lambda ba^{2}y_{1},\mu ba^{2}y_{2}\semic
\mu^{-1}\alpha_{1},\lambda^{-1}\alpha_{2}\semic \lambda K, \mu L)\;.
\qquad
\end{multline*} 
In particular, the equation $Q^2-4\al1\al2LK$ is invariant under the
$(\la,\mu)$-action. We compute the condition that this quartic has a singular
point $P$ outside the quadric $Q$. We find the condition
\[
2Q(P)\,\D Q(P) - 4\al1\al2L(P)\,\D K(P) -4K(P)\,\D (\al1\al2L) = 0\;,
\]
so
\[
\D (\al1\al2L) = \frac{Q(P)}{2K(P)}\D Q(P) -\frac{Q^2(P)}{4K^2(P)}\D K(P)\;.
\]
Note that the right-hand side of this equation is homogeneous of degree 0
in $P$. This shows that $\al1\al2L$ is determined by the position 
of the point $P$. As the formula does not work if $P$ lies on $\{K=0\}$,
we blow up $\P^3$ in the intersection of $Q$ and $K$ and look at a
neighbourhood of the strict transform of the quadric $Q$. 
We realise the blow-up as subset of a $\P^1$-bundle over $\P^3$.
The singular points  lie above the singular points of $Q\cap K$, but away
from the strict transform  $\wt Q$ 
of the quadric. Note that the normal bundle
of $\wt Q$ is of type $(-1,-1)$, as we have blown up a curve of type
$(3,3)$ on $Q$. It is natural to take the product $\al1\al2$ 
as inhomogeneous fibre coordinate, so that the blow-up is given by
\begin{equation}\label{eq:point}
Q -2\al1\al2K=0\;.
\end{equation}
But we rather consider this equation in a rank $2$ vector bundle
over $\P^3$, described by the action
\[
\mu\cdot(x_{i}\semic \alpha_{1},\alpha_{2}) 
=(\mu x_{i} \semic \mu^{-1}\alpha_{1},\alpha_{2})\;.
\]
Now we have all the ingredients to describe $L$ in Equations
\eqref{ydef}. As a linear form is determined by its differential,
we can define $L$ by
\[
\D L = \D Q(P) -\al1\al2\D K(P)\;.
\]
We consider the coefficients of $K$, $\al1$, $\al2$ and the
point $P$ (related by Equation \eqref{eq:point}) as 
coordinates.
In this way, we obtain the normalisation of the stratum $\wl\Pi$
(or rather of a closure).
We have an isomorphism in a neighbourhood of $\al1=\al2=0$, but, in general, 
several singular points can lie on one and the same quartic branch
surface.
We will see this phenomenon again in Example \ref{nonumber}.
\end{example}

\subsection{The family $\wt\cz \to\Pi$}
Having constructed the family $\cy\to\Pi$, we use birational transformations
of the total space to obtain a family, which is
a deformation of \KLB manifolds. We also give
divisors $\wt\cs_{1}$, $\wt\cs_{2}$ in $\wt\cz$.

To see what has to be done, we first describe how the rational map $W
\dashrightarrow Y$ of Lemma \ref{lem:birational} can be factored.
To be consistent with later notation, we call $W$ for $Y^+$  and factor the
rational map $Y^+ \dashrightarrow Y$ as composition 
$Y^+ \longleftarrow Y^- \lra Y$ of a blow-up and a small contraction. In the
process, all objects on $W$ (as defined in previous sections)
pick up a plus sign as upper index. 
On $W=Y^+$, we therefore have the divisors 
\[
E_i^+:\quad w_0=w_i=0
\qquad\text{ and }\qquad 
R_i^+:\quad   \psi_i=0\;.
\]
The rational  map is not defined on the following two disjoint curves:
\begin{eqnarray*}
    B^+_{1}=R^+_1\cap E^+_1&:&\quad \psi_{1}=0\;,\;\; w_{0}=w_{1}=0\;,\\
    B^+_{2}=R^+_2\cap E^+_2&:&\quad \psi_{2}=0\;,\;\; w_{0}=w_{2}=0\;.
\end{eqnarray*}
To describe the blow-up of $B^+_1$, we first describe its blow-up in the
ambient $\P^2$-bundle over $Q$. Inside the fibred product
\[
\P(\sier Q \oplus \sier Q(-2,-1) \oplus \sier Q(-1,-2))
  \times_Q
  \P(\sier Q \oplus \sier Q(-2,-1) \oplus \sier Q(-1,0))
\]
with  fibre coordinates $(w_{0}:w_{1}:w_{2})$ and
$(z_{0}:z_{1}:z_{2})$ this blow-up  can be described by the equations
\[
\Rank \begin{pmatrix}w_0& w_1 & \psi_1 \\ z_0& z_1 &z_2\end{pmatrix}  \leq 1\;.
\]
The strict transform $Y^-$  of $Y^+$ is given by the five maximal
minors of the matrix
\[
\begin{pmatrix}w_0& w_1 & \psi_1 \cr z_0& z_1 &z_2 \\ 
w_2 & \vp w_0\end{pmatrix} \;.
\]
The exceptional divisor is a ruled surface $F^-_1$, given by
$\psi_{1}=w_{0}=w_{1}=z_1=0$. We can extend our rational map by
\[
(y_{0}:y_{1}:y_{2}) = (z_{0}:\psi_{2}z_{1}:z_{2}w_2)\;.
\]
It is everywhere defined in a neighbourhood of $F^-_1$. The blow-up
of $B_2^+$ introduces a ruled surface $F^-_2 \subset Y^{-}$.

The morphism $Y^- \to Y$ blows down several curves, which lie in the
strict transforms $R_1^-$ and  $R_2^-$ of $R_1^+$ and  $R_2^+$. This
fact is most easily seen  from the given formula for the rational
map $Y^+ \dashrightarrow Y$. On $Y^+$, these curves form the fibre
over the point $P$ in which $L$ is the tangent plane to $Q$, given
by $R^+_1\cap R^+_2\colon \psi_1=\psi_2=0$, and the lines
$\psi_1=w_1=\vp=0$, $\psi_2=w_2=\vp=0$. If $P$ is in general
position, exactly seven lines are blown down, introducing seven new 
$A_1$ singularities. Otherwise the given equations define the blown
down curves  with possibly non-reduced structure and we get less but higher
singularities. Mostly this will not influence our constructions, but
it is important to distinguish whether or not the point $P$ lies on
the discriminant curve $\vp=0$.

Figure \ref{figeen} shows, in the general case, the divisors involved
and the exceptional curves.
\begin{figure}\centering
  \begin{tikzpicture}
    \newcommand{\thick}{1.0pt}
    \newcommand{\Mitte}{1.65}
    \newcommand{\thin}{0.4pt}
    \newsavebox{\Yplus}
    \savebox{\Yplus}{
      \tkzDefPoint(2.5,2.3){X}
      \tkzDefPoint(5.0,2.9){A}
      \tkzDefPoint(0,2.9){B}
      \tkzDefPoint(2.5,-4.3125){O1}
      \tkzDrawSegments[line width = \thin](X,A X,B)
      \tkzDrawArc(O1,A)(B)
      \tkzDefPoint(2.5,1.0){Y}  
      \tkzDefPoint(0,0.4){C}
      \tkzDefPoint(5.0,0.4){D}
      \tkzDefPoint(2.5,7.6125){O2}
      \tkzDrawSegments[line width = \thin](Y,C Y,D)
      \tkzDrawArc(O2,C)(D)
      \tkzDrawSegments[line width = \thin](A,D B,C)
      \draw [line width = \thick] (X) -- (Y);
      \tkzDefPoints{%
        1.0/\Mitte/E11,
        1.5/\Mitte/E12,
        2.0/\Mitte/E13,
        3.0/\Mitte/E21,
        3.5/\Mitte/E22,
        4.0/\Mitte/E23}
      \tkzDefMidPoint(B,X)  \tkzGetPoint{F12}
      \tkzDefMidPoint(B,F12)\tkzGetPoint{F11}
      \tkzDefMidPoint(F12,X)\tkzGetPoint{F13}
      \tkzDefMidPoint(Y,D)  \tkzGetPoint{F22}
      \tkzDefMidPoint(Y,F22)\tkzGetPoint{F21}
      \tkzDefMidPoint(F22,D)\tkzGetPoint{F23}
      \draw [line width = \thick] (E11) -- (F11) (E12) -- (F12) (E13) -- (F13);
      \draw [line width = \thick] (E21) -- (F21) (E22) -- (F22) (E23) -- (F23);
      \node at (2.5,2.8) {$E_{1}^{+}$};
      \node at (2.5,0.5) {$E_{2}^{+}$};
      \node at (1.0,1.2) {$R_{1}^{+}$};
      \node at (4.0,2.15){$R_{2}^{+}$};
    }
    \newsavebox{\Y}
    \savebox{\Y}{
      \tkzDefPoint(2.5,2.3){X}
      \tkzDefPoint(5.0,2.9){A}
      \tkzDefPoint(0,2.9){B}
      \tkzDefPoint(2.5,-4.3125){O1}
      \tkzDrawSegments[line width = \thin](X,A X,B)
      \tkzDrawArc(O1,A)(B)
      \tkzDefPoint(2.5,1.0){Y}  
      \tkzDefPoint(0,0.4){C}
      \tkzDefPoint(5.0,0.4){D}
      \tkzDefPoint(2.5,7.6125){O2}
      \tkzDrawSegments[line width = \thin](Y,C Y,D)
      \tkzDrawArc(O2,C)(D)
      \tkzDefPoints{%
      0.3/\Mitte/M,
      4.7/\Mitte/N}
      \tkzDrawSegments[line width = \thin](M,B M,C N,A N,D M,N X,Y)
      \tkzDefPoints{%
        2.5/\Mitte/E0,
        1.0/\Mitte/E11,
        1.5/\Mitte/E12,
        2.0/\Mitte/E13,
        3.0/\Mitte/E21,
        3.5/\Mitte/E22,
        4.0/\Mitte/E23}
      \tkzDrawPoints[size=4pt](E0,E11,E12,E13,E21,E22,E23)
      \node at (2.5,2.8) {$E_{1}$};
      \node at (2.5,0.5) {$E_{2}$};
      \node at (1.0,1.2) {$R_{1}$};
      \node at (4.0,2.15){$R_{2}$};
      \node at (1.0,2.15){$F_{1}$};
      \node at (4.0,1.2) {$F_{2}$};
    }
    \newsavebox{\Yminus}
    \savebox{\Yminus}{
      \tkzDefPoint(2.5,2.9){X}
      \tkzDefPoint(5.0,3.5){A}
      \tkzDefPoint(0,3.5){B}
      \tkzDefPoint(2.5,-3.7125){O1}
      \tkzDrawSegments[line width = \thin](X,A X,B)
      \tkzDrawArc(O1,A)(B)
      \tkzDefPoint(2.5,1.0){Y}  
      \tkzDefPoint(0,0.4){C}
      \tkzDefPoint(5.0,0.4){D}
      \tkzDefPoint(2.5,7.6125){O2}
      \tkzDrawSegments[line width = \thin](Y,C Y,D)
      \tkzDrawArc(O2,C)(D)
      \tkzDefPoints{%
      0.3/2.3/M,
      4.7/1.6/N,
      2.35/2.3/P,
      2.65/1.6/Q}
      \tkzDrawSegments[line width = \thin](M,B M,C N,A N,D M,P N,Q X,P Y,Q)
      \draw [line width = \thick] (P) -- (Q);
      \tkzDefPoints{%
        0.55/2.3/M1,
        4.45/1.6/N1,
        4.45/2.7/A1,
        0.55/1.2/C1}
      \tkzDefMidPoint(M1,P)  \tkzGetPoint{E11}
      \tkzDefMidPoint(M1,E11)\tkzGetPoint{E12}
      \tkzDefMidPoint(E11,P) \tkzGetPoint{E13}
      \tkzDefMidPoint(Q,N1)  \tkzGetPoint{E21}
      \tkzDefMidPoint(Q,E21) \tkzGetPoint{E22}
      \tkzDefMidPoint(E21,N1)\tkzGetPoint{E23}
      \tkzDefMidPoint(P,A1)  \tkzGetPoint{F21}
      \tkzDefMidPoint(P,F21) \tkzGetPoint{F22}
      \tkzDefMidPoint(F21,A1)\tkzGetPoint{F23}
      \tkzDefMidPoint(C1,Q)  \tkzGetPoint{F11}
      \tkzDefMidPoint(C1,F11)\tkzGetPoint{F12}
      \tkzDefMidPoint(F11,Q) \tkzGetPoint{F13}
      \draw [line width = \thick] (E11) -- (F11) (E12) -- (F12) (E13) -- (F13);
      \draw [line width = \thick] (E21) -- (F21) (E22) -- (F22) (E23) -- (F23);
      \node at (2.5,3.35) {$E_{1}^{-}$};
      \node at (2.5,0.5) {$E_{2}^{-}$};
      \node at (0.7,1.0) {$R_{1}^{-}$};
      \node at (4.3,2.9) {$R_{2}^{-}$};
      \node at (0.7,2.9) {$F_{1}^{-}$};
      \node at (4.3,1.0) {$F_{2}^{-}$};
    }
    \node at (0,0)     {\usebox{\Yplus}};
    \node at (7.0,0)   {\usebox{\Y}};
    \node at (3.5,5.1) {\usebox{\Yminus}};
    \node (T)  at (6.0,4.7) {$Y^{-}$};
    \node (BL) at (4.7,3.5) {$Y^{+}$};
    \node (BR) at (7.2,3.5) {$Y$};
    \draw[->,line width = 0.6pt] (T) -- (BL);
    \draw[->,line width = 0.6pt] (T) -- (BR);
    \pgfresetboundingbox
    \path[use as bounding box] (0,0) rectangle (12,9);
  \end{tikzpicture}
\caption{Factorisation of the map $Y^+ \dashrightarrow Y$.}\label{figeen}
\end{figure}
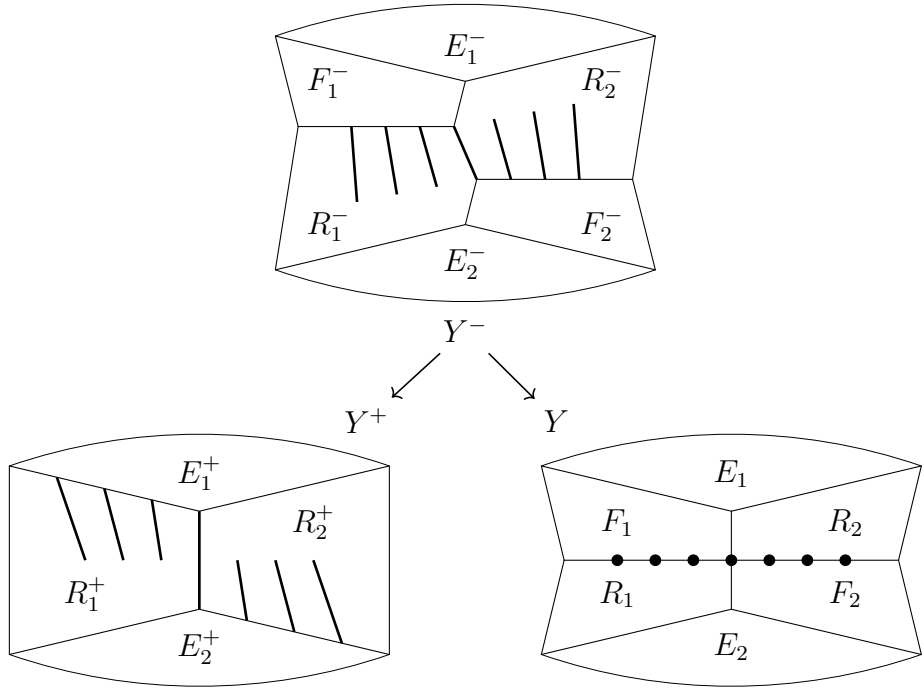
In the case where the point $P$ lies on the discriminant, both curves in
the fibre are contracted to a $cA_2$ singularity. We also give a
general picture for this situation, see Figure \ref{figtwee}.
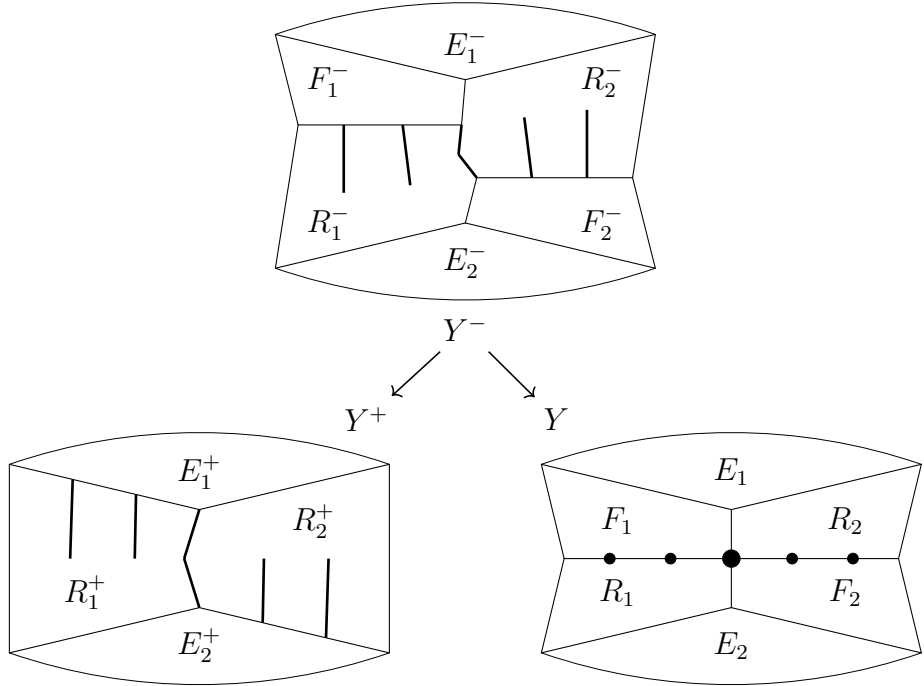
\begin{figure}\centering
    \begin{tikzpicture}
    \newcommand{\thick}{1.0pt}
    \newcommand{\Mitte}{1.65}
    \newcommand{\thin}{0.4pt}
    \newsavebox{\YplusD}
    \savebox{\YplusD}{
      \tkzDefPoint(2.5,2.3){X}
      \tkzDefPoint(5.0,2.9){A}
      \tkzDefPoint(0,2.9){B}
      \tkzDefPoint(2.5,-4.3125){O1}
      \tkzDrawSegments[line width = \thin](X,A X,B)
      \tkzDrawArc(O1,A)(B)
      \tkzDefPoint(2.5,1.0){Y}  
      \tkzDefPoint(0,0.4){C}
      \tkzDefPoint(5.0,0.4){D}
      \tkzDefPoint(2.5,7.6125){O2}
      \tkzDrawSegments[line width = \thin](Y,C Y,D)
      \tkzDrawArc(O2,C)(D)
      \tkzDrawSegments[line width = \thin](A,D B,C)
      \tkzDefPoints{%
        2.3/\Mitte/E0,
        0.8/\Mitte/E11,
        1.65/\Mitte/E12,
        3.35/\Mitte/E21,
        4.2/\Mitte/E22}
      \draw [line width = \thick] (X) -- (E0) (E0) -- (Y);
      \tkzDefPoint(0.83333333,2.7){F11}
      \tkzDefMidPoint(X,F11)\tkzGetPoint{F12}
      \tkzDefPoint(4.16666666,0.6){F22}
      \tkzDefMidPoint(Y,F22)\tkzGetPoint{F21}
      \draw [line width = \thick] (E11) -- (F11) (E12) -- (F12);
      \draw [line width = \thick] (E21) -- (F21) (E22) -- (F22);
      \node at (2.5,2.8) {$E_{1}^{+}$};
      \node at (2.5,0.5) {$E_{2}^{+}$};
      \node at (1.0,1.2) {$R_{1}^{+}$};
      \node at (4.0,2.15){$R_{2}^{+}$};
    }
    \newsavebox{\YD}
    \savebox{\YD}{
      \tkzDefPoint(2.5,2.3){X}
      \tkzDefPoint(5.0,2.9){A}
      \tkzDefPoint(0,2.9){B}
      \tkzDefPoint(2.5,-4.3125){O1}
      \tkzDrawSegments[line width = \thin](X,A X,B)
      \tkzDrawArc(O1,A)(B)
      \tkzDefPoint(2.5,1.0){Y}  
      \tkzDefPoint(0,0.4){C}
      \tkzDefPoint(5.0,0.4){D}
      \tkzDefPoint(2.5,7.6125){O2}
      \tkzDrawSegments[line width = \thin](Y,C Y,D)
      \tkzDrawArc(O2,C)(D)
      \tkzDefPoints{%
      0.3/\Mitte/M,
      4.7/\Mitte/N}
      \tkzDrawSegments[line width = \thin](M,B M,C N,A N,D M,N X,Y)
      \tkzDefPoints{%
        2.5/\Mitte/E0,
        0.9/\Mitte/E11,
        1.7/\Mitte/E12,
        3.3/\Mitte/E21,
        4.1/\Mitte/E22}
      \tkzDrawPoints[size=4pt](E11,E12,E21,E22)
      \tkzDrawPoints[size=6.5pt](E0)
      \node at (2.5,2.8) {$E_{1}$};
      \node at (2.5,0.5) {$E_{2}$};
      \node at (1.0,1.2) {$R_{1}$};
      \node at (4.0,2.15){$R_{2}$};
      \node at (1.0,2.15){$F_{1}$};
      \node at (4.0,1.2) {$F_{2}$};
    }
    \newsavebox{\YminusD}
    \savebox{\YminusD}{
      \tkzDefPoint(2.5,2.9){X}
      \tkzDefPoint(5.0,3.5){A}
      \tkzDefPoint(0,3.5){B}
      \tkzDefPoint(2.5,-3.7125){O1}
      \tkzDrawSegments[line width = \thin](X,A X,B)
      \tkzDrawArc(O1,A)(B)
      \tkzDefPoint(2.5,1.0){Y}  
      \tkzDefPoint(0,0.4){C}
      \tkzDefPoint(5.0,0.4){D}
      \tkzDefPoint(2.5,7.6125){O2}
      \tkzDrawSegments[line width = \thin](Y,C Y,D)
      \tkzDrawArc(O2,C)(D)
      \tkzDefPoints{%
      0.3/2.3/M,
      4.7/1.6/N,
      2.45/2.3/P,
      2.65/1.6/Q}
      \tkzDrawSegments[line width = \thin](M,B M,C N,A N,D M,P N,Q X,P Y,Q)
      \tkzDefPoints{%
        2.41/1.91/E0,
        0.9/2.3/E11,
        4.1/1.6/E22,
        0.9/1.4/F11,
        4.1/2.5/F22}
      \tkzDefMidPoint(P,E11)  \tkzGetPoint{E12}
      \tkzDefMidPoint(Q,E22)  \tkzGetPoint{E21}
      \tkzDefMidPoint(P,F22)  \tkzGetPoint{F21}
      \tkzDefMidPoint(Q,F11)  \tkzGetPoint{F12}
      \draw [line width = \thick] (P) -- (E0) (E0) -- (Q);
      \draw [line width = \thick] (E11) -- (F11) (E12) -- (F12);
      \draw [line width = \thick] (E21) -- (F21) (E22) -- (F22);
      \node at (2.5,3.35) {$E_{1}^{-}$};
      \node at (2.5,0.5) {$E_{2}^{-}$};
      \node at (0.7,1.0) {$R_{1}^{-}$};
      \node at (4.3,2.9) {$R_{2}^{-}$};
      \node at (0.7,2.9) {$F_{1}^{-}$};
      \node at (4.3,1.0) {$F_{2}^{-}$};
    }
    \node at (0,0)     {\usebox{\YplusD}};
    \node at (7.0,0)   {\usebox{\YD}};
    \node at (3.5,5.1) {\usebox{\YminusD}};
    \node (T)  at (6.0,4.7) {$Y^{-}$};
    \node (BL) at (4.7,3.5) {$Y^{+}$};
    \node (BR) at (7.2,3.5) {$Y$};
    \draw[->,line width = 0.6pt] (T) -- (BL);
    \draw[->,line width = 0.6pt] (T) -- (BR);
    \pgfresetboundingbox
    \path[use as bounding box] (0,0) rectangle (12,9);
  \end{tikzpicture}
\caption{Factorisation in case $P$ lies on the discriminant curve.}
\label{figtwee}
\end{figure}

\smallskip
We now consider  the family $\cy\to\Pi$. 
In the base space $\Pi$, we have two
divisors $\Delta_i$, given by $\al i=0$ ($i=1,2$). Over these, the
fibres of the family   are reducible. In fact, the section
$\ce_i\colon y_0=y_i=0$ of the $\P^2$-bundle over $\Delta_i$ is a
component. The remaining component of $\cy|_{\Delta_{i}}$ will be
denoted by $Y_{i}$. The intersection $E_{i}$ of $\ce_i$ and $Y_i$
lies over $Q$ and is given by $\al i= y_0=y_i=Q=0$. Over
$\Lambda=\Delta_1\cap \Delta_2$ there are three components,
$\ce_1|_\Lambda$, $\ce_2|_\Lambda$ and $Y_\Lambda$, which is the
intersection of $Y_1$ and $Y_2$. The space $Y_\Lambda$ over
$\Lambda$ is given by Equations (\ref{yyy}).

\begin{remark}
It is crucial for the construction of the birational transformations to
observe that the plane $L$ is tangent to $Q$ for all parameters in
$\Delta_{1}\cup\Delta_{2}$. This follows from Proposition \ref{prop:sing}. 
  It is interesting to note that away from $\Delta_{1}\cup\Delta_{2}$, that is, 
  if $\alpha_{1}\alpha_{2}\ne0$, where $L$ in general is not tangent
  to $Q$, even if $L$ happens to be tangent to $Q$, over a point
  of tangency, there is no singularity of the fibre of $\cy$, provided that 
  $L$ was not tangent to $Q$ at a point on $K$. 
\end{remark}

Some of the divisors we encountered in describing $Y$ (see Figure
\ref{figeen}) also extend over  $\Delta_1$ and others
over $\Delta_2$. If over $\Delta_1$, where $\al1=0$,  also
$Q=L=0$, then  the lines $\{y_1=0\}$ in the fibres of the
$\P^2$-bundle lie on $Y_1$. As $Q=L=0$ consists of two lines
$l_1=\pr1^{-1}(a)$ and $l_2=\pr2^{-1}(b)$, one from each ruling, we
have two ruled surfaces in each fibre of our family over $\Delta_1$.
As the ruled surfaces are given by $y_1=0$ they are the restriction
of the bundle $\P(\co\oplus \co(-2)) \rightarrow \P^{3}$ to a line,
so each is a Hirzebruch surface $\mathbb{F}_{2}$. We set
$F_1=\pi^{-1}(l_1)\cap \{y_1=0\}$ and $R_2=\pi^{-1}(l_2)\cap
\{y_1=0\}$, consistent with our earlier notation. We consider $F_1$
as a family of ruled surfaces over $\Delta_1$. Correspondingly, we have
$F_2$ and $R_1$ over $\Delta_2$.
Note that $F_{1}\cap F_{2} = \pi^{-1}(l_{1}\cap l_{2}) \cap \{y_{1}=y_{2}=0\}$ 
is a single point in each fibre over $\Lambda$. This point is a singularity of
the corresponding fibre of $\cy$.

The deformation of the divisors $\wt S_1$ and  $\wt S_2$ will
be constructed from the following divisors in $\cy$:
\[
\begin{array}{rcl}
  \cR_1 &=& \{ L=0,\; y_2=0, \; \al2 y_1 - y_0 Q =0\}\;,\\
  \cR_2 &=& \{ L=0,\; y_1=0, \; \al1 y_2 - y_0 Q =0\}\;.
\end{array}
\]
Over each point of $\Pi\setminus\Delta_2$, we have $\al2\ne 0$, thus
$\cR_1$ is isomorphic to the plane defined by $L$ in $\P^3$. 
Similarly, over each point of $\Pi\setminus\Delta_1$, 
$\cR_2$ is isomorphic to the same plane. Over $\Delta_2$, however, 
$\cR_1$ splits into three components: $F_2$, $R_1$ and $\{L=0, y_0=y_2=0\}$, a
family of planes in $\ce_2$. 
Similarly, over $\Delta_1$, $\cR_2$ splits into $F_1, R_2$ and $\{L=0,
y_0=y_1=0\}$. Only the components $R_i$ survive our construction, the other
components will be contracted, so that we shall eventually have a family of
irreducible divisors.

The construction of the deformation $\wt\cz\to\Pi$ proceeds
by constructing  the following maps, which are maps of families over
the base space $\Pi$:
\[
\begin{xymatrix}@=1pc
  {&&&\wt\cy\ar[dl]\ar[dr]&&&&&\\
  \cy&&\cy^{-}\ar[ll]\ar@{-->}[rr]&&\cy^{+}\ar[rr]&&\cz&&\wt\cz\makebox[0pt][l]{$\;.$}\ar[ll]}
\end{xymatrix}
\]

Let us give a brief overview over these constructions before we
enter a detailed description.

\begin{enumerate}
\item The morphism $\cy^-\to \cy$ is the
  simultaneous small partial resolution of the extra singularities in each
  fibre, introduced by the plane $L$.
\item  The rational map $\cy^-\dashrightarrow\cy^+$ is the
 simultaneous flop of all lines on the transforms  $F_1^-$ and $F_2^-$ of $F_1$
 and $F_2$. We construct the flop explicitly as blow-up followed by blowing
 down in the other direction.
\item The morphism $\cy^+\to\cz$  contracts the strict transforms
 $\ce_i^+$ of the components $\ce_i$.
\item The final step $\wt\cz \rightarrow \cz$ is
  the simultaneous small resolution of the remaining singularities in the
  fibres.
\end{enumerate}

In the second and third steps  the family is modified only over
$\Delta_{1}\cup\Delta_{2}$. For $p\not\in\Delta_{1}\cup\Delta_{2}$,
the fibres of $\wt{\mathcal{Z}}$ are small resolutions of the fibres
of $\mathcal{Y}$.

\subsection{$\cy^-\lra\cy$}\label{subsec:31}
By construction of $\Pi$, a small simultaneous resolution of
all isolated singular points exists. In the first step, we only want 
to resolve the new singularities, introduced by the plane $L$;
we take care of the remaining singularities in the final step.
The reason is that we want to stay as close as possible to the spaces
defined by our equations. In particular, after the second step, 
we want the space $Y^+=W$ as component of a fibre, and not a small
resolution of it. This also means that the construction works over
the base space $\wl \Pi$: the covering $\Pi\to\wl\Pi$ is only needed
for the final step.
Therefore, we consider only a partial resolution.

Over $\Lambda$, the partial small resolution is the partial
resolution  $Y^-$, described above in factoring the rational map
$W=Y^+\dashrightarrow Y$ (cf.~Figures \ref{figeen} and \ref{figtwee}).
The divisors $\cR_i \subset \cy$ will be blown up in each
fibre to become divisors $\cR_i^- \subset \cy^-$.

\begin{remark} \label{rmk-explicit}
We give a more explicit description of the partial
resolution, although this is not necessary for the proof. 
We first locate the singularities. Then, we use that, over $\Lambda$, which 
small resolution of the extra singularities has to be chosen is determined by 
$W\dashrightarrow Y_{\Lambda}$.
The condition for a singular point in a fibre of $\cy \to \Pi$ is
given by \eqref{sing}.
Above $\Lambda$, that is, for $\al1=\al2=0$, the isolated singular
points lie  in $y_1=y_2=0$ on $Q=KL=0$ and are given by the
condition that the differential $\D(KL)$ is proportional to $\D Q$.

On the  intersection of $K=0$ and $L=0$, one has
$\D(KL)=0$, so also for $\al1$, $\al2\neq0$, there are isolated 
singular points at $y_1=y_2=0$ on $Q=KL=0$. In general, this gives six 
of the seven  ordinary double points introduced by $L$.
Over $\Lambda$, in general, the seventh singularity is the point of
intersection of $F_{1}$ and $F_{2}$.

The first equation $y_1y_2-KLy_0^2=0$ of \eqref{ydef}
leads to two globally defined partial small resolutions.
Namely, we can consider the closure $\wh \cy_1$ of the graph of the
rational map $\cy \dashrightarrow  \P(\co\oplus \co(-1))$ given by
$(z_0:z_1)=(y_1:Ky_0)=(Ly_0:y_2)$, or $\wh \cy_2$ coming from the
map $(z_0:z_1)=(y_1:Ly_0)=(Ky_0:y_2)$. The space $\wh \cy_1$ is
given by the equations
\begin{gather*}
 \Rank  \begin{pmatrix} y_1 & Ly_0 & z_0 \\ Ky_0 & y_2 & z_1\end{pmatrix} 
\leq 1\;, \\
 \al2y_{1}+\al1y_{2}-Qy_{0}=0\;.
\end{gather*}
The projection $\wh \cy_1 \to \cy$ is an isomorphism outside the set given by
$K=Q=L=0$ and $y_{1}=y_{2}=0$. 
In the chart  $z_0=1=y_0$, we  have the equations $y_2=Lz_{1}$ and 
$K=y_1z_1$ (and $\al2y_{1}+\al1Lz_{1}-Q=0$), so for a singular point of 
$K$, we have only a partial resolution, and we indeed keep exactly
the singularity we want to leave until the final step.

The preimage of each of the points given by $K=Q=L=0$ and $y_{1}=y_{2}=0$ is a
$\P^{1}$. These curves lie in the strict transform of $y_2=L=0$, so
over $\Lambda$ in $F_2^- \cup R_1^-$. Therefore,
$\wh\cy_1$ is not the correct resolution for all singularities. We  
use it to resolve only those singularities which are near those which are
contained in $F_{1}$ over $\Delta_{1}$. For the singularities close to those
in $F_2$ over $\Delta_{2}$, we use $\wh\cy_2$. This gives a simultaneous small
resolution of (in general) 
six of the extra singularities for fibres over a neighbourhood
of $\Lambda$ in $\Pi$.

For the last singular point, introduced by $L$,
we only describe the resolution over
$\Lambda$. We consider two cases, depending on whether the point of
tangency $P$ lies on the discriminant curve or not. We treat here only
the more difficult case. The assumption that $L$ is a tangent plane
implies that we can write, at least locally on the base space,
 $Q=LM-N_1N_2$, where $M$, $N_1$ and $N_2$
are linear forms. The equations $L=N_i=0$ define the line $l_i$ on $Q$,
so $F_i$ is now given by $N_i=y_i=0$, and $R_i$ by $N_i=y_{3-i}=0$. 
We derive the
equation
\[
My_1y_2-KN_1N_2y_0^2=0\;.
\]
The small resolution comes about as the graph of a map to the 
$\P^1\times\P^1$-bundle  over $\P^{3}$
\[\P\left(\sier{}(-2)\oplus\sier{}\right) 
\times_{\P^{3}} \P\left(\sier{}\oplus\sier{}(-1)\right)\;.\]
On the target, we have fibre coordinates $(z_0:z_1\semic z_0':z_1')$. 
We set
\[
\frac{z_0}{z_1}=\frac{My_1}{N_2y_0}= \frac{KN_1y_0}{y_2}\;, \qquad
\frac{z_0'}{z_1'}=\frac{My_1}{N_2Ky_0}= \frac{N_1y_0}{y_2}\;.
\]
The exceptional curve of the small resolution consists of two
intersecting rational curves.
We look at the chart centred at their intersection point;
it is given by $z_0'=1$, $z_1=1$. Then $y_2=z_1'N_1y_0$,
$My_1=z_0N_2y_0$ (note that $M\neq 0$ if $L=N_1=N_2=0$) 
and $K=z_0z_1'$. If $P$ is a singular point of $K$, 
we have only a partial resolution. The exceptional curves lie in the
intersection of the strict transforms $R_i^+$ of the $R_i$.

\smallskip
What we have achieved now is that the strict transforms $F^-_i$ of
$F_i$, being isomorphic to  $F_i$, are disjoint ruled surfaces with
smooth neighbourhoods, cf.~Figure \ref{figeen} and \ref{figtwee}.

The remaining singularities are not
contained in one of the subvarieties $\ce_{i}^-, F_{i}^-$,
as over $\Lambda$
they satisfy $y_{1}=y_{2}=0$ and are mapped to the singular points
of the curve $Q=K=0$ in $\P^{3}$.  
Therefore, they do not affect our constructions until the final step.
\end{remark}

\subsection{$\cy^{-} \dashrightarrow \cy^{+}$}
We flop all the lines in  the disjoint union of $F_{1}^-$ and
$F_{2}^-$. Because of their disjointness, we can study the flop for
each component separately. We study both cases at the same time
using our notational convention of dropping all indices, that is, we write
$\Delta$, $E^-$, $ F^-$, and so on, instead of $\Delta_{i}$, $E_{i}^-$,
$F_{i}^-$. 

We start the  construction of the flop by  blowing up $F^-$. To describe the
result, we first compute the normal bundle $N_{F^-/\cy^-}$ of
$F^-\subset \cy^-$. More precisely, we compute the restriction of it
to the fibre over any point $p\in\Delta$. We drop the index $p$ from
the more correct way of writing. We recall that $F^-\cong
\mathbb{F}_{2}$, and so, its Picard group is generated by the classes of
its negative section $B^-=F^-\cap E^-$ and of a fibre $f$ of the
projection $F^- \rightarrow B^-$.

\begin{proposition}
\[
N_{F^-/\cy^-}\cong \sier{F^-}(-B^--f)^{\oplus 2}\;.
\]
\end{proposition}

\begin{proof}
The restriction $\cy^-|_\Delta$ consists of the two components
$\ce^-$ and $Y^-$. The fibres of $Y^-\rightarrow\Delta$ over
$\Lambda$ consist of two components but all other fibres are
irreducible. The inclusions $F^- \subset Y^- \subset \cy^-$ equip us
with an exact sequence of normal bundles:
\[
0 \lra N_{F^-/Y^-} \lra
                N_{F^-/\cy^-} \lra
                N_{Y^-/\cy^-}|_{F^-}
    \lra 0\;.
\]
Because this is a sequence of locally free sheaves (in a
neighbourhood of $F^-$), its restriction to a fibre over $p$ is
still exact, justifying our abuse of notation. To compute
$N_{Y^-/\cy^-}|_{F^-}$, we observe that $Y^-$  is a divisor in
$\cy^-$, hence $N_{Y^-/\cy^-}|_{F^-} \cong \sier{\cy^-}(Y^-) \otimes
\sier{F^-}$. Using $\sier{\cy^-}(\cy^-|_\Delta) \otimes \sier{F^-}
\cong \sier{F^-}$, the equality $Y^- = \cy^-|_\Delta - \ce^-$ in
$\Pic(\cy^-)$ and the fact that $\ce^-$ and $F^-$ intersect
transversally along $B^-$, we obtain $\sier{\cy^-}(Y^-) \otimes
\sier{F^-} \cong \sier{F^-}(-\ce^-\cdot F^-) \cong
\sier{F^-}(-B^-)$.

We write $N_{F^-/Y^-}\cong \sier{F^-}(aB^-+bf)$ with certain
integers $a$, $b$. To compute these numbers, we observe that they
appear as intersection numbers in $Y^-$ as follows: $(F^-\cdot f) =
(N_{F^-/Y^-}\cdot f) = ((aB^-+bf)\cdot f)_{F^-} = a$ and $(F^-\cdot
B^-) = (N_{F^-/Y^-}\cdot B^-) = ((aB^-+bf)\cdot B^-)_{F^-} = b-2a$.

Using $F^-\cdot E^-=B^-$, we see that $(F^-\cdot B^-) = (F^-\cdot
F^-\cdot E^-) = (B^-\cdot B^-)_{E^-} = 0$, because $B^-$, as a curve
in $E^-$, is a fibre of one of the projections of $E^-\cong
\P^{1}\times\P^{1}$ to $\P^{1}$. This gives $b=2a$.

Next, we compute $a=(F^-\cdot f)$. For all $p\in\Delta$, we choose  a
fibre by specifying a point in $B^-$. As $B^-$ is a line on the
image of the  quadric $Q$ in the section $\ce^-\cong \P^3$, it
intersects the image of a general plane (not tangent to $Q$)
transversally at one point. As we now have two flat families over
$\Delta$, the surfaces $F^-$ and the fibres $f$, the intersection
number $(F^-\cdot f)$ does not depend on $p$ (here we use that
$\Delta$ is connected). So, we can restrict our attention to a
general point $p\in\Lambda$. The map $\cy^-\to\cy$ is an isomorphism
in the neighbourhood of the fibre $f$, so we compute $(F\cdot f)$ on
$Y$. Now  $Y$ is reducible. The component isomorphic to $\P^{3}$
does not meet $F$, so we can compute inside the component $Y_0$, being a
fibre of $Y_{\Delta}\to\Delta$.
Because $F$ is defined by $y_{i}=0$ over one of the lines given by
$Q=L=0$, the fibre of $Y_0 \rightarrow Q$ which contains $f\subset
F$ has a second component, which intersects $F$ transversally at one
point.  The intersection number of $F$ with any fibre of the
projection of the conic bundle is zero, as $F$ is disjoint to
generic fibres. Therefore,  $a=(F\cdot f) = -1$.

Our  exact sequence has therefore the following form:
\begin{equation}\label{sequ:curve}
  0 \lra \sier{F^-}(-B^--2f)  \lra
                N_{F^-/\cy^-}
    \lra
                \sier{F^-}(-B^-)
    \lra 0\;.
\end{equation}
To see that this sequence does not split, we restrict it to the curve
$B^-$. As $F^-$ intersects $\ce^-$ transversally in $B^-$, we have
$N_{F^-/\cy^-}|_{B^-}\cong N_{B^-/\ce^-}$. If the sequence splits,
then $N_{F^-/\cy^-}|_{B^-} \cong \sier{F^-}(-B^--2f)|_{B^-} \oplus
\sier{F^-}(-B^-)|_{B^-} \cong \sier{B^-} \oplus \sier{B^-}(2)$,
which is not the case as $B^-$ is an ordinary line in $\ce^-\cong
\P^{3}$ with normal bundle $\sier{B^-}(1)^{\oplus 2}$. The statement
follows. \end{proof}

Let $\wt\cy$ be the blow up of $F^-$ in $\cy^-$ (or more precisely,
of $F^-_1$ and $F^-_2$). Note that $F^-$ is a family of surfaces over $\Delta$,
so that its codimension in $\cy^-$ is equal to $2$.
We denote the exceptional locus of the blow
up by $\wt F$. From  our computation of the normal bundle, we know
$\wt F \cong
\P(N_{F^-/\cy^-}^{\vee}) \cong F^-\times \P^{1}$
and $N_{\wt F/\wt\cy}\cong \sier{\wt F}(-1) \cong \sigma^{\ast}
\sier{F^-}(-B^--f) \otimes \tau^{\ast} \sier{\P^{1}}(-1)$, where
$\sigma$, $\tau$ are the two projections of $\wt F \cong F^-\times
\P^{1}$ onto its factors.

We are now going to contract $\wt F$ inside $\wt \cy$.  To define
this contraction, we note that the isomorphism $\wt F\cong F^-\times
\P^1$ makes $\wt F \to \wt B$ into a $\P^{1}$-bundle, where $\wt
B\subset\wt \cy$ is the preimage of $B^-\subset\cy^-$. 
We have a $\P^{1}$-bundle $\wt F \to  \wt B$ with fibre $\wt f$. 
We compute $N_{\wt F/\wt\cy}|_{\wt f} \cong \sigma^{\ast} 
      \sier{F^-}(-B^--f) \otimes
        \tau^{\ast} \sier{\P^{1}}(-1)  \otimes \sier{\wt f}\cong
\sier{\wt f}(-1)$, because $(f\cdot f)_{F^-}=0$ and $(B^-\cdot
f)_{F^-}=1$. By the Castelnuovo--Moishezon--Nakano criterion (Theorem \ref{CMN}),
 $\wt F$ can be contracted.

This shows that there exists a morphism $\wt\cy \to \cy^+$
which contacts both $\wt F_i$.
The image of $\wt F$ in $\cy^+$ is denoted as $F^+$. This
surface no longer lies in $Y^+$, it is contained in $\ce^+$. It
intersects $Y^+$ along the curve $B^+:=F^+ \cap E^+$. The pair $(E^-,B^-)$ is
isomorphic to $(E^+,B^+)$ under the flop. 
The component $Y^+$ is obtained from
$Y^-$ by contracting $F^-$ along the fibration $F^-\to B^-$.

We describe the subvarieties involved in the flop:
\[
\xymatrix@=1pc
{
  \ce^- && E^-\ar[ll]\ar[rr] && Y^- && && \ce^+ && E^+\ar[ll]\ar[rr] && Y^+ \\
  && && &\ar@{<-->}[rr]& && && && \\
  && B^-\ar[rr]\ar[uu] && F^-\ar[uu] && && F^+\ar[uu] && B^+\ar[uu]\ar[ll] &&\qquad .
}
\]
All arrows on both sides are inclusions.

We consider the effect of the flop on the divisor $\cR$.  
We write $\cR_i$ because both indices occur.
As $F_{3-i}^-\subset \cR_i^-$, the blow-up does not change
$\cR_i^-$. We obtain divisors $\wt\cR_i$ in
$\wt\cy$ which are isomorphic to $\cR_i^-$. During
the contraction to $\cy^+$, however, the component $F_{3-i}$ of
$\wt\cR_i$ over $\Delta_{3-i}$ will be contracted. The
resulting divisor $\cR_i^+\subset\cy^+$ still has irreducible
fibres over $\Pi\setminus\Delta_{3-i}$, but consists of two components only
over $\Delta_{3-i}$.

\subsection{$\cy^+\lra \cz$}\label{subsec:33}
We continue to omit the subscripts $i$ to deal with two disjoint
subvarieties at the same time. The goal of this step is the
contraction of the strict transform $\ce^+$ of $\ce^-$ inside
$\cy^+$ to a $\P^{1}$-bundle over $\Delta$, thereby making the
fibres of the obtained family $\cz\to\Pi$ irreducible.

The first useful observation is that $\ce^+$ is the blow-up of
$\ce^-$ in $B^-$. Over each point $p\in \Delta$, the curve $B^-$ is
an ordinary line in $\ce^-\cong \P^3$, so blowing up makes $\ce^+$
into a $\P^2$-bundle over a $\P^1$-bundle over $\Delta$.

To contract $\ce^+$ inside $\cy^+$, we first construct
a map $c\colon \ce^+\rightarrow C$.
We define
\[
C:= \P(\delta_{\ast}(\sier{\ce^-}(1)\otimes \mathcal{I}_{B^-}))\;,
\]
where $\delta:\ce^-\rightarrow \Delta$ is the projection,
$\sier{\ce^-}(1)$ denotes the pull back of $\sier{\P^{3}}(1)$ under
the projection $\ce^-\cong \P^{3}\times\Delta \rightarrow \P^{3}$
and $\mathcal{I}_{B^-}\subset \sier{\ce^-}$ denotes the ideal sheaf
of $B^-\subset \ce^-$. Because $B^-$ is a family of lines, $C$ is a
$\P^{1}$-bundle over $\Delta$. There exists a natural morphism
$c\colon \ce^+\rightarrow C$. For any $p\in\Delta$, the fibres are
the strict transforms of the planes in $\ce^-$ which contain the
line $B^-$.

To show that we can contract $\ce^+$ inside $\cy^+$ along the morphism 
$c$, we
have to show that the normal bundle $N_{\ce^+/\cy^+}$ restricts to
$\sier{\P^{2}}(-1)$ on the fibres of $c$.

As a divisor in $\cy^+$, we can write $\ce^+$ as the difference
$\cy^+|_{\Delta} - Y^+$ and, as before, we obtain $N_{\ce^+/\cy^+}
\cong \sier{\cy^+}(-Y^+) \otimes \sier{\ce^+}$. Each plane in a
fibre intersects $Y^+$ along a line, so indeed, the normal bundle
restricts to $\sier{\P^{2}}(-1)$ on the fibres of $c$. 

The above implies
the existence of a morphism $\cy^+ \rightarrow \cz$ contracting
$\ce^+_1$ and $\ce^+_2$ to $\P^{1}$-bundles $C_1$ and $C_2$ over
$\Delta_{1}$ and $\Delta_{2}$ respectively.
This morphism contracts the additional component of $\cR_i^+$ over
$\Delta_{3-i}$, so that the image $\cs_i$ of $\cR_i^+$ in
$\cz$ is a family of irreducible divisors.

\subsection{$\wt\cz\lra \cz$}
In the final step, we resolve the remaining fibre singularities. Over
$\Lambda$, the small resolution is given by construction of the family of
\KLB manifolds we started out with. We end up with a family
of smooth manifolds $\wt\cz\rightarrow\Pi$.
The strict transforms of the  divisors $\mathcal{S}_i\subset\mathcal{Z}$ are
divisors $\widetilde{\mathcal{S}}_i \subset\widetilde{\mathcal{Z}}$.

\subsection{The main theorem}
\begin{theorem}
Let  $\wt\cz\rightarrow\Pi$ be a family, constructed as above from a
given stratum of \KLB manifolds. The fibres over
$\Lambda\subset \Pi$ are \KLB manifolds, and the family
$\wt\cz\rightarrow\Pi$ together with the two
divisors $\wt\cs_1$ and $\wt\cs_2$ is a deformation, which
locally around $\Lambda$ is versal for
deformations of triples $(\wt Z,\wt S_1,\wt S_2)$.
\end{theorem}

\begin{proof}
The construction of $\wt\cz$ and Lemma \ref{lem:birational} guarantee that the 
fibres over $\Lambda$ are indeed \KLB manifolds.
Let $\wt Z_p$ be an \KLB manifold over the point
$p\in\Lambda\subset\Pi$ in the family $\wt\cz\rightarrow\Pi$. The space
$T^1_f$ of infinitesimal deformations for our 
deformation problem is described in the proof of Theorem \ref{raakzss}.
We have to show that the Kodaira--Spencer map $T_p\Pi
\to T^1_f$ is surjective.  Changing the lines in our construction,
that is, changing $L$, gives  surjectivity on the image of
$H^0(\cn_{\wt S_{1,p}})\oplus H^0(\cn_{\wt S_{2,p}})$. Therefore, we
have to study the image of $T^1_f$ in $H^1(\Theta_{\wt Z_p})$. On
the subspace of deformations  coming from deforming the conic
bundle, we again have surjectivity,  as the stratum $\Lambda_0$
gives a versal (but not miniversal) deformation of the conic bundle.
We are left with showing that the image of the $\al1$, $\al2$
deformations under the Kodaira--Spencer map span the two-dimensional
kernel of $H^1(\cn_{\wt C_{1,p}})\oplus H^1(\cn_{\wt C_{2,p}})\to
H^1(\cn_{\wt S_{1,p}}|_{\wt C_{1,p}}) \oplus H^1(\cn_{\wt
S_{2,p}}|_{\wt C_{2,p}})$. Here, the notation $\wt C_{i,p}$
means that we consider the fibre over $p$ of the $\P^1$-bundle
$\wt C_i\subset\wt\cz$, 
which is the isomorphic pre-image under the small resolution
of the bundle  $C_i\subset\cz$,  defined in Section \ref{subsec:33}.

Given a 1-parameter deformation $\wt\cz_T\to T$ of $\wt Z_p$, the image of the
Kodaira--Spencer map  in $H^1(\Theta_{\wt Z_p})$ is given by the connecting
homomorphism in the cohomology sequence of the sequence
\begin{equation}\label{kodsp}
0\lra \Theta_{\wt Z_p} \lra \Theta_{\wt \cz_T}|_{\wt Z_p} \lra
\cn_{\wt Z_p/\wt \cz_T}\cong \sier{\wt Z_p} \lra 0\;.
\end{equation}
We consider a 1-parameter deformation with base space $T\subset\Delta_2$,
given in terms of the equation of the form (\ref{ydef}) by
\begin{eqnarray*}
y_{1}y_{2}-KLy_{0}^{2}&=&0\;,\\
\al1y_{2}-Qy_{0}&=&0\;,
\end{eqnarray*}
where $KL$ is unchanged. 
We look at the image of the Kodaira--Spencer map under
the surjection $H^1(\Theta_{\wt Z_p})\to H^1(\cn_{\wt
C_{1,p}})\oplus H^1(\cn_{\wt C_{2,p}})$. As the curves in question
are disjoint from the singular points, we can do the computation on
$Z_p$, and consequently drop all tildes from the notation. The exact
sequence (\ref{kodsp}) gives two normal bundle sequences,
\begin{equation}\label{nbsone}
0\lra N_{C_{1,p}/ Z_p} \lra N_{C_{1,p}/ \cz_T} \lra N_{ Z_p/
\cz_T}|_{C_{1,p}}\cong \sier{C_{1,p}} \lra 0\;,
\end{equation}
and
\begin{equation}\label{nbstwo}
0\lra N_{C_{2,p}/ Z_p} \lra
N_{C_{2,p}/ \cz_T} \lra N_{ Z_p/ \cz_T}|_{C_{2,p}}\cong
\sier{C_{2,p}} \lra 0\;.
\end{equation}
The restriction map $H^0(\sier{Z_p})\cong H^0(\cn_{ Z_p/ \cz_T})
 \to  H^0( N_{ Z_p/\cz_T}|_{C_{1,p}} ) \oplus
H^0( N_{ Z_p/\cz_T}|_{C_{2,p}} )
\cong
H^0(\sier{C_{1,p}})\oplus
H^0(\sier{C_{2,p}})$ is the diagonal embedding. Furthermore, we know
that $N_{C_{i,p}/ Z_p}\cong\sier{C_{i,p}}(-2) \oplus
\sier{C_{i,p}}(-2)$ for $i=1,2$.

We first look at the second sequence \eqref{nbstwo}. 
In our construction  $C_2$ is
a $\P^1$-bundle over $\Delta_2$. So, the curve $C_{2,p}$ lies in a
family $C_{2,T} \to T$. We can compute $N_{C_{2,p}/ \cz_T}$ from the
sequence
\[
0\lra N_{C_{2,p}/ C_{2,T}} \lra N_{C_{2,p}/ \cz_T} \lra N_{ C_{2,T}/
\cz_T}|_{C_{2,p}} \lra 0\;.
\]
As $N_{ C_{2,T}/ \cz_T}|_{C_{2,p}} \cong N_{C_{2,p}/
Z_p}\cong\sier{\P^1}(-2) \oplus \sier{\P^1}(-2)$, we conclude that
the sequence splits, and that $N_{C_{2,p}/ \cz_T}\cong \sier{}\oplus
\sier{}(-2)\oplus \sier{}(-2)$. Therefore, also the sequence
(\ref{nbstwo}) splits and the connecting homomorphism
$H^0(\sier{C_{2,p}})\to H^1(N_{C_{2,p}/ Z_p})$ is the zero map.

Secondly, we look at the first sequence \eqref{nbsone}. The map
$Y_2^+|_T\to \cz_T$ contracts the $\P^2$ bundle $\ce_1^+|_p$ to the
curve $C_{1,p}$. As $\ce_1^+|_p$ is the blow up of $\P^3$ in an
ordinary line, it is isomorphic to $\P(\sier{}\oplus \sier{}(-1)\oplus
\sier{}(-1))$ as bundle over $\P^1$. This shows that the normal
bundle $N_{C_{1,p}/ \cz_T}$ is a twist of $\sier{}\oplus
\sier{}(1)\oplus \sier{}(1)$, and as its degree is $-4$, we find that
$N_{C_{1,p}/ \cz_T}\cong \sier{}(-2)\oplus \sier{}(-1)\oplus
\sier{}(-1)$. Therefore, the connecting homomorphism
$H^0(\sier{C_{1,p}})\to H^1(N_{C_{1,p}/ Z_p})$ is non-trivial.

The Kodaira--Spencer map for the analogous deformation in $\Delta_1$
is non-trivial on the other factor.
So, indeed, the images span the two-dimensional kernel in question.
\end{proof}


\section{The fibres over $\Delta_i$}\label{sec:4}
In this section, we investigate the structure of the fibres of the
family $\wt Z$ over $\Delta_1$  and $\Delta_2$. These spaces are halfway between
modifications of conic bundles (over $\Lambda=\Delta_1\cap \Delta_2$)
and double solids. Of the two pencils
of surfaces on the conic bundle one survives. As this is not
compatible with the real structure, these manifolds do not figure in
the twistor literature.

If $\al1\neq0$, but $\al2=0$, we can eliminate $y_2$ and have only
one equation in the $\P^1$-bundle $\P(\co\oplus\co(-2))$ 
over $\P^3$. 
As before, we deal with a single fibre for fixed $\al1\ne0, \al2=0$,
but continue to simplify notation by not introducing additional subscripts.
After dividing by $y_0$, which cuts away the component $\ce_2$,
we have the following equation for $Y_2$:
\begin{equation}\label{delta}
  y_{1}Q-\al1KLy_{0}=0\;.
\end{equation}
It describes the blow-up of $\P^3$ in the curve $Q=KL=0$. The
singularities of the blow-up occur at the singularities of the
curve. These have to be resolved with a small resolution. For the
singularities on $L=0$, it is determined by our construction. To
describe it in terms of a modification of $\P^3$, we first study the
local description of the blow-up of a curve with ordinary double
point in a smooth threefold.

Let the curve be given by  $z=xy=0$. The blow-up $sz-txy=0$ can be
covered by two affine charts, the first one $s=1$, which is smooth,
containing the strict transform of $z=0$, the second one ($t=1$) having an
$A_1$-singularity: $sz-xy=0$. A small resolution of this $A_1$-singularity is
given by $(s:y)=(x:z)=(u:v)$. We have in total
three charts, with
$(x,y,z) = (x,y,txy) = (zu,y,z) = (x,vs,xv)$. The same manifold is obtained by
first blowing up  the branch $z=x=0$, setting $(x:z)=(u:v)$, and then
the strict transform $y=v=0$ of $y=z=0$, setting $(y:v)=(s:t)$.
Interchanging the role of $x$ and $y$ gives the other small
resolution.

We now return to the space $Y_2$, given as subspace of a
$\P^1$-bundle over $\P^3$ by Equation (\ref{delta}) for
$\al1\neq0$. Over $l_1\subset Q$ lies $R_1$, over $l_2\subset Q$
lies $F_2$. The surfaces $F_1$ lie only over $\al1=0$. The planes
$R_2$, given by $y_1=L=0$ are contained in $Y_2$. The singular point
of $Y_2$ above the point of tangency of the quadric and its tangent
plane  $L=0$ has in general non-zero $y_1$-coordinate. The correct small
resolution is determined by the condition that no exceptional curve
is contained in the strict transform $F_2^-$. 

This small resolution of $Y_{2}$ is obtained from $\P^{3}$ by first
blowing up $l_1$, then the curve $K=0$ on the strict transform of
$Q$, introducing singularities coming from the singular points of
the curve and finally blowing up $l_2$, as lying on the strict
transform of $Q$. The exceptional surface $F_2^-$ is a $\P^1$-bundle
over $\P^1$. The flop $\cy^-\dashrightarrow \cy^+$ has on $Y_2^-$
the effect of contracting $F_2^-$ again. So, after the second step,
blowing up the curve $K=0$ on the quadric, we have already obtained
the space $Y_2^+$ (over $\al1\neq0$). This discussion proves the
following proposition.

\begin{proposition}\label{blowpp}
Let $\wt Z$ be an \KLB manifold, which is a small resolution
of a space $Z$ defined by Equation \eqref{klb} with $n=3$. Fix
a line $l_1=\pr1^{-1}(s)$ of the first ruling of the quadric $Q$.
The fibres of the family  $\wt \cz\to \Pi$ over $\Delta_2\setminus
\Lambda$, constructed from lines $l_2$ of the other ruling, such
that $L$ in Equations \eqref{ydef} satisfies $\{Q=L=0\}=l_1\cup
l_2$, are isomorphic for all $l_2$. Such a manifold $\wt Z_2$ is a
modification of $\P^3$, which is the composition of the following
birational maps
\[
\P^3 \longleftarrow Y_2' \longleftarrow Y_2^+ \lra Z_2
\longleftarrow \wt Z_2\;,
\]
where the first map is the blow-up of $l_1$, the second the blow-up
of the curve $K=0$ on the strict transform $Q'$ of $Q$ and the map
to $Z_2$ is the blow-down of the strict transform $Q^+$ along the
first ruling. The singularities of $Z_2$ are the same as those of
$Z$, and the map $\wt Z_2 \lra Z_2 $ resolves them in the same way
as $\wt Z \lra Z$ does.
\end{proposition}

We can use this direct description of the threefold $Z_2$ to study
the surfaces $\wt S_1$ and $\wt S_2$. The surface $\wt S_2$ moves
in a pencil, which contains some singular elements; these are described
in the proof of the following proposition. 

\begin{proposition}
Let $\wt Z_2$ be a modification of $\P^3$ as in the previous
proposition. It contains one pencil of surfaces $\wt S_2$, 
of which the smooth elements are the
blow-up of $\P^2$ in three (possibly infinitely near) collinear
points. Each surface intersects a surface $\wt S_1$, which does not
move in a pencil, and is the blow-up of $\P^2$ in three (possibly
infinitely near) non-collinear points, if non-singular.
\end{proposition}

\begin{proof}
The pencil of planes through the line $l_1\subset \P^3$ gives rise
to a pencil of planes $R_2'$ on $Y_2'$. The blow up introduces the
exceptional surface $R_1'\cong \P^1\times \P^1$. The strict
transform $Q'$ of the quadric intersects $R_1'$ in the diagonal of $R_1'$, and
each surface $R_2'$ in the line $l_2'$ lying in it. The curve $K=0$
on $Q'$ intersects $l_2'$ with multiplicity $3$. 

First suppose that the
intersection consists of three distinct points. Then, the strict
transform $R_2^+$ is the blow-up of $\P^2$ in three collinear
points. The image $S_2$ under the contraction map $Y_2^+\to Z_2$ is
isomorphic to $R_2^+$, and the small resolution does not change it,
as $S_2$ does not pass through any singular point in this case. 

If $l_2'$ is
tangent to the curve $K=0$ on $Q'$ at a smooth point, then $R_2^+$
is the blow-up of $\P^2$ in a rectilinear scheme of length $3$
(supported on one or two points) and therefore has a singularity
which remains under the subsequent rational maps. 
A local model for the simplest situation is that one blows up
$(x,y,z)$-space in the curve $x=z-y^2$, which is tangent to the plane
$z=0$. The strict transform of the plane has an ordinary double point.
It can also be obtained by blowing up the plane in the ideal $(x,y^2)$.

If the line $l_2'$
passes through a singular point of the curve $K=0$ on $Q'$, the
blow-up $R_2^+$ is also singular, but the small resolution $\wt Z_2
\to Z_2$ resolves it, and $\wt S_2$ is the blow-up of $\P^2$ in
three collinear points, some of which are infinitely near, unless
the line is also tangent to a branch of the curve. 

The surface
$R_1'$ is also blown up in a scheme of length three, which lies on
the diagonal. In the next step, the strict transform of the diagonal
is blown down. If the three points are in general position, then it
is easy to see that the surface $S_1$ has six $(-1)$-curves, the
strict transforms of both lines through the three points, which form
a cycle, and the surface is a Del Pezzo surface of degree 6.
Analysing the other possibilities is left to the reader.
\end{proof}

\section{Examples} \label{examples}

\subsection{Twistor spaces}
We specialise to the case of LeBrun twistor spaces, as considered in the
Introduction. So, $K$ is the product of three linear factors, defining
smooth curves of type $(1,1)$ on the quadric. 

Via the embedding 
$(x_0:x_1:x_2:x_3)=(s_0t_0: s_0t_1 : s_1t_0 : s_1t_1)$, the real
structure $(s_0:s_1\semic t_0:t_1)\mapsto (\bar t_0: \bar
t_1\semic\bar s_0: \bar s_1)$ on $Q$  gives rise to a
real structure on our family \eqref{ydef}, given by
\[
(x_0:x_1:x_2:x_3\semic y_0:y_1:y_2\semic  \al1,\al2) \mapsto (\bar x_0:\bar
x_2:\bar x_1:\bar x_3\semic \bar y_0:\bar y_2:\bar y_1\semic
\bar{\alpha}_{2} ,\bar{\alpha}_{1})\;.
\]
Furthermore, we have a real structure on the space of forms 
$L_1L_2L_3$, and the conditions we impose are compatible with the
real structure.
As our constructions are also compatible with the real structure, we get a
real structure on our versal deformation. The fibres over $\Delta_i$
do not occur in the real deformation, which lies over
$(\al1,\al2)=(\alpha,\bar\alpha)$. In particular, 
$\alpha\bar\alpha=|\alpha|^2\neq0$ for $\alpha\neq0 $.

A well-known argument of Donaldson and Friedman 
\cite[proof of Theorem 4.1]{df}, see also \cite[\S 5]{lb2} and
\cite[Proposition 2.1]{ca1}, shows that the fibres over real points of our
versal deformation are indeed twistor spaces, at least in a neighbourhood of
$\Lambda$ in $\Pi$. As a consequence, we obtain another proof of a result of
Honda \cite[Theorem 2.1]{ho4}, which states the existence of degenerate double
solids as twistor spaces, see Example \ref{nonumber} below.


We indicate how the twistor lines are deformed in the general  case. 
By \cite{ku} the general twistor line on a LeBrun twistor space lies above the
intersection of $Q\subset \P^3$ with a real hyperplane, in fact, one
has a whole $S^1$ of twistor lines over the same curve. The curve
can be given by a positive hermitian matrix $\psi$ with determinant
1. It intersects each of the three distinguished planes $L_i$, given
by a positive definite hermitian matrix $\vp_i$, at two points. One
of them has to be lifted to $w_1=0$, the other to $w_2=0$. The
choice is made in the following way. The secular equation $\det
(\la\psi -\vp_i)$ has two real solutions $\la_i>1>\mu_i$ and one can
write $\vp_i=\la_i\psi + f_i(s)\bar f_i(t)=\mu_i\psi - g_i(s)\bar
g_i(t)$ with $f_i$, $g_i$ unique up to a factor from $S^1\subset
\C^*$. Let $(s,t)$ and $(\bar t,\bar s)$ be the two
intersection points. Then, either $f_i(s)=\bar g_i(t)=0$ or $\bar
f_i(t)=g_i(s)=0$ and $\frac{\vp_i}{\psi}(s,\bar s)
-\frac{\vp_i}{\psi}(\bar t,t)=\la_i-\mu_i$ in the first case, while
being equal to $\mu_i-\la_i$ in the second case. One chooses a fixed
sign; by looking at a suitable degeneration, it turns out that one
needs always the negative sign.

We consider a double solid, which is a small real deformation of a LeBrun
twistor space. The same plane in $\P^3$ (assumed to be general) now
contains a smooth quartic branch curve and the twistor lines lie
over the real contact conics in one of the $63$ such systems. We
have to choose the correct one. Each system is determined by one of
six pairs of bitangents. We note that four bitangents are already
given, the intersections of the planes $L_i$, $L_4=L$ with the given
plane. We can write the equation $L_1L_2L_3L_4-Q^2$ in three ways as
symmetric determinantal, like
\[
\begin{vmatrix}L_1L_2& Q \\ Q & L_3L_4
\end{vmatrix}\;.
\]
From this matrix, one gets a system of contact conics, namely $\la^2
L_1L_2+2\la\mu Q+\mu^2L_3L_4$. To find the four other pairs of
bitangents, one computes when the conic degenerates. This gives an
equation of degree $4$ in $\la$ and $\mu$. In this way, we get all
bitangents.

As the double solid degenerates to the conic bundle, the quartic
curve degenerates to the conic section counted twice, with eight
marked points (the intersection of the conic with the four planes
$L_i$), so to a hyperelliptic curve of genus $3$. This
degeneration was already studied by Felix Klein \cite{kl}. The
bitangents degenerate to the 28 lines joining the eight points. Kurke's
construction divides the points into two groups of four. The
diagonals of the quadrangles thus found are the limits of the six pairs of
bitangents in the system of contact conics we are looking for.

Given a double solid, there is no preferred choice of a ruling of $Q$.
But by fixing coordinates as we did, we have made a choice.
Suppose that $\wt Z$ is a double solid twistor space,
occurring in our family. Reversing the roles of the rulings
means flopping all exceptional curves. By degenerating to 
$\al1=\al2=0$, we obtain a small resolution of the isolated
singularities, which does not lie over $\Lambda$.
But we can interchange the rulings by an automorphism of $\P^3$,
to obtain an isomorphic double solid (with in general a
different $KL$), which does occur in our family.  
Seen from another
perspective, given a deformation of a LeBrun twistor space, we can
interchange the rulings in the whole construction and get the same
general fibre (with 13 singular points). The difference is seen in
the small resolution of the 13th singular point.

\begin{proposition}
A double solid twistor space, which is a small deformation of a
LeBrun twistor space, is transformed into a twistor space by
flopping all exceptional curves.
\end{proposition}

We do not know if the double solid can be degenerated in more ways,
maybe due to some additional symmetries. The results of Honda
\cite{ho}, who finds only two small twistor resolutions for spaces
with extra symmetry, indicate that this is not the case.

\begin{example}[Honda's deformation with torus action] 
\label{nonumber}
A LeBrun twistor space with torus action has an equation
$w_1w_2+\vp_1\vp_2\vp_3 w_0^2=0$, where the $\vp_i$ define three conics
with common intersection points. These are real for our real structure
(note that it differs from Honda's \cite{ho2})
if we take $\vp_i=
a_is_0t_0+b_is_1t_1$. The singular points are $(1:0\semic0:1)$ and
$(0:1\semic1:0)$. If we now choose two lines, one in each ruling,
invariant under the $\C^*$-action on the quadric, they have to pass
through the singular points. Invariance under the involution
dictates that we take a degenerate conic in the pencil, either
$s_0t_0$ or $s_1t_1$. Both give an equivariant deformation, in
accordance with Honda's results. We proceed to describe the first
one explicitly. The general fibre is a double solid, the branch quartic of which
has an equation of the form $Q^2+4\al1\al2KL$, where $Q=x_0x_3-x_1x_2$ and
$KL$ is a polynomial of degree $4$ in $x_0$ and $x_3$ only, depending on
$\al1\al2$. We therefore take Equations \eqref{ydef} in the form
\begin{equation}\label{torus}
\begin{array}{rcl}
y_{1}y_{2}+
(a_1x_0+b_1x_3)(a_2x_0+b_2x_3)(a_3x_0+b_3x_3)(x_0+sx_3)y_{0}^{2}&=&0\;,\\
\al2y_{1}+\al1y_{2}-Qy_{0}&=&0\;,
\end{array}
\end{equation}
where $s$ is a function of $\al1\al2$, to be determined.
We now consider $\beta=1/\al1\al2$ and $s$ as independent variables,
and ask when the double solid has an $A_1$-singularity, besides the
two singular points at $(0:1:0:0)$ and $(0:0:1:0)$ of type $\wt
E_7$. The corresponding $A_1$-singularity of the branch quartic lies outside
$Q=0$, so the vanishing of the derivatives w.r.t.\ $x_1$ and $x_2$ gives
$x_1=x_2=0$. The condition for a singular point is then that $\beta
x_0^2x_3^2+4K(x_0,x_3)L(x_0,x_3\semic s)=0$ has a multiple root. 
This means that the discriminant of this binary form has to vanish. This
condition defines a curve  in a $\P^2$
with affine coordinates $(\beta:s:1)$. 
As Zariski observed \cite{zar}, it is a rational sextic curve with 
six cusps (the maximal number),  
being the dual of the rational quartic parametrised by
\[\psi(x_0,x_3)=(x_0^2x_3^2:x_0K:x_3K)\;.\] 
Indeed, the locus $\beta x_0^2x_3^2+4x_{0}K+4sx_{3}K=0$ is the incidence
correspondence between $\P^2$ and its dual, restricted to the curve defined by
$\psi$. Requiring a multiple root means that the line corresponding to the
point $(\beta:4:4s)$ in the dual plane is tangent to the curve given by $\psi$.

The dual curve can be parametrised by
the cross product $\psi_0\times \psi_3$, where we denote the partial
derivative w.r.t.~$x_i$ by the  subscript $i$. We obtain
\[
(\beta:s:1) = \bigl(-4K^2:x_0^2x_3(2K-x_3K_3):x_0x_3^2(2K-x_0K_0)\bigr)\;.
\]
From this formula, we find the 
curve in the $(s, \alpha_{1}\alpha_{2})$-plane which
describes the dependence of $s$ on $\al1\al2$.  
We write $(x_0:x_3)=(t:1)$ and $K(t,1)=A_0+A_1t+A_2t^{2}+A_0t^3$,
so that $A_{3}=a_{1}a_{2}a_{3},
A_{2}=a_{1}a_{2}b_{3}+a_{1}b_{2}a_{3}+b_{1}a_{2}a_{3}$,
$A_{1}=a_{1}b_{2}b_{3}+b_{1}a_{2}b_{3}+b_{1}b_{2}a_{3}$ and
$A_{0}=b_{1}b_{2}b_{3}$.
A straightforward calculation now shows  that 
\[
s=\frac{-t(A_0-A_2t^2-2A_3t^3)}{2A_0+A_{1}t-A_{3}t^{3}}
\qquad\text{ and }\qquad
\alpha_{1}\alpha_{2}=
\frac{-t(2A_0+A_1t-A_3t^3)}{4(A_0+A_1t+A_2t^{2}+A_0t^3)^{2}}\;.
\]
In the $(s,\al1\al2)$-plane this is a curve of degree 10, with 6 cusps,
3 double points and two very singular points at infinity.

For example, if $K=(x_0+x_3)(x_0+2x_3)(x_0+\frac12x_3)$, we obtain
\begin{align*}
s&=-t(2-7t^2-4t^3)/(4+7t-2t^3)\;,\\
\al1\al2&=-t(4+7t-2t^3)/2(1+t)^2(1+2t)^2(2+t)^2\;.
\end{align*}
For $(-15+\sqrt{97})/16<t<0$, we find that $\al1\al2$
is positive, and that the threefold has no real points besides the
extra singular point. For $t=(-15+\sqrt{97})/16$, the rational curve
has a double point, with $(s,\al1\al2)=(\frac14,\frac{16}9)$;
the corresponding double solid has not one, but two extra
double points. There
are two real cusps in this example, one at $t\approx 0.25$ and the other at
$t\approx 3.996$. 
We include a picture of the curve, and an enlarged view around the 
origin in the $(s, \alpha_{1}\alpha_{2})$-plane 
(Figs.\ \ref{figdrei} and \ref{figvier}).
\begin{figure}
\centering
  \begin{tikzpicture}
    \begin{axis}[
      width=54.6mm,
      height=61mm,
      axis lines = middle,
      x axis line style=-,
      axis y line*=right,
      xmin=-129,xmax=0,
      ymin=-0.5,ymax=1.3,
      minor x tick num=4,
      major tick length=2mm,
      ytick = {-0.4,-0.2,0,0.2,0.4,0.6,0.8,1.0,1.2},
      yticklabels = {$-0.001$,,0,,0.001,,0.002,,0.003},
      variable=t,
      ] 
      \addplot [
      domain=2.17:4,
      ]({\FuncX},{400*\FuncY});
      \addplot [
      domain=4:60,
      ]({\FuncX},{400*\FuncY});
    \end{axis}
    \pgfresetboundingbox
    \path[use as bounding box] (0,-5.05) rectangle (2.8,4.5);
  \end{tikzpicture}
  \begin{tikzpicture}
    \begin{axis}[
      width=121mm,
      height=110mm,
      axis lines = middle,
      xmin=-3.5,xmax=5.3,
      ymin=-4.2,ymax=7.7,
      ytick = {-4,...,7},
      x label style={at={(axis description cs:1.0,0.355)},anchor=north},
      y label style={anchor=east},
      xlabel={$s$},
      ylabel={$\alpha_{1}\alpha_{2}$},
      variable=t,
      ] 
      \addplot [
      domain=-4.1:-2.25,
      ]({\FuncX},{\FuncY});
      \addplot [
      domain=-1.76:-1.48,
      ]({\FuncX},{\FuncY});
      \addplot [
      domain=-1.44:-1.24,
      ]({\FuncX},{\FuncY});
      \addplot [
      domain=-0.76:-0.65,
      ]({\FuncX},{\FuncY});
      \addplot [
      domain=-0.646:-0.58,
      ]({\FuncX},{\FuncY});
      \addplot [
      domain=-0.416:1.52,
      ]({\FuncX},{\FuncY});
      \addplot[black,only marks]
      table {
        x    y
        0.25 1.7777777778
      };
    \end{axis} 
  \end{tikzpicture}
\caption{\small The dependence of $s$ on $\alpha_{1}\alpha_{2}$ with a far away
  component at another scale.} 
\label{figdrei}
\end{figure}
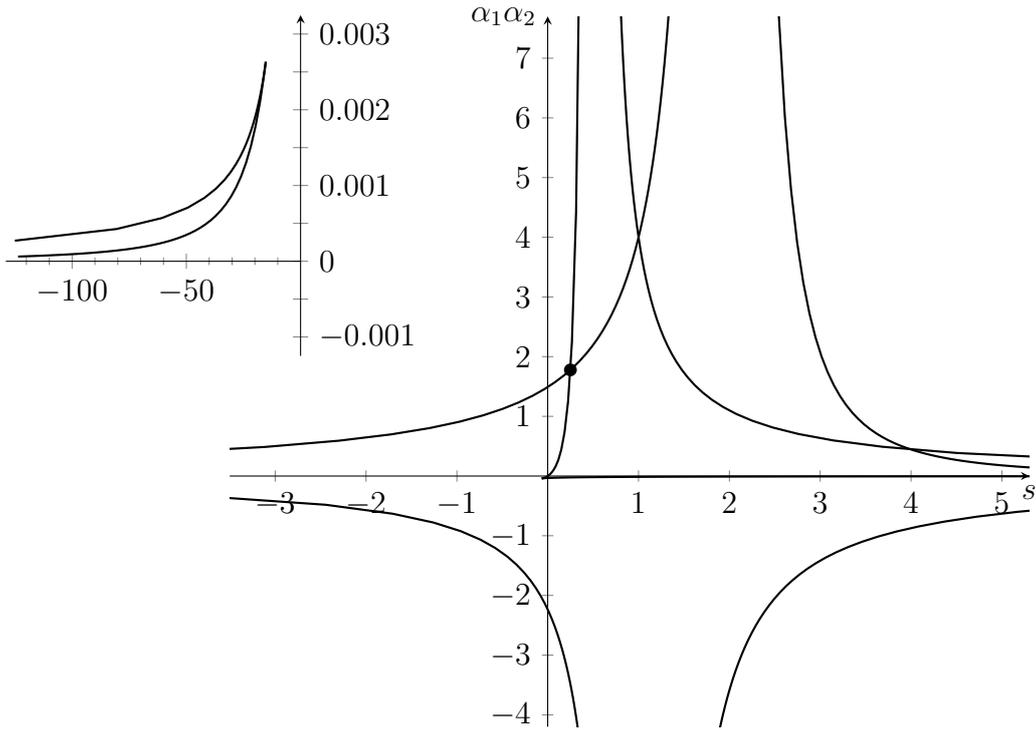

\begin{figure}
    \begin{tikzpicture}
    \begin{axis}[
      width=118mm,
      height=123.9mm,  
      axis lines = middle,
      axis equal image,
      xmin=-0.07,xmax=0.33,
      ymin=-0.08,ymax=0.34,
      xtick = {-0.05,0.05,0.1,0.15,0.2,0.25,0.3},
      xticklabels = {,,0.1,,0.2,,0.3},
      ytick = {-0.05,0.05,0.1,0.15,0.2,0.25,0.3},
      yticklabels = {,,0.1,,0.2,,0.3},
      xlabel={$s$},
      ylabel={$\alpha_{1}\alpha_{2}$},
      variable=t,
      ] 
      \addplot [
      domain=-0.19:0.75,
      ]({\FuncX},{\FuncY});
    \node[right] at (170,280) {$t<0$};
    \node[below] at (320,55)  {$t>0$};
    \end{axis}
  \end{tikzpicture}
\caption{Enlarged view around the origin.}
\label{figvier}
\end{figure}
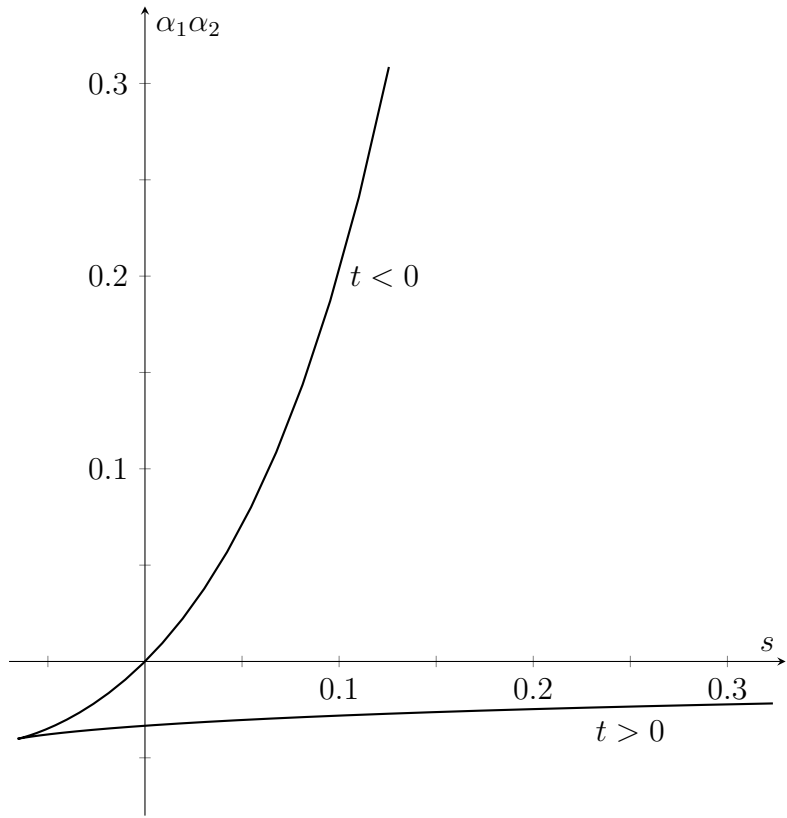

We remark that the simultaneous flop of all singularities is induced
by the global involution permuting $x_1$ and $x_2$, showing that in
this case the two twistor spaces are isomorphic, see also 
\cite[Thm.\ 8.1]{ho}.
\end{example}

\subsection{Quartics with 14 double points}
A quartic, which is given by an equation of the form $Q^2-L_1L_2L_3L_4=0$, in
general, has $12$ double points. 
The condition that it has a $13$th singular point, 
is a rather complicated one, at least in terms
of the coefficients of the planes $L_i$, as can be seen from the
explicit example with additional symmetry above.  

With one additional singular point, the situation becomes much easier.
A quartic surface with 14 double points has six tropes, and its
equation can be written in terms of them in the irrational form
\begin{equation} \label{irrational}
\sqrt{x_1x_2}+\sqrt{y_1y_2}+\sqrt{z_1z_2}=0\;,
\end{equation}
where $x_i, y_i, z_i$ are linear forms defining the tropes, see, 
for example, \cite[No. 55]{hud}, or in rational form as
$(x_1x_2)^2+(y_1y_2)^2+(z_1z_2)^2-2y_1y_2z_1z_2-2x_1x_2z_1z_2-2x_1x_2y_1y_2=0$.
By singling out four planes, we can write it in our preferred form
$(x_1x_2+y_1y_2-z_1z_2)^2-4x_1x_2y_1y_2=0$. The singular points come
in two types, eight being given by $x_i=y_j=z_k=0$, and six by
equations of the type $x_1=x_2=y_1y_2-z_1z_2=0$. If the product
$x_1x_2$ tends to zero, the quartic degenerates to a double quadric.
But this is exactly the quadric in our construction.

We start therefore with the quadric $Q\colon x_0x_3-x_1x_2=0$, parametrised by
$(x_0:x_1:x_2:x_3)=(s_0t_0: s_0t_1 : s_1t_0 : s_1t_1)$. On $Q$, we take a
discriminant curve consisting of two smooth conics and a pair of lines.
Two  lines  intersecting at a point 
$P = (a_0: a_1\semic b_0:b_1)$ on the quadric
are given by the intersection of $Q$ and the tangent plane
$a_1b_1x_0-a_1b_0x_1-a_0b_1x_2+a_0b_0x_3=0$ at the point $P$.
We write $K=K_1K_2L_0$ with $K_1$, $K_2$ general linear forms, and
we take $L_0$ as tangent plane $\la\mu x_0-\la x_1-\mu x_2 +x_3=0$
defined by the point  $(1: \la \semic 1:\mu)$. 
These data define an \KLB manifold. For our construction, 
we add a tangent plane
$L_3=a_1b_1x_0-a_1b_0x_1-a_0b_1x_2+a_0b_0x_3$.
In general, we get 14 double points, as Fig.~\ref{figvijf} illustrates.
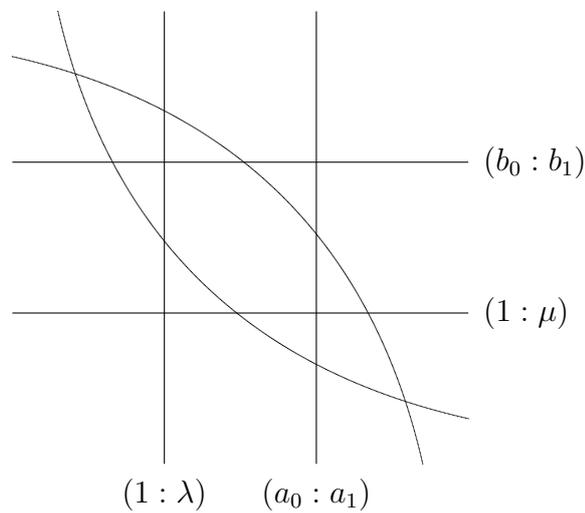
\begin{figure}
\unitlength=2cm
\begin{picture}(3,3.4)(0,-.3)
\put(0,1){\line(1,0){3}} \put(0,2){\line(1,0){3}}
\put(3.1,2){\makebox(0,0)[l]{$(b_0:b_1)$}}
\put(2,-.1){\makebox(0,0)[t]{$(a_0:a_1)$}}
\put(1,-.1){\makebox(0,0)[t]{$(1:\la)$}}
\put(3.1,1){\makebox(0,0)[l]{$(1:\mu)$}}
\put(1,0){\line(0,1){3}} \put(2,0){\line(0,1){3}}
\qbezier(0,2.7)(2.2,2.2)(2.7,0)
\qbezier(3,0.3)(.8,.8)(.3,3)
\end{picture}
\caption{Discriminant curve for a 14-nodal conic bundle.}
\label{figvijf}
\end{figure}
If the two conics are not tangent, as in the picture, we may
normalise the equation, and take 
$K=K_1K_2L_0=(x_0-px_3)(x_1-qx_2)(\la\mu x_0-\la x_1-\mu x_2 +x_3)$.
In general, we can also fix the intersection point of the two
lines, making  $\la=\mu=0$, but if the curve has additional symmetry,
which happens if the intersection point of the lines coincides with
one of the other two points, or if the  lines pass through both points,
then we have to consider a larger family.

After choosing the extra point on the quadric, we have two pairs
of lines. The intersection points of the two lines also
determine tangent planes, $L_1=\la b_1x_0-\la b_0x_1-b_1x_2+b_0x_3$
and $L_2=a_1\mu x_0-a_1x_1-a_0\mu x_2+a_0x_3$. 
The quadric $L_0L_3-L_1L_2=0$ passes through the four lines, so it is
an element of the pencil of quadrics through these lines.
We compute that its equation is, in fact, a multiple of $Q$:
\[
L_0L_3-L_1L_2= (a_1-\la a_0)(b_1-\mu b_0)(x_0x_3-x_1x_2)\;.
\]
So, if the point $P = (a_0: a_1\semic b_0:b_1)$ 
lies on one of the two
lines of $L_0$, the formula breaks down.
Actually, we excluded such cases in the construction
to avoid non-isolated singularities.  Otherwise we can write a
deformation of the form $Q^2-4\al1\al2K_1K_2L_0L_3$, with
$Q=L_0L_3-L_1L_2+\al1\al2 K_1K_2$. In the description of our family, 
we insisted that we kept the quadric unchanged. This can, of course, 
be achieved by a coordinate transformation, as long as the quadric
is non-singular, but it is easier to work with the present formula. 
To get rid of the term $(a_1-\la a_0)(b_1-\mu b_0)$, we replace $\al i$ by
$(a_1-\la a_0)(b_1-\mu b_0)\al i$ and divide. Then, we obtain formulas
which always hold.

We now write down the formula defining a versal deformation of
\KLB manifolds on the given stratum:
\begin{eqnarray*}
y_{1}y_{2}-K_1K_2L_0L_3y_{0}^{2}&=&0\;,\\
\al2y_{1}+\al1y_{2}-Qy_{0}&=&0\;,
\end{eqnarray*}
where
\begin{eqnarray*}
Q&=&x_0x_3-x_1x_2+(a_1-\la a_0)(b_1-\mu b_0)\al1\al2K_1K_2\;,\\
K_1&=&c_{10}x_0+c_{11}x_1+c_{12}x_2+c_{13}x_3\;,\\
K_2&=&c_{20}x_0+c_{21}x_1+c_{22}x_2+c_{23}x_3\;,\\
L_0&=&\la\mu x_0-\la x_1-\mu x_2 +x_3\;,\\
L_1&=&\la b_1x_0-\la b_0x_1-b_1x_2+b_0x_3\;,\\
L_2&=&a_1\mu x_0-a_1x_1-a_0\mu x_2+a_0x_3\;,\\
L_3&=&a_1b_1x_0-a_1b_0x_1-a_0b_1x_2+a_0b_0x_3\;.
\end{eqnarray*}
The general fibre is a double solid with branch surface
$Q^2-4\al1\al2K_1K_2L_0L_3=0$, with in general 14 double points.

If the point  $P = (a_0: a_1\semic b_0:b_1)$ lies on one of the 
two lines of $L_0$, the conic bundle has a singular line. In this case, the
line $L_0=L_3=0$ lies on the quadric $Q$ 
for all $\al1\al2$: if say $a_1=\la a_0$, then the line $L_0=L_3=0$ 
is given by $\la x_0-x_2=\la x_1 - x_3=0$. 
The  line is  then
a singular line of the quartic surface $Q^2-4\al1\al2K_1K_2L_0L_3=0$. So, 
presumably also in this case, our construction goes through, with appropriate
changes, but we have not checked details.

We note that the equations for the double solids only involve the
four products $\al1a_0$, $\al1a_1$, $\al2b_0$ and $\al2b_1$, which
can be considered as coordinates on our versal deformation of the
\KLB manifold. To describe the other fibres, which are not
double solids, and the total family, one needs the choice of the
extra plane, which gives the structure of the product ${\rm Bl}_0
\C^2 \times {\rm Bl}_0 \C^2$ transversal to 
$\Lambda$.

\subsection{Kummer surfaces}  
A similar description  is possible if the
general fibre is a double solid, branched over a Kummer surface.
We shall relate our family to the moduli space of Kummer surfaces.

In this case, there is a finite number of 
\KLB manifolds, which all are  small modifications of a unique
conic bundle. The discriminant of the bundle consists of 
three reducible conics, given as the intersection of the quadric 
with three tangent planes in general position.
The versal deformation has a four-dimensional base space $B$, with   
a $\C^*$-action. The quotient  $B/\C^*$  under the $\C^*$-action  is a
three-dimensional $A_1$-singularity.

A modern presentation of classical results on the 
moduli space of Kummer surfaces with level two structure
 can be found in  \cite[Ch.~3]{hunt}.
The moduli space  in question is nowadays known as Igusa quartic
$\ci_4$. It is a compactification of the Siegel modular threefold
for $\Gamma_2(2)$, and a resolution is the  non-singular Satake 
compactification.  The Igusa quartic is a rational variety,
which is  dual to the 10-nodal Segre cubic $\cs_3$. The duality map
$\cs_3 \dashrightarrow  \ci_4$ blows up the 10 $A_1$ singularities,
but blows down 15 quadric surfaces to singular lines on the quartic.

The Igusa quartic not only parametrises Kummer surfaces, but
also gives a direct construction of the parametrised objects: 
the tangent plane at a (general) point
intersects $\ci_4$ in a 16-nodal surface (15 as intersection
with the singular lines, plus one extra from the point of
tangency), a fact already mentioned by Hudson \cite[\S\ 80]{hud}.
For special sections, corresponding to boundary points of the
Satake compactification, one obtains  Pl\"ucker
surfaces, quartics with a double line and 8 isolated
double points in general (for a description
of these surfaces, see \cite[Art.\ 83]{Jessop}).

Dually, projecting the Segre cubic $\cs_3$ from a general point
onto a linear space gives a birationally equivalent double solid,
branched along a Kummer surface. The $16$ nodes arise from the 10 nodes 
of $\cs_3$ and  from six lines of the cubic lying in the tangent cone
at the projection point. More details can be found in Baker's book
\cite{baker}. The construction degenerates if the projection point
becomes one of the $10$ nodes. A neighbourhood of each of them
is isomorphic to the quotient $B/\C$  alluded to above.
Our family yields deformations of the \KLB manifold, together
with two divisors, specified by giving a point of the quadric $Q$.
We obtain this quadric by blowing up the node of  $B/\C$. We therefore
find a neighbourhood of the quadric $Q$ on a 
common resolution of $\ci_4$ and $\cs_3$ (which is provided by the
Satake compactification).

We proceed to give explicit equations, similar to Equation \eqref{irrational}
above, which will allow us to describe our family with $\C^*$-action.
We follow Baker \cite{baker}. 
We start by parametrising the Segre cubic $\cs_3$. 
The linear system of quadrics through
the five points in $\P^3$ (for which we take
the vertices of the coordinate tetrahedron
and the point $(1:1:1:1)$) defines a birational map to
$\cs_3$. It blows up
the five points, but the $10$ lines connecting the points are blown 
down to the singular points. 

We take as basis of the linear system
\[
\begin{split}
x&=z_1(z_0-z_2)\;, \\
y&=z_2(z_0-z_3)\;, \\
z&=z_3(z_0-z_1)\;, 
\end{split}
\qquad
\begin{split}
x'&=z_2(z_1-z_0)\;, \\
y'&=z_3(z_2-z_0)\;, \\
z'&=z_1(z_3-z_0)\;, 
\end{split}
\]
with coordinates $(z_0 : z_1 :z_2 :z_3)$ on $\P^3$.
In this way, we embed $\cs_3$ in a hyperplane in $\P^5$. The equations 
of the image are
\begin{multline*}
\hfil x+y+z+x'+y'+z'=0\;, \\
xyz+x'y'z'=0\;. \hfil
\end{multline*}
To determine the dual variety $\ci_4$, 
we compute the Jacobian matrix of these
two polynomials. We set $\xi=yz$, $\xi'=y'z'$, and so on. Using the
parametrisation of $\cs_3$, one checks that the following equation
in irrational form holds:
\[
\sqrt{(\eta-\zeta')(\eta'-\zeta)}+\sqrt{(\zeta-\xi')(\zeta'-\xi)}
 +\sqrt{(\xi-\eta')(\xi'-\eta)}=0\;.
\]
Writing $a=\eta-\zeta'$, and so on, we find the Igusa quartic $\ci_4$
embedded in the hyperplane $a+b+c+a'+b'+c'=0$ in
$\P^5$,  with equation in irrational form:
\[
\sqrt{aa'}+\sqrt{bb'}+\sqrt{cc'}=0\;.
\]
We remark that equations with even more symmetry are known, but
they are not suitable for our purpose. 
In the coordinates $(a,b,c,a',b',c')$, the parametrisation
of $\ci_4$ is given by
\[
\begin{split}
a&=(z_3-z_1)z_2(z_0-z_1)(z_0-z_3)\;, \\
a'&=(z_1-z_3)z_1z_3(z_0-z_2)\;,
\end{split}
\]
and the other variables by cyclic permutation (on $(123)$).

The tangent space of $\ci_4$ at the corresponding point
can be computed from the given parametrisation of
the Segre cubic by a suitable coordinate transformation, or directly
from the defining equations. The result is
\begin{multline*}
z_1z_3a+z_2z_1b+z_2z_3c+z_2(z_1+z_3-z_0)a'+z_3(z_2+z_1-z_0)b'
  +z_1(z_3+z_2-z_0)c' =0 \;,\\
a+b+c+a'+b'+c'=0\;.\hfil
\end{multline*}

We now look at a neighbourhood of $1$ of the $10$ quadrics on a 
common resolution of the Segre cubic and the Igusa quartic.
Let $\wt{\cs} \to \cs_3$ the resolution obtained by blowing up the
singular points, so the exceptional divisors are quadric surfaces.
We also get $\wt\cs$ by first blowing up the five  points in 
$\P^3$ (which already gives a small resolution) and then 
blowing up the strict transforms of the $10$ lines.

Specifically, we choose the quadric obtained
by blowing up the strict transform of the line $(z_0:z_1:0:0)$
in $\P^3$.
We first blow up the points $(1:0:0:0)$ and  $(0:1:0:0)$.
In $\P^3\times \P^2$, the blow-up is given by 
\[
\Rank \begin{pmatrix}z_{0}z_1 & z_2 & z_3 \\
                a_{01} & a_2 & a_3 \end{pmatrix}
  \leq 1 \;.
\]
We take the chart $a_{01}=1$. Next, we blow up the strict transform
$a_2=a_3=0$ of the line. We describe it with homogeneous
coordinates, that is, with a $\C^*$-action.
We explain this for  the blow up of the origin in $\C^2$. Consider coordinates
$(u,v,t)$ with action $\la \cdot (u,v,t) = (\la u,\la v,\la^{-1}
t)$. The quotient $\C^3//\C^*$ is just $\C^2$: the invariant
functions $x=ut$ and $y=vt$ are the coordinates. The quotient only
parametrises closed orbits. We take out the $t$-axis,
consisting of the
origin, which lies in the closure of all non-closed orbits,
and a  non-closed orbit:
consider the set $B\colon u=v=0$. 
The quotient $(\C^3\setminus
B)/\C^*$ is the blow-up of $\C^2$.

At this point, we can introduce our four-dimensional base space,
as the other ruling of the quadric can be treated in the same way.
This means also replacing $z_0$ and $z_1$.
We therefore set
\[
\begin{split}
z_0&=\al1 a_0\;, \\
z_1&=\al1 a_1\;,
\end{split}
\qquad
\begin{split}
a_2&=\al2 b_0\;, \\
a_3&=\al2 b_1\;.
\end{split}
\]
The coordinates $(a_0:a_1\semic b_0:b_1\semic \al1,\al2)$
describe the product ${\rm Bl}_0
\C^2 \times {\rm Bl}_0 \C^2$ of blow-ups of the origin in $\C^2$.
We substitute the above values
in the first equation for the tangent space, add a multiple of 
the second equation and divide out
common factors to obtain
\begin{multline*}
 \quad
\al1\al2\bigl(a_1b_0a+a_1b_1b+
(a_1b_0+a_1b_1-\al1\al2a_0a_1b_0b_1)c\\ {}+
(a_0b_0+a_1b_1-\al1\al2a_0a_1b_0b_1)a'
+(a_0b_1+a_1b_0-\al1\al2a_0a_1b_0b_1)b'\bigr)
  +c' =0\;,\quad 
\end{multline*}
from which we conclude that $c'=-\al1\al2c''$ for some expression $c''$
not involving $c'$.
A rational form of the quartic equation for the hyperplane section 
is
\[
(-\al1\al2cc''+aa'-bb')^2+4\al1\al2aa'cc''=0\;.
\]
We can therefore write a formula for the
deformation of the \KLB manifold as
\begin{equation}\label{kummer}
\begin{split}
y_{1}y_{2}+aa'cc''y_{0}^{2}&=0\;,\\
\al2y_{1}+\al1y_{2}-(-\al1\al2cc''+aa'-bb')y_{0}&=0\;,\\
a+b+c+a'+b'-\al1\al2c''&=0\;.
\end{split}
\end{equation}
The variable $c$ can be eliminated, if its coefficient
does not vanish, so certainly in a neighbourhood of the
quadric (for $\al1\al2$ small).
Equations \eqref{kummer} describe therefore a family 
over a neighbourhood of the
exceptional locus of ${\rm Bl}_0 \C^2 \times {\rm Bl}_0 \C^2 $.

From the other nine quadrics on $\ci_4$, only four are visible
in our chart: the strict transforms of the lines 
through $(1:1:1:1)$ and the vertices of the coordinate tetrahedron
in $\P^3$. For the line $z_0-z_3=z_1-z_3=0$, we find the locus
$1-\al1\al2a_1b_1=1-\al1\al2a_0b_1=0$ (so $a_0=a_1$) in our
base space. Using the explicit expression for the last equation of
\eqref{kummer}
\begin{multline*}
\quad
(1-\al1\al2a_1b_0)a+(1-\al1\al2a_1b_1)b+
(1-\al1\al2a_1(b_0+b_1)+\al1^2\al2^2a_0a_1b_0b_1)c\\{}+
(1-\al1\al2a_0b_0)(1-\al1\al2a_1b_1)a'
+(1-\al1\al2a_0b_1)(1-\al1\al2a_1b_0)b' =0 \;,\quad
\end{multline*}
we  obtain that it reduces to $a=0$.
 The Kummer surface degenerates to the
double quadric $(\al1\al2cc''+bb')^2=0$, but the fibre in the family
of Equations \eqref{kummer}
is reducible: $y_1y_2=0$ and 
$\al2y_{1}+\al1y_{2}+(\al1\al2cc''+bb')y_{0}=0$. These are two sections
of the $\P^2$-bundle over $\P^3$, tangent along a quadric.
So, only above the originally chosen quadric, at $\al1=\al2=0$, we
have \KLB manifolds.


\end{document}